\documentclass[a4paper]{article}
\usepackage{graphicx}        % standard LaTeX graphics tool
\usepackage{multicol} % used for the two-column index
\usepackage[bottom]{footmisc}% places footnotes at page bottom
\usepackage{color,amssymb,amsmath,epsfig,ifthen,graphicx,twoopt,tabularx,xargs,a4wide,answers}
\usepackage[utf8]{inputenc}

\usepackage{amsfonts}
\usepackage[bbgreekl]{mathbbol}
\DeclareSymbolFontAlphabet{\mathbbm}{bbold}
\DeclareSymbolFontAlphabet{\mathbb}{AMSb}%

\usepackage[numbers]{natbib}
\usepackage{amsthm}
\usepackage[boxruled,algo2e]{algorithm2e}

\usepackage{enumitem,verbatim}

\usepackage[textwidth=2cm,textsize=footnotesize]{todonotes}    %  [disable] shuts it off

\usepackage{xr-hyper}
\usepackage{hyperref}
\hypersetup{colorlinks=true,citecolor=blue, urlcolor=blue, linkcolor=black, filecolor=black}
\usepackage{cleveref}
\newtheorem{theorem}{Theorem}[section]
\crefname{theorem}{theorem}{theorems}
\Crefname{Theorem}{Theorem}{Definitions}

\newtheorem{definition}{Definition}[section]
\crefname{definition}{definition}{definitions}
\Crefname{Definition}{Definition}{Definitions}

\crefname{corollary}{corollary}{corollaries}
\Crefname{Corollary}{Corollary}{Corollaries}

\newtheorem{proposition}[theorem]{Proposition}
\crefname{proposition}{proposition}{propositions}
\Crefname{Proposition}{Proposition}{Propositions}

\newtheorem{lemma}[theorem]{Lemma}
\crefname{lemma}{lemma}{lemmas}
\Crefname{Lemma}{Lemma}{Lemmas}

\crefname{example}{example}{examples}
\Crefname{example}{Example}{Examples}

\newtheorem{remark}{Remark}[section]
\crefname{remark}{remark}{remarks}
\Crefname{remark}{Remark}{Remarks}

\crefname{assumption}{assumption}{assumptions}
\Crefname{assumption}{Assumption}{Assumptions}

\crefname{enumi}{point}{point}
\Crefname{enumi}{Point}{Point}

\theoremstyle{remark}
\def\lcag{\mathbbm{G}}%\mathrm{T}}
\def\adjoint{\mathsf{H}}

\def\povm{p.o.v.m.}

\def\fiarma{FIARMA}
\newcommand{\norm}[1]{{
    \left\| #1 \right\|}} % norm
\newcommand{\osego}[1]{{\left( #1 \right)}} % open segment ] [
\newcommand{\oseg}[1]{{\left( #1 \right]}} % semi-open segment ] ]
\newcommandx\cspanarg[2][1=]{\ensuremath{\overline{\mathrm{Span}}^{#1}\left(#2\right)}}

\newcommand{\varsqrt}[1]{\left(#1\right)^{1/2}}
\newcommand{\spec}{{\rm \spec}}

\newcommand{\range}{\mathrm{Im}}
\newcommand{\crange}{\overline{\mathrm{Im}}}

\newcommand{\tnorm}[1]{{\left\vert\kern-0.25ex\left\vert\kern-0.25ex\left\vert #1 
    \right\vert\kern-0.25ex\right\vert\kern-0.25ex\right\vert}} % triple norm

\newcommand{\fracintoptransfer}[1]{\mathrm{FI}_{#1}}
\newcommandx{\fiarmapred}[1][1={\mapol,\arpol,D}]{\mathrm{\Phi}^{\dagger}_{#1}}
\newcommandx{\fiarmapredcoef}[2][1={\mapol,\arpol,D}]{\aop^{\dagger}_{#2}\lr{#1}}
\newcommand{\filterprocess}[1]{F_{#1}}

\newcommand{\processtransfer}[1]{\cS_{#1}}

\newcommand{\essinfarg}[1]{\mathop{#1\text{-}\mathrm{essinf}}}
\newcommand{\essuparg}[1]{\mathop{#1\text{-}\mathrm{essup}}}

\newcommand{\bigo}[1]{O \left( #1 \right)}

\def\unitdisk{\mathbb{D}}
\def\unitcircle{\mathbb{U}}

\def\rset{\mathbb{R}}
\def\cset{\mathbb{C}}
\def\zset{\mathbb{Z}}
\def\nset{\mathbb{N}}

\usepackage{amsfonts}
\usepackage[bbgreekl]{mathbbol}
\DeclareSymbolFontAlphabet{\mathbbm}{bbold}
\DeclareSymbolFontAlphabet{\mathbb}{AMSb}%
\def\mapol{\bbtheta}
\def\arpol{\bbphi}
\def\fiarmacol{\aleph}
\def\apol{\mathbbm{p}}

  {\begin{enumerate}}%
  {\end{enumerate}}

\def\rmi{\mathrm{i}}
\def\rme{\mathrm{e}}
\def\rmd{\mathrm{d}}

\newcommandx{\aslim}[1]{\ensuremath{\stackrel{#1\text{a.s.}}{\longrightarrow}}}  % david removed dash - before \text
\newcommand{\1}{\mathbbm{1}}

\def\bx{\mathbf{x}}
\def\bX{\mathbf{X}}
\def\by{\mathbf{y}}

\def\eqsp{\;}

\newcommand{\aop}{\mathrm{P}}
\newcommand{\bop}{\mathrm{Q}}
\newcommand{\cop}{\mathrm{C}}
\newcommand{\fop}{\mathrm{F}}
\def\calA{\mathcal{A}}

\def\cF{\mathcal{F}}
\def\cE{\mathcal{E}}

\def\cH{\mathcal{H}}
\def\cG{\mathcal{G}}
\def\cI{\mathcal{I}}

\def\Vset{\mathsf{V}}

\def\borel{\mathcal{B}}
\newcommand{\pscal}[2]{\left\langle #1, #2 \right\rangle}

\def\b0{{\bf 0}}

\def\bx{{\bf x}}
\def\by{{\bf y}}

\def\bX{{\bf X}}

\def\cA{\mathcal{A}}

\def\cC{\mathcal{C}}
\def\cD{\mathcal{D}}
\def\cE{\mathcal{E}}
\def\cF{\mathcal{F}}
\def\cG{\mathcal{G}}
\def\cH{\mathcal{H}}
\def\cI{\mathcal{I}}
\def\cJ{\mathcal{J}}

\def\cL{\mathcal{L}}
\def\cM{\mathcal{M}}
\def\cN{\mathcal{N}}

\def\cQ{\mathcal{Q}}
\def\cR{\mathcal{R}}
\def\cS{\mathcal{S}}

\def\cV{\mathcal{V}}

\def\tr{\mathrm{Tr}}

\newcommand\lspan[1]{\mathrm{Span}\left(#1\right)}
\newcommand\cspan[1]{\overline{\mathrm{Span}}\left(#1\right)}
\newcommand\algo[1]%
{
    \begin{center} %
    \begin{tabular} {||p{10 cm}l ||}%
               #1 &  \\
    \hline
    \end{tabular}
    \end{center}
}
\newcommand{\Vsigma}{\ensuremath{\mathcal{V}}}

\newcommand{\chunk}[4][]%
{\ifthenelse{\equal{#1}{}}{\ensuremath{{#2}_{#3:#4}}}{\ensuremath{#2^#1}_{#3:#4}}
}

\def\esp{\mathbb{E}}
\newcommandx\prob[2][1=,2=]{\ensuremath{{\mathbb P}_{#1}^{#2}}}
\newcommand{\PP}[1][]{\ifthenelse{\equal{#1}{}}{\ensuremath{\mathbb{P}}}{\ensuremath{\mathbb{P}\left( #1 \right)}}}
\newcommandx{\PParg}[2][1=]{\PP_{#1}\left(#2\right)}
\newcommand{\PE}[1][]{\ifthenelse{\equal{#1}{}}{\esp}{\ensuremath{{\mathbb E}\left[ #1 \right]}}}
\newcommandx{\PEarg}[2][1=]{\PE_{#1}\left[#2\right]}
\newcommand{\PVar}{\ensuremath{\operatorname{Var}}}
\newcommandx\var[2][1=]{\ensuremath{\PVar_{#1}\left( #2\right)}}
\newcommandx\cvar[3][1=]{\ensuremath{\PVar_{#1}\left( \left. #2 \right| #3 \right)}}

\newcommand{\PCov}{\ensuremath{\operatorname{Cov}}}
\newcommand\Cov[2]{\PCov\left(#1,#2\right)}
\newcommandx\cov[3][1=]{\ensuremath{\mathrm{Cov}_{#1}\left( #2,#3 \right)}}
\newcommandx\ccov[3][1=]{\ensuremath{\mathrm{Cov}_{#1}\left( \left. #2 \right| #3 \right)}}

\newcommand{\CPP}[3][]
{\ifthenelse{\equal{#1}{}}{\PP\left[\left. #2 \, \right| #3 \right]}{\mathbb{P}_{#1}\left(\left. #2 \, \right | #3 \right)}}
\newcommand{\CPE}[3][]
{\ifthenelse{\equal{#1}{}}{\PE\left[ \left. #2 \right| #3
    \right]}{\mathbb{E}_{#1} \left[ \left. #2 \right| #3 \right]}}
\newcommandx\cprob[4][1=,2=]{\ensuremath{\PP_{#1}^{#2}\left[ \left. #3 \right|
      #4 \right]}}
\newcommandx\HMCP[2][1=]{
\ifthenelse{\equal{#1}{}}{\PP_{#2}}{\PP_{#1,#2}}
}
\newcommandx\HMCE[2][1=]{
\ifthenelse{\equal{#1}{}}{\PE_{#2}}{\PE_{#1,#2}}
}
\newcommandx\HMCEarg[3][1=]{
\ifthenelse{\equal{#1}{}}{\PE_{#2}\left[#3\right]}{\PE_{#1,#2}\left[#3\right]}
}
\def\loikhi2{\mathbf{\chi^2}}

\newcommandx{\proj}[2]{\ensuremath{\operatorname{proj}\left( \left. #1\right|#2\right)}}

\newcommand{\U}{\ensuremath{U}}

\newcommand{\URoot}{\ensuremath{R}}
\newcommand{\UCov}[1][]%
{%
\ifthenelse{\equal{#1}{}}{\URoot \URoot^t}{\URoot_{#1} \URoot^t_{#1}}%
}

\newcommand{\VRoot}{\ensuremath{S}}
\newcommand{\VCov}[1][]%
{%
\ifthenelse{\equal{#1}{}}{\VRoot \VRoot^t}{\VRoot_{#1} \VRoot^t_{#1}}%
}

\newcommand{\LDX}[2]{\ensuremath{L}}

\newcommand{\postdx}[3][]%
{%
\ifthenelse{\equal{#1}{}}{\ensuremath{\psi_{#2|#3}}}{\ensuremath{\psi_{#1,#2|#3}}}%
}

\newcommand{\epostdx}[3][]%
{%
\ifthenelse{\equal{#1}{}}{\ensuremath{\hat{\psi}_{#2|#3}}}{\ensuremath{\hat{\psi}_{#1,#2|#3}}}%
}

\newcommandx{\predx}[3][1=\bX]{#1_{#2|#3}}   % prediction of boldface X as default
\newcommand{\predpx}[3][]%
{%
\ifthenelse{\equal{#1}{}}{\ensuremath{\varphi_{#2|#3}}}{\ensuremath{\varphi_{#1,#2|#3}}}%
}
\newcommandx\cesp[4][1=,2=]{\ensuremath{{\mathbb E}_{#1}^{#2}\left[ \left. #3 \right| #4 \right]}}

\newcommand{\filt}[2][]%
{%
\ifthenelse{\equal{#1}{}}{\ensuremath{\phi_{#2}}}{\ensuremath{\phi_{#1,#2}}}%
}
\newcommand{\pred}[3][]%
{%
\ifthenelse{\equal{#1}{}}{\ensuremath{\phi_{#2|#3}}}{\ensuremath{\phi_{#1,#2|#3}}}%
}
\newcommand{\post}[3][]%
{%
\ifthenelse{\equal{#1}{}}{\ensuremath{\phi_{#2|#3}}}{\ensuremath{\phi_{#1,#2|#3}}}%
}
\newcommand{\logl}[2][]%
{%
\ifthenelse{\equal{#1}{}}{\ensuremath{\ell_{#2}}}{\ensuremath{\ell_{#1,#2}}}%
}
\newcommand{\lhood}[2][]%
{%
\ifthenelse{\equal{#1}{}}{\ensuremath{\mathrm{L}_{#2}}}{\ensuremath{\mathrm{L}_{#1,#2}}}%
}
\newcommand{\cc}[2][]%
{%
\ifthenelse{\equal{#1}{}}{\ensuremath{c_{#2}}}{\ensuremath{c_{#1,#2}}}%
}
\newcommand{\forvar}[2][]%
{%
\ifthenelse{\equal{#1}{}}{\ensuremath{\alpha_{#2}}}{\ensuremath{\alpha_{#1,#2}}}%
}
\newcommand{\nforvar}[2][]%
{%
\ifthenelse{\equal{#1}{}}{\ensuremath{\bar{\alpha}_{#2}}}{\ensuremath{\bar{\alpha}_{#1,#2}}}%
}

\newcommand{\BK}[2][]%
{%
\ifthenelse{\equal{#1}{}}{\ensuremath{\mathrm{\mathrm{B}}_{#2}}}{\ensuremath{\mathrm{B}_{#1,#2}}}%
}
\newcommand{\filtfunc}[2][]%
{%
\ifthenelse{\equal{#1}{}}{\ensuremath{\tau_{#2}}}{\ensuremath{\tau_{#1,#2}}}%
}
\newcommand{\filtmean}[2][]
{\ifthenelse{\equal{#1}{}}{{\ensuremath{\hat{X}_{#2|#2}}}}{\ensuremath{\hat{X}_{#1,#2|#2}}}
}
\newcommand{\filtcov}[2][]
{\ifthenelse{\equal{#1}{}}{\ensuremath{\Sigma_{#2|#2}}}{\ensuremath{\Sigma_{#1,#2|#2}}}}
\newcommand{\postmean}[3][]
{\ifthenelse{\equal{#1}{}}{\ensuremath{\hat{X}_{#2|#3}}}{\ensuremath{\hat{X}_{#1,#2|#3}}}
}
\newcommand{\postcov}[3][]
{\ifthenelse{\equal{#1}{}}{\ensuremath{\Sigma_{#2|#3}}}{\ensuremath{\Sigma_{#1,#2|#3}}}}

\newcommandx{\QEM}[4][1=,4=]{\ensuremath{\mathcal{Q}_{#1}(#4;#2 \, ; #3)}}

\def\tore{\mathbb{T}}
\def\btore{\mathcal{B}(\tore)}

\newcommandx\sequence[3][2=t,3=\zset]{\ensuremath{\left(#1_{#2}\right)_{#2 \in #3 }}}
\newcommandx\dsequence[4][3=t,4=\zset]{\ensuremath{\left( (#1_{#3}, #2_{#3})\right)_{#3 \in #4}}}
\newcommandx{\sequencen}[2][2=n\in\nset]{\ensuremath{\left(#1\right)_{#2}}}

\def\bpm{\left[\begin{matrix}}
\def\epm{\end{matrix}\right]}
\def\bma{\begin{matrix}}
\def\ema{\end{matrix}}

\newcommand{\be}{\begin{equation}}     % begin eqn
\newcommand{\ee}{\end{equation}}        % end eqn

\newcommand{\eg}{\textit{e.g.}}

\newcommandx\lnorm[3][1=]{\left\lVert #2 \right\rVert^{#1}_{#3}}

\newcommand{\Id}{\mathrm{Id}}

\newcommandx\supnorm[2][1=]{| #2 |^{#1}_\infty}

\newcommandx\ball[3][1=]{\mathrm{B}_{#1} (#2,#3)}

\newcommandx{\prohosym}[1][1=]{{\boldsymbol\rho}_{#1}}
\newcommandx{\proho}[3][1=]{\prohosym{#1}\left(#2,#3\right)}

\newcommandx{\pp}[1][1=\mu]{\ensuremath{#1\-\mathrm{a.e.}}}

\renewcommand{\-}{\mbox{-}}

\newcommandx{\as}[1][1=\PP]{\ensuremath{#1\-\mathrm{a.s.}}}

\newcommandx{\oscnorm}[3][1=,3=]{\operatorname{osc}^{#1}_{#3}\left(#2\right)}

\newcommandx{\tvdist}[3][1=]{\ensuremath{d^{#1}_{\mathrm{TV}}}(#2,#3)}

\newcommandx{\VnormFunc}[3][1=]{\ensuremath{\left|#2\right|_{\mathrm{#3}}^{#1}}}

\newcommand{\continuousfunctionset}[1]{\mathrm{C}_b(#1)}

\newcommand{\lipschitzfunctionset}[1]{\mathrm{Lip}(#1)}
\newcommand{\boundedlipschitzfunctionset}[1]{\mathrm{Lip}_b(#1)}

\newcommand{\simplemeasfunctionsetarg}[1]{\mathbb{F}_s\left(#1\right)}

\newcommandx\functionsetarg[2][1=]{
\ifthenelse{\equal{#1}{c}}{\continuousfunctionset{\mathsf{#2}}}%fonctions continues
{\ifthenelse{\equal{#1}{bc}}{\mathrm{C}_b(#2)}%fonctions continues born\'{e}es
{\ifthenelse{\equal{#1}{u}}{\mathrm{U}(#2)}%fonctions uniform\'{e}ment continues
{\ifthenelse{\equal{#1}{bu}}{\mathrm{U}_b(#2)}%fonctions uniform\'{e}ment continues born\'{e}es
{\ifthenelse{\equal{#1}{l}}{\lipschitzfunctionset{#2}}%fonctions lipschitz
{\ifthenelse{\equal{#1}{bl}}{\boundedlipschitzfunctionset{#2}}%fonctions lipschitz born\'{e}es
{\mathbb{F}_{#1}(#2)}%le reste
}}}}}}

\newcommandx\functionsetspec[1][1=]{
\ifthenelse{\equal{#1}{c}}{\mathrm{C}}%fonctions continues
{\ifthenelse{\equal{#1}{bc}}{\mathrm{C}_b}%fonctions continues born\'{e}es
{\ifthenelse{\equal{#1}{u}}{\mathrm{U}}%fonctions uniform\'{e}ment continues
{\ifthenelse{\equal{#1}{bu}}{\mathrm{U}_b}%fonctions uniform\'{e}ment continues born\'{e}es
{\ifthenelse{\equal{#1}{l}}{\mathrm{Lip}}%fonctions lipschitz
{\ifthenelse{\equal{#1}{bl}}{\mathrm{Lip}_b}%fonctions lipschitz born\'{e}es
{\mathbb{F}_{#1}}%le reste
}}}}}}

\newcommandx{\taboo}[3][1=,3=]{\left(\leftidx{_#1}{#2}{}\right){^{#3}}}

\newcommandx\vectornorm[2][1=]{\left| #2 \right|^{#1}}  %%%% norme de vecteurs pas fonctions

\newcommand{\ensemble}[2]{\left\{#1\,:\eqsp #2\right\}}
\newcommand{\set}[2]{\ensemble{#1}{#2}}

\newcommandx{\plim}[1]{\ensuremath{\stackrel{#1\-\text{prob}}{\longrightarrow}}}
\newcommandx{\dlim}[1]{\ensuremath{\stackrel{#1}{\Longrightarrow}}}

\newcommand{\mae}{\ensuremath{\mathrm{a.e.}}}

\newcommandx\measureset[3][1=\mathrm{s},3=]{\mathbb{M}^{#3}_{#1}(#2)}
\newcommandx\measuresetmetric[2][1=1]{\mathbb{M}_{#1}(\mathcal{B}(\mathsf{#2}))}  %%% espace de mesures sur un metric muni de sa tribu borelienne
\newcommandx\measuresetspec[1][1=\mathrm{s}]{\mathbb{M}_{#1}}

\newcommand{\abs}[1]{\left\vert #1 \right\vert}

\newcommandx\canonicalkernel[1][1=P]{\mathbb{K}_{#1}}

\newcommand{\lrav}[1]{\left|#1 \right|}
\newcommand{\lr}[1]{\left(#1 \right)}
\newcommand{\lrb}[1]{\left[#1 \right]}

\newcommand{\lrcb}[1]{\left\{#1 \right\}}
\newcommand{\leb}{\mathrm{Leb}}
\newcommand{\lebtore}{\leb_{\tore}}

\usepackage{mathrsfs}
\usepackage{stmaryrd}

\makeatletter
\DeclareFontEncoding{LS1}{}{}
\DeclareFontSubstitution{LS1}{stix}{m}{n}
\DeclareMathAlphabet{\kernelope}{LS1}{stixscr}{m}{n}
\makeatother

\numberwithin{equation}{section}

\title{Hilbert space-valued fractionally integrated autoregressive  moving
  average processes with long memory operators}

\author{Amaury Durand\,\footnote{LTCI, Telecom Paris, Institut Polytechnique de
    Paris.}\;\,\footnote{EDF R\&D, TREE, E36, Lab Les Renardieres, Ecuelles, 77818 Moret sur Loing, France. } \and François Roueff\,$^*$
   \addtocounter{footnote}{-3}
   \footnote{Math Subject Classification. Primary: 60G22, 60G12; Secondary: 47A56, 46G10} \addtocounter{footnote}{-1}
   \footnote{ Spectral representation of random processes. Functional time series. Hilbert space}  \addtocounter{footnote}{-1}}

\begin{document}
\sloppy
\maketitle

\abstract{Fractionally integrated autoregressive moving average
  (FIARMA) processes have been widely and successfully used to model
  and predict univariate time series exhibiting long range
  dependence. Vector and functional extensions of these processes have
  also been considered more recently. Here we study these processes by
  relying on a spectral domain approach in the case where the
  processes are valued in a separable Hilbert space $\cH_0$. In this
  framework, the usual univariate long memory parameter $d$ is
  replaced by a long memory \emph{operator} $D$ acting on $\cH_0$,
  leading to a class of $\cH_0$-valued FIARMA($D,p,q$) processes, where
  $p$ and $q$ are the degrees of the AR and MA polynomials. When $D$
  is a normal operator, we provide a necessary and sufficient
  condition for the $D$-fractional integration of an $\cH_0$-valued
  ARMA($p,q$) process to be well defined. Then, we derive the best
  predictor for a class of causal FIARMA processes and study how this
  best predictor can be consistently estimated from a finite sample of
  the process. To this end, we provide a general result on quadratic
  functionals of the periodogram, which incidentally yields a result
  of independent interest. Namely, for any ergodic stationary process
  valued in $\cH_0$ with a finite second moment, the empirical
  autocovariance operator converges, in trace-norm,
  to the true autocovariance operator almost surely at each lag.
  
\tableofcontents
\section{Introduction}
Over the past several decades, the study of weakly stationary time series valued in a separable
Hilbert space has been an active field of research. For example, functional ARMA processes were discussed in
\cite{bosq00, spangenberg_strictly_2013, klepsch_prediction_2017}, a
spectral theory was detailed in \cite{panaretos13,
  panaretos13Cramer, these-tavakoli-2015} and several estimation
methods were studied in \cite{hormann2010weaklydependent,
  horvath2012inference, hoermann15, horvath-clt,KOKOSZKA-review,
  BERKES2016150, kokozska2019-inv, VANDELFT2019}. However, these
references mainly focus on short memory processes. The study of
long memory processes valued in a separable Hilbert space is a more
recent topic as seen in
\cite{rackauskas2011,characiejus2013central,characiejus2014operator,
  duker_limit_2018, Li-long-memory-2020}. More specifically, in \cite[Section~4]{Li-long-memory-2020},  the fractionally
integrated autoregressive moving average (often  abbreviated as ARFIMA, but we prefer to use \fiarma\ for reasons that will be made explicit in \Cref{rem-def-fiarma-fiarma})
processes are generalized to the
case of \emph{curve}, or \emph{functional}, time series. In short, the
authors consider the functional case
in which the Hilbert space is an $L^2$ space of real valued functions
defined on a compact subset of $\rset$, say $[0,1]$, and they introduce
the time series $(X_t)_{t\in\zset}$ valued in this Hilbert
space defined by
\begin{equation}\label{eq:farima-robinson-fiarma}
X_t(v) = Y_t+\sum_{k=1}^{\infty} \frac{\prod_{j=0}^{k-1}(d+j)}{k!} Y_{t-k}(v)\;, \quad t \in \zset, v \in [0,1]\;,
\end{equation}
where $-1/2 < d < 1/2$ and $Y_t$ is a functional ARMA process. As pointed out in
\cite[Remark 9]{Li-long-memory-2020}, taking the same $d$ for all
$v\in[0,1]$ in~(\ref{eq:farima-robinson-fiarma})
is highly restrictive compared to other long memory processes recently
introduced. For instance in \cite{characiejus2013central,characiejus2014operator}, 
they consider long memory processes of the form
$$
X_t(v)=\sum_{k=0}^\infty(1+k)^{-\mathrm{n}(v)}\;\epsilon_{t-k}(v)\;,\quad t\in\zset\;,\;v\in\Vset\;,
$$
where $(\Vset,\Vsigma,\xi)$ is a $\sigma$-finite measure space, and
$(\epsilon_t)_{t\in\zset}$ is a white noise valued in
$L^2(\Vset,\Vsigma,\xi)$. Since the ratio
in~(\ref{eq:farima-robinson-fiarma}) is asymptotically equivalent to
$(1+k)^{-1+d}$ as $k\to\infty$, this new process is, in fact, close to
the previous one in the case where $n(v)=d-1$ for all $v\in V$.
A formulation that is not restricted to an $L^2$ space
was proposed in \cite{duker_limit_2018} where the author considers long memory processes of the form
\begin{equation}
  \label{eq:duker-fiarma}
X_t=\sum_{k=0}^\infty(1+k)^{-N}\;\epsilon_{t-k}\;,\quad t\in\zset\;.  
\end{equation}
Here, $(\epsilon_t)_{t\in\zset}$ is a white noise valued in a
separable Hilbert space $\cH_0$ and $N$ is a bounded normal operator
on $\cH_0$.  This suggests defining \fiarma\ processes
in~(\ref{eq:farima-robinson-fiarma}) with $d$ replaced by a function
$\mathrm{d}(v)$, or in the case where it is valued in an arbitrary
separable Hilbert space $\cH_0$, by a bounded normal operator $D$
acting on this space.

Therefore, in this paper, we fill this gap by providing a definition
of \fiarma\ processes valued in a separable Hilbert space $\cH_0$ with
a long memory operator $D$, taken as a bounded linear operator on
$\cH_0$. If $D$ is normal, then we can rely on its singular value
decomposition and find necessary and sufficient conditions to ensure
that the $\cH_0$-valued FIARMA process with long memory operator $D$
is well defined.  This allows us to compare \fiarma\ processes with
the processes defined by~(\ref{eq:duker-fiarma}) as in
\cite{duker_limit_2018}.  Our definition relies on linear filtering in
the spectral domain.  It is a well known fact that linear filtering of
real valued time series in the time domain is equivalent to pointwise
multiplication by a transfer function in the frequency domain. This
duality also applies to Hilbert space valued time series using a
proper spectral representation for them. In this context, pointwise
multiplication becomes a pointwise application of an operator-valued
transfer function defined on the set of frequencies. A complete
account is provided in \cite{surveyREFnew}. Here, we rely on the
spectral approach to define a $D$-fractional integration filter acting
on a weakly stationary process $X$ valued in $\cH_0$. We provide
necessary and sufficient conditions for this filter to be well defined
on a $X$, when $X$ is a $\cH_0$-valued ARMA process and $D$ a normal
operator. When the ARMA process is causal, we derive the best
predictor of $X_t$ given its past $(X_s)_{s<t}$. It is thus of
interest to investigate whether this best predictor can be
consistently estimated from a finite sample $X_1,\dots,X_n$. We
provide a positive answer to this question when the long memory
parameter operator $D$ has a positive definite real part, under mild
additional conditions. To this end, we study quadratic functionals
based on the periodogram of $X_1,\dots,X_n$. A result of this study,
which is of independent interest, and appears to be novel based on our
up-to-date-knowledge, is the following:
\begin{theorem}\label{thm:1}
  Let $\cH_0$ be a separable Hilbert space and let $(X_t)_t$ be an $\cH_0$-valued
  ergodic stationary process defined on $(\Omega,\cF,\PP)$ and satisfying
  $\PEarg{\norm{X_0}_{\cH_0}^2}<\infty$. Let us define, for all
  $n\geq1$ and $1\leq k\leq n$,
  \begin{align}
      \label{eq:centering}
                X^c_{n,k}=X_k-\frac1n\sum_{j=1}^nX_j \;.
  \end{align}
  Then, we have, for all $h\in\zset$, 
  \begin{align}
  \label{eq:cv-as-emp-cov}
 \lim_{n\to\infty} \frac1n\sum_{\stackrel{1\leq k,k'\leq
        n}{k-k'=h}}\lr{X^c_{k,n}}\otimes
     \lr{X^c_{k',n}} = \cov{X_h}{X_0} \quad
     \text{in $\cS_1(\cH_0)$}\;,\quad\as\;,
  \end{align}
      where $\cS_1(\cH_0)$ is the space of trace-class operators endowed
      with the trace-norm. 
\end{theorem}

This paper is organized as follows.  We first recall in
\Cref{sec:preliminaries-fiarma} the necessary definitions and facts on
operator theory and linear filtering needed for our purpose. Then, the
construction of \fiarma\ processes is introduced and discussed in
\Cref{sec:construction-fiarma} with a focus on the case where the long
memory operator is normal. In \Cref{sec:estimation},
the prediction of FIARMA processes is studied. To this end, in
\Cref{sec:est-main-assumpt-main}, we provide
general results for parametric contrast estimation in the spectral
domain, based on a finite sample. Then, in
\Cref{sec:examples-est-application}, we show how to apply these results
for FIARMA prediction. 
Finally, proofs are provided in
\Cref{sec:proofs-fiarma}. In particular, \Cref{thm:1} is proven in
\Cref{sec:proof-thm:1}.

\section{Preliminaries and useful notation}\label{sec:preliminaries-fiarma}

\subsection{Operators, measurability and
  integrals}\label{sec:ope-integration-fiarma}
Throughout this paper, we denote by $\cL_b(\cH_0,\cG_0)$ the set
of continuous linear operators defined on the separable (complex)
Hilbert space $\cH_0$ onto the separable (complex) Hilbert space
$\cG_0$. The operator norm on $\cL_b(\cH_0,\cG_0)$ is denoted by
$\norm{\cdot}_{\infty}$. We denote by $\cS_{\infty}(\cH_0,\cG_0)$ its
subset of \emph{compact} operators, and by $\cS_1(\cH_0,\cG_0)$ and
$\cS_2(\cH_0,\cG_0)$, its subsets of \emph{trace-class} and
\emph{Hilbert-Schmidt} operators, respectively, with their respective
norms denoted by $\norm{\cdot}_1$ and $\norm{\cdot}_2$. We follow the
usual convention of omitting $\cG_0$ in the notation of operator
spaces when $\cG_0=\cH_0$.  We use the notation $\aop^\adjoint$ for
the Hermitian adjoint of an operator $\aop \in \cL_b(\cH_0,\cG_0)$.
An operator $\aop\in\cL_b(\cH_0)$ is said to be \emph{normal} if
$\aop\aop^\adjoint=\aop^\adjoint\aop$ and we denote by $\cN(\cH_0)$
the set of normal bounded operators. We further denote by
$\cL_b^+(\cH_0)$, $\cS_1^+(\cH_0)$ and $\cS_2^+(\cH_0)$ the sets of
positive, positive trace-class and positive Hilbert-Schmidt
operators. Here \emph{positive} refers to
\emph{positive-semidefinite}, that is, $\aop$ is positive if
$\pscal{\aop x}{x}_{\cH_0}\geq0$ for all $x$. For a positive
operator $\aop$, the operator $\aop^{1/2}$ is the unique positive
operator satisfying $\left(\aop^{1/2}\right)^2 = \aop$. A
general and detailed presentation of operator theory can be found in
\cite{Weidmann-operators-hilbert}.

We will make extensive use of integrals of functions valued in a
Banach space (see \cite[Chapter~1]{dinculeanu2011vector} for
details). Given a measure space $(\Lambda, \calA,\mu)$, a Banach space
$(E,\norm{\cdot}_E)$ and $p\in[1,\infty]$, we denote by
$\cL^p(\Lambda, \calA, E, \mu)$ the space of functions
$f : \Lambda \to E$ which are Borel-measurable
such that $\int \norm{f}_E^p \, \rmd \mu$ (or
$\essuparg{\mu} \norm{f}_E$ for $p=\infty$) is finite. Its quotient
space for the $\mu$-a.e. equality is denoted by
$L^p(\Lambda, \calA, E, \mu)$. We use the same
notation for $E=\cL_b^+(\cH_0)$, $\cS_1^+(\cH_0)$ or $\cS_2^+(\cH_0)$,
in which case $L^p(\Lambda, \calA, E, \mu)$ is a cone subset of the
corresponding $L^P$ space. 

In the particular case where
$E=\cL_b(\cH_0,\cG_0)$ for two separable Hilbert spaces $\cH_0,\cG_0$,
we also use a weaker notion of measurability. Namely, we say
that a function $\Phi : \Lambda \to \cL_b(\cH_0,\cG_0)$ is
\emph{simply measurable} if for all $x \in \cH_0$,
$\lambda \mapsto \Phi(\lambda)x$ is measurable as a $\cG_0$-valued
function. The set of simple measurable functions from $(\Lambda,
\calA)$ to $\cL_b(\cH_0,\cG_0)$ is denoted by
$\simplemeasfunctionsetarg{\Lambda, \calA, \cH_0,\cG_0}$ where, again,
we ommit $\cG_0$ if $\cH_0=\cG_0$. A mapping $\Phi :
  \Lambda \to E$ with $E=\cS_1(\cH_0,\cG_0)$ or
$E=\cS_2(\cH_0,\cG_0)$ is simply measurable if and only if it is Borel
measurable (see \Cref{lem:meas-schatten} in
\cite{surveyREFnew}). A useful consequence is that, if
$\Phi \in L^1(\Lambda, \cA, \cS_1^+(\cH_0), \mu)$, then the function
$\Phi^{1/2} : \lambda \mapsto \Phi(\lambda)^{1/2}$ is in
$L^2(\Lambda, \cA, \cS_2^+(\cH_0), \mu)$.

\subsection{Linear filtering of Hilbert space-valued time series in the spectral domain}
\label{sec:line-filt-hilb}

This section gathers the spectral theory used for
linear filtering of times series valued in a separable Hilbert
space. We refer the reader to
\cite[\Cref{sec:spectral-analysis}]{surveyREFnew} for details. In the
following, we denote by $\tore$ the set $\rset/2\pi\zset$, which can be
represented by an interval such as $[-\pi,\pi)$. Let
$(\Omega, \cF, \PP)$ be a probability space and $\cH_0$ a separable
Hilbert space. We recall that the expectation of
$X \in L^2(\Omega, \cF, \cH_0, \PP)$ is the unique vector
$\PE[X] \in \cH_0$ satisfying
$$
\pscal{\PE[X]}{x}_{\cH_0} = \PE[\pscal{X}{x}_{\cH_0}], \quad \text{for all } x \in \cH_0\;.
$$
The covariance
operator between $X,Y \in L^2(\Omega, \cF, \cH_0, \PP)$ is the unique
linear operator $\Cov{X}{Y} \in \cL_b(\cH_0)$ satisfying
$$
\pscal{\Cov{X}{Y} y}{x}_{\cH_0} = \Cov{\pscal{X}{x}_{\cH_0}}{\pscal{Y}{y}_{\cH_0}}, \quad \text{for all } x,y \in \cH_0 \; .
$$
A process  $X := (X_t)_{t \in \zset}$ is
   said to be an $\cH_0$-valued, weakly stationary process if
\begin{enumerate}[label=(\roman{enumi})]
\item For all $t \in \zset$, $X_t \in L^2(\Omega, \cF, \cH_0, \PP)$.
\item For all $t \in \zset$, $\PE[X_t] = \PE[X_0]$. We say that $X$ is centered if $\PE[X_0] = 0$.
\item For all $t,h \in \zset$, $\cov{X_{t+h}}{X_{t}} = \cov{X_h}{X_0}$.
\end{enumerate}
We denote by $\cM(\Omega, \cF, \cH_0, \PP)$ the space of all centered
random variables in $L^2(\Omega, \cF, \cH_0, \PP)$. Let $\cH=\cM(\Omega, \cF, \cH_0, \PP)$ and $X = (X_t)_{t \in \zset} \in \cH^\zset$
be a centered, weakly stationary, $\cH_0$-valued time series.
As explained in \cite[\Cref{sec:spectral-analysis}]{surveyREFnew}, a
spectral representation for $X$ amounts to define a random Gramian-orthogonally scattered measure $\hat{X}$ on $(\tore,\btore)$ such that 
\begin{equation}
  \label{eq:sp-rep}
X_t = \int \rme^{\rmi\lambda\,t} \; \hat{X}(\rmd \lambda)\quad\text{for all}\quad t \in \zset\;.  
\end{equation}
The intensity measure $\nu_X : \btore \to \cS_1^+(\cH_0)$ of $\hat{X}$
is called the \emph{spectral operator measure} and is characterized by
the identity
$$
\Cov{X_h}{X_0} = \int \rme^{\rmi h \lambda}\,\nu_X(\rmd\lambda) \;, \quad \text{for all } h\in\zset\;. 
$$
The spectral operator measure is a trace-class Positive
Operator-Valued Measure (\povm) in the sense that it is a mapping
from $\btore$ to $\cS_1^+(\cH_0)$ which is $\sigma$-additive for the
$\norm{\cdot}_1$-norm. Note that, in this case, the mapping
$\norm{\nu_X}_1 : A \mapsto \norm{\nu_X(A)}_1$ is a finite
non-negative measure. Throughout this paper, we use the
Radon-Nikodym property of the trace-class \povm\ $\nu_X$ which is a
consequence of Theorem 1 in \cite[Chapter~III,
Section~3]{diestel1977vector}. Namely, for any $\sigma$-finite
non-negative measure $\mu$ on $(\tore, \btore)$, which dominates
$\norm{\nu_X}_1$, there exists a unique
$g \in L^1(\tore,\btore, \cS_1^+(\cH_0), \mu)$ such that, for all
$A \in \btore$, $\nu_X(A) = \int_A g \, \rmd \mu$. In this case, we
say that $g$ is the \emph{spectral operator density} of $X$ with
respect to $\mu$, and we write $\rmd\nu_X=g\,\rmd\mu$. In the following, when we say that $g$ is the
spectral operator density of $X$ with respect to a $\sigma$-finite
non-negative measure $\mu$, it is implicitly assumed that $\mu$
dominates $\norm{\nu_X}_1$.

Let us now briefly introduce the linear filtering in the spectral
domain. We only state the facts that will be useful in the following
and refer the reader to \cite{surveyREFnew} for further details. The spectral
representation~(\ref{eq:sp-rep}) can be extended to define a
Gramian-isometric mapping from the 
modular \emph{spectral domain} $\widehat{\cH}^X$ to the modular \emph{time domain}
$\cH^X$, also denoted as an integral
with respect to $\hat{X}$, namely,
$$
Y = \int \Phi(\lambda) \hat{X}(\rmd \lambda)\;,\quad Y\in\cH^X\,,\;\Phi\in\widehat{\cH}^X\;.
$$
Here, $\cH^X$ is the smallest closed linear subspace of $\cM(\Omega,
\cF, \cH_0, \PP)$, which contains $\set{X_t}{t\in\zset}$ and is stable through the left multiplication by
any operator of $\cL_b(\cH_0)$. The space $\widehat{\cH}^X$ is its
spectral counterpart, a space of operator-valued functions defined on
$(\tore,\borel(\tore))$ which only depends on $\nu_X$ and is
stable through the same module action, namely, through left multiplication by an
operator of $\cL_b(\cH_0)$. Conversely, given
an operator-valued function $\Phi$ defined on $(\tore,\borel(\tore))$,
we denote by $\processtransfer{\Phi}(\Omega,\cF,\PP)$ the class of all
centered weakly stationary processes $X$ such that
$\Phi\in\widehat{\cH}^X$. Then, the time-shift invariant linear filter with
\emph{transfer operator function} $\Phi$ is the mapping defined on 
$\processtransfer{\Phi}(\Omega,\cF,\PP)$ by mapping a centered weakly
stationary process $X$ (the input) to the  centered weakly
stationary process $Y$ (the output) defined by
$$
Y_t=\int \rme^{\rmi t\lambda}\Phi(\lambda)\;\hat{X}(\rmd\lambda)\;,\qquad t\in\zset\;,
$$
which we also write 
\begin{equation}\label{eq:time-domains-notation-transfer-fiarma}
Y=\filterprocess{\Phi}(X) \quad \text{or} \quad \hat Y(\rmd\lambda)=\Phi(\lambda)\hat X(\rmd\lambda) \; .
\end{equation}
We will use the following result, where we characterize the domain of definition of a
filter $\filterprocess{\Phi}$ in the case where $\Phi$ is valued in
$\cL_b(\cH_0)$. It follows by applying
\cite[\Cref{prop:carac-processtransfer}]{surveyREFnew} with
$\lcag=\zset$ and $\cG_0=\cH_0$.
\begin{proposition}
  \label{prop:carac-processtransfer-discrete}
  Let $\cH_0$ be a separable Hilbert space, $(\Omega,\cF,\PP)$ be a
  probability space, and
  $\Phi\in\simplemeasfunctionsetarg{\tore, \borel(\tore),
    \cH_0}$. Let  $X$ be an $\cH_0$-valued centered weakly stationary
  process admitting $g_X$ as  a spectral operator density 
  with respect to a $\sigma$-finite non-negative measure $\mu$ on
  $(\tore,\btore)$. Then, the mapping $\norm{\Phi g_X \Phi^\adjoint}_1$ is
  measurable from $(\tore,\btore)$ to
  $(\rset,\borel(\rset))$, and we have
  $X\in\processtransfer{\Phi}(\Omega,\cF,\PP)$ if and only if
  $  \int \norm{\Phi g_X \Phi^\adjoint}_1\;\rmd\mu<\infty$.
\end{proposition}

\section{Hilbert space-valued  \fiarma\ processes}\label{sec:construction-fiarma}
In this section, we propose a definition of \fiarma\ processes valued
in a separable Hilbert space thus extending the definition of
\cite[Section~4]{Li-long-memory-2020} to an operator long memory
parameter. This definition is introduced in \Cref{sec:arma-fiarma}
where we also recall known results on the existence of ARMA processes.
We then state the main results, namely 1)
\Cref{thm:cns-fi-operator-functional-fiarma} where necessary and
sufficient conditions are given for a weakly stationary $\cH_0$-valued
process $X$ to belong the domain of definition of the fractional
integration operator filter, 2)
\Cref{cor:cns-arma-fi-operator-functional-fiarma} where we specify
these conditions to
the case where $X$ is an ARMA process, thus ensuring the existence of
\fiarma\ processes, and 3) \Cref{prop:longmemory-fiarma} where we
compare the obtained \fiarma\
processes to the processes introduced in \cite{duker_limit_2018}. The
first two points are found in \Cref{sec:hilbert-valued-fiarma} and the third in
\Cref{sec:other-long-memory-fiarma}.

\subsection{Definition of \fiarma\ processes}\label{sec:arma-fiarma}
Let $\cH_0$ be a separable Hilbert space. In the following, for all $D\in\cL_b(\cH_0)$ and
$z\in\cset\setminus[1,\infty)$, we will use 
$$
(1-z)^D=\exp(D\ln(1-z)) = \sum_{k=0}^\infty \frac1{k!}\;(D\ln(1-z))^k\;,
$$
where $\ln$ denotes the principal complex logarithm, so that $z\mapsto\ln(1-z)$
is holomorphic on $\cset\setminus[1,\infty)$, and so is
$z\mapsto(1-z)^D$, as a $\cL_b(\cH_0)$-valued function (see
\cite[Chapter~1]{gohberg-leiterer09} for an overview on the subject). Let us now introduce the fractional integration operator transfer function. 

\begin{definition}[Fractional integration operator transfer function]
  Let $\cH_0$ be a separable Hilbert space and $D \in \cL_b(\cH_0)$. We
  define the $D$-order fractional integration operator transfer
  function $\fracintoptransfer{D}$ by
$$
\fracintoptransfer{D}(\lambda)=
\begin{cases}
\left(1 - \rme^{-\rmi \lambda} \right)^{-D} & \text{ if
$\lambda\neq0$,}\\
0                                           & \text{ otherwise.}
\end{cases}
$$
\end{definition}
Using the properties of $z\mapsto(1-z)^D$ recalled previously, we see
that $\fracintoptransfer{D}$ is a mapping from $\tore$ to
$\cL_b(\cH_0)$, continuous on $\tore\setminus\{0\}$.  Then, we have
$\fracintoptransfer{D}\in\simplemeasfunctionsetarg{\tore, \btore,
  \cH_0}$ and we can define the filter
$\filterprocess {\fracintoptransfer{D}}$ as
in~(\ref{eq:time-domains-notation-transfer-fiarma}) of which the domain of
definition are the centered weakly stationary $\cH_0$-valued processes
$X\in\processtransfer{\fracintoptransfer{D}}(\Omega,\cF,\PP)$.
Since $\fracintoptransfer{D}$ has a singularity at the null frequency,
the domain $\processtransfer{\fracintoptransfer{D}}(\Omega,\cF,\PP)$ 
is not obvious. For
instance, in the scalar case, it is well known that if $X$ has a
positive and continuous spectral density at the null frequency, then
$\filterprocess{\fracintoptransfer{d}}(X)$ is well defined if and only if $d<1/2$.
We provide a complete description of
$\processtransfer{\fracintoptransfer{D}}(\Omega,\cF,\PP)$ in
\Cref{sec:hilbert-valued-fiarma} when $D$ is a normal operator. 

A \emph{fractionally integrated autoregressive moving average} (FIARMA)
process is simply the output of the filter in the case where $X$ is an
$\cH_0$-valued autoregressive moving average (ARMA) process. Let us
first recall a basic result on the existence of weakly stationary
ARMA processes (see \cite[Corollary~2.2]{spangenberg_strictly_2013}). 
\begin{theorem}
  Let $\cH_0$ be a separable Hilbert space and $p,q$ be two positive
  integers. Let $A_1,\dots,A_p\in\cL_b(\cH_0)$,
  $B_1,\dots,B_q\in\cL_b(\cH_0)$ and $Z=(Z_t)_{t\in\zset}$ be an
  $\cH_0$-valued white noise (i.e. a centered weakly stationary
  $\cH_0$-valued process with constant spectral density
  operator). Suppose that
  \begin{equation}
  \label{eq:Qdef-fiarma}
\arpol(z):=\Id_{\cH_0}-\sum_{k=1}^p A_k z^k\quad\text{is invertible for all $z\in\unitcircle$,  }
\end{equation} 
where  $\unitcircle= \set{z \in \cset}{\abs{z} = 1}$ is the complex
unit circle.
Then,
\begin{equation}
  \label{eq:inverse-pol-filter-fiarma}
  X_t - \sum_{k=1}^pA_kX_{t-k} = Z_t+\sum_{k=1}^qB_kZ_{t-k}\;,\quad t\in\zset\;,
\end{equation}
admits a unique weakly stationary solution. This solution is
called an $\cH_0$-valued ARMA($p,q$) process. 
\end{theorem}
Explicit constructions of the solution in the time domain can be
found in
\cite{bosq00,spangenberg_strictly_2013,klepsch_prediction_2017}, under
various assumption.
Using a spectral approach, with $\arpol$ as in~(\ref{eq:Qdef-fiarma})
and $\mapol(z):=\Id_{\cH_0}+\sum_{k=1}^p B_k z^k$, the solution is more directly given by
$$
\hat X(\rmd
\lambda)=\lrb{\arpol(\rme^{-\rmi\lambda})}^{-1}\mapol(\rme^{-\rmi\lambda})\hat Z(\rmd\lambda)\;,
$$
using the notation introduced
in~\eqref{eq:time-domains-notation-transfer-fiarma}. In the following,
for any integer $d\in\nset$, $\mathcal{P}_d(\cH_0)$ denotes the set of
polynomials $\apol$ of degree $d$ with coefficients in $\cL_b(\cH_0)$,
such that $\apol(0)=\Id_{\cH_0}$ and $\mathcal{P}_d^\ast(\cH_0)$
denotes the subset of all $\apol\in\mathcal{P}_d(\cH_0)$, which are
invertible on $\unitcircle$. In particular,~(\ref{eq:Qdef-fiarma}) is
equivalent to saying that
$\arpol\in\mathcal{P}_d^\ast(\cH_0)$. 
Time domain approaches for defining ARMA processes are easier to
derive when Condition~(\ref{eq:Qdef-fiarma}) is extended on the closed
unit disk, that is,
  \begin{equation}
  \label{eq:Qdef-fiarma-causal}
\arpol(z)=\Id_{\cH_0}-\sum_{k=1}^p A_k z^k\quad\text{is invertible for all $z\in\overline{\unitdisk}$,  }
\end{equation} 
where $\unitdisk := \set{z \in \cset}{\abs{z} < 1}$ and
$\overline{\unitdisk} = \set{z \in \cset}{\abs{z} \leq 1}$ denote the
open and closed complex unit discs of $\cset$, respectively. We do not
need Condition~(\ref{eq:Qdef-fiarma-causal}) for defining FIARMA
processes. However, we will assume that $\arpol$ and $\mapol$ satisfy
such a condition to derive predictors (see \Cref{sec:estimation}), as
in the well known case of univariate FIARMA processes.

We can now define \fiarma\ processes as follows.
  \begin{definition}[Hilbert space-valued \fiarma\ processes]
    \label{def:fiarma-fiarma} Let $\cH_0$ be a separable Hilbert space
    and $p,q$ be two non-negative integers.  Let $D\in\cL_b(\cH_0)$,
    $\mapol\in\mathcal{P}_q(\cH_0)$,
    $\arpol\in\mathcal{P}_p^\ast(\cH_0)$ and $Z$ be an $\cH_0$-valued
    centered white noise. Let $X$ be the ARMA($p,q$) process defined
    by
    $\hat
    X(\rmd\lambda)=[\arpol(\rme^{-\rmi\lambda})]^{-1}\mapol(\rme^{-\rmi\lambda})\hat
    Z(\rmd\lambda)$ and suppose that
    $X\in\processtransfer{\fracintoptransfer{D}}(\Omega,\cF,\PP)$. Then,
    the process defined by
    $Y=\filterprocess {\fracintoptransfer{D}}(X)$, that is, with
    spectral representation given by
    \begin{equation}
      \label{eq:fiarma-def-spectral-fiarma}
      \hat
      Y(\rmd\lambda)=\fracintoptransfer{D}(\lambda)[\arpol(\rme^{-\rmi\lambda})]^{-1}\mapol(\rme^{-\rmi\lambda})\hat
      Z(\rmd\lambda) \;,
    \end{equation}
    is called a \fiarma\ process of order $(p,q)$ with long memory
    operator $D$, abbreviated as \fiarma$(D,p,q)$.
  \end{definition}
  \begin{remark}\label{rem-def-fiarma-fiarma}
    \Cref{def:fiarma-fiarma} extends the usual definition of
    univariate ($\cset$ or $\rset$-valued) ARFIMA($p,d,q$) processes
    to the Hilbert space-valued case. In the general case, we use the
    acronym FIARMA to indicate the order of the operators in the
    definition~(\ref{eq:fiarma-def-spectral-fiarma}), where the
    fractional integration operator appears on the left of the
    autoregressive operator, which then is  on the left of the
    moving average operator. We also respected this order in the list
    of parameters $(D,p,q)$. Following this convention, an
    ARFIMA$(p,D,q)$ process is, in turn, defined as the solution
    of~(\ref{eq:inverse-pol-filter-fiarma}), with $Z$ defined as a
    FIARMA$(0,D,q)$ process. Having this convention in mind is
    important since the ARFIMA$(p,D,q)$ process does not coincide with
    the FIARMA$(D,p,q)$ process, except in highly unique instances such as
    the univariate case where all
    operators commute.
  \end{remark}
  \Cref{def:fiarma-fiarma} extends the definition of ARFIMA curve time
  series proposed in \cite{Li-long-memory-2020} where it is restricted
  to the case where $D$ is a scalar operator,
  that is $D : f \mapsto d \times f$ for a constant
  $-1/2<d<1/2$. In this particular case, it is rather straightforward to show that
  $X\in\processtransfer{\fracintoptransfer{D}}(\Omega,\cF,\PP)$ for
  any ARMA process $X$ by directly making use of
  \Cref{prop:carac-processtransfer-discrete}, However, in
  \Cref{ex:rem-existence-fiarma-particular-cases}\ref{item:li-robinson-case},
  it will be obtained as a special case of
  \Cref{cor:cns-arma-fi-operator-functional-fiarma}.
\subsection{Existence of \fiarma\ processes}
\label{sec:hilbert-valued-fiarma}
In this section, we provide a necessary and sufficient condition for
the existence of \fiarma\ processes as defined in
\Cref{def:fiarma-fiarma} in the case where $D$ is a \emph{normal}
operator. In this case, we can rely on the singular value
decomposition of $D$ (see \cite[Theorem~9.4.6,
Proposition~9.4.7]{conway1994course}). Namely, if $D \in \cN(\cH_0)$,
then there exists a $\sigma$-finite measure space
$(\Vset, \Vsigma, \xi)$, a unitary operator
$U : \cH_0 \to L^2(\Vset, \Vsigma, \xi)$ and
$\mathrm{d} \in L^\infty(\Vset, \Vsigma, \xi)$, such that
\begin{equation}
  \label{eq:SVD-normal-op-fiarma}
U D U^\adjoint = M_{\mathrm{d}}\;,  
\end{equation}
where $M_{\mathrm{d}}$ denotes the pointwise multiplicative operator
on $L^2(\Vset, \Vsigma, \xi)$ associated to $\mathrm{d}$, that is
$M_{\mathrm{d}} : f \mapsto \mathrm{d}\times f$. We say that $D$ has a
\emph{singular value function} $\mathrm{d}$ on
$L^2(\Vset, \Vsigma, \xi)$ with a \emph{decomposition operator}
$U$. Using the decomposition operator, we can rely on the process
$UX=(UX_t)_{t\in\zset}$ valued in $\cG_0:=L^2(\Vset, \Vsigma,
\xi)$. Note that $\cG_0$ is separable because it is isometrically
isomorphic to $\cH_0$ through the unitary operator $U$. It is straightforward to
check that, if $g_X$ is the spectral operator density of $X$ with
respect to a non-negative measure $\mu$ on $(\tore,\btore)$, then the
function $g_{UX} \in L^2(\tore,\btore,\cS_1^+(\cG_0),\mu)$ defined by
$g_{UX}(\lambda) = U g_X(\lambda) U^\adjoint$, for all
$\lambda\in\tore$, is the spectral operator density of $UX$ with
respect to $\mu$. Note that we can always find a function
$h\in L^2(\tore,\btore,\cS_2(\cG_0),\mu)$ such that
$g_{UX}(\lambda) = h(\lambda)[h(\lambda)]^\adjoint$ for $\mu$-$\mae$
$\lambda \in \tore$ (take \eg\ $h = g_{UX}^{1/2}$). Then,
\cite[Theorem~6.11]{Weidmann-operators-hilbert} gives that, for all
$\lambda\in\tore$, the operator $h(\lambda)$ can be written as an
integral operator with a kernel $\kernelope{h}(\cdot,\cdot;\lambda)$
in $L^2(\Vset^2, \Vsigma^{\otimes 2}, \xi^{\otimes 2})$. In the
following, we need the measurability of $\kernelope{h}$ on
$(\Vset^2\times\tore, \Vsigma^{\otimes 2}\otimes\btore)$ as given by
the following lemma.
\begin{lemma}\label{prop:existence-joint-kernelope-fiarma}
  Let $(\Vset,\Vsigma,\xi)$ be a $\sigma$-finite measure space and
  suppose that the Hilbert space $\cG_0=L^2(\Vset,\Vsigma,\xi)$ is
  separable. Let $K$ be a measurable function from $(\Lambda, \cA)$ to
  $\lr{\cS_2(\cG_0),\borel(\cS_2(\cG_0))}$. Then, there exists a function
  $\kernelope{K} : (v,v',\lambda)\mapsto\kernelope{K}(v,v';\lambda)$
  measurable from
  $(\Vset^2\times\Lambda,\Vsigma^{\otimes 2}\otimes\cA)$ to
  $(\cset,\borel(\cset))$ such that, for all $\lambda\in\Lambda$,
  $f\in\cH_0$ and $v\in\Vset$,
  \begin{equation}
    \label{eq:K-kernelope-identity}
    [K(\lambda)f](v)=\int \kernelope{K}(v,v';\lambda)\,f(v')\;\xi(\rmd v')\;.      
\end{equation}
Moreover, if $K\in\cL^2(\Lambda,\calA,\cS_2(\cG_0),\mu)$ for some
non-negative measure $\mu$ on $(\Lambda,\cA)$, then
$\kernelope{K}\in\cL^2(\Vset^2\times\Lambda,\Vsigma^{\otimes2}\otimes\calA,\xi^{\otimes2}\otimes\mu)$.
\end{lemma}
Based on this lemma, for all $\lambda\in\Lambda$,
the identity~(\ref{eq:K-kernelope-identity}) defines
$(v,v')\mapsto\kernelope{K}(v,v';\lambda)$ uniquely over $\Vset^2$ up to a
$\xi^{\otimes2}$-null set.
This allows us to introduce the following definition. 
\begin{definition}[Joint kernel of $\cS_2$-valued functions]
  \label{def:joint-kernels}
  Under the assumptions of
  \Cref{prop:existence-joint-kernelope-fiarma}, we call
  $\kernelope{K}$ the $\Lambda$-\emph{joint kernel} of $K$.
\end{definition}
Assuming that $D$ is normal allows us to characterize the domain of
definition of the $D$-order fractional integration operator filter, as
shown in the following result, which may be of independent interest.
\begin{theorem}\label{thm:cns-fi-operator-functional-fiarma}
  Let $\cH_0$ be a separable Hilbert space and
  $X=(X_t)_{t\in\zset}$ be a centered $\cH_0$-valued weakly stationary time
  series defined on $(\Omega,\cF,\PP)$ with spectral operator density
  $g_X$ with respect to a non-negative measure $\mu$ on
  $(\tore,\btore)$. Let $D$ be in $\cN(\cH_0)$ with singular value function
  $\mathrm{d}$ on $\cG_0:=L^2(\Vset, \Vsigma, \xi)$ and decomposition
  operator $U$.  Let $h\in L^2(\tore,\btore,\cS_2(\cG_0),\mu)$ be such that
  $\lambda\mapsto h(\lambda)[h(\lambda)]^\adjoint$ is the spectral operator density of $UX=(UX_t)_{t\in\zset}$ with respect to $\mu$, that is,
  $
  h(\lambda)[h(\lambda)]^\adjoint=U\,g_X(\lambda)\,U^\adjoint
  \quad\text{for }\mu\text{-}\mae\;\lambda\in\tore\;. 
  $
  Let $\kernelope{h}$ denote the $\tore$-joint kernel
  function of $h$. Then, the following assertions are
  equivalent.
  \begin{enumerate}[label=(\roman*)]
  \item \label{item:cns-fi-operator-functional1-fiarma} We have
    $X\in\processtransfer{\fracintoptransfer{D}}(\Omega,\cF,\PP)$.
  \item \label{item:cns-fi-operator-functional2-fiarma}  There exists
    $\eta\in(0,\pi)$ such that
    $$
\int_{\Vset^2\times\lr{(-\eta,\eta)\setminus\{0\}}}
|\lambda|^{-2\Re(\mathrm{d}(v))} \left|\kernelope{h}(v,v';\lambda)\right|^2\;\xi(\rmd
v)\xi(\rmd v')\mu(\rmd\lambda)<\infty\;.
$$
\item \label{item:cns-fi-operator-functional3-fiarma} We have
  \begin{equation}
  \label{eq:cns-fi-operator-functional-new-fiarma-pos}
\int_{\Vset^2\times\lr{(-\pi,\pi)\setminus\{0\}}}
|\lambda|^{-2\Re_+(\mathrm{d}(v))} \left|\kernelope{h}(v,v';\lambda)\right|^2\;\xi(\rmd
v)\xi(\rmd v')\mu(\rmd\lambda)<\infty\;\;,
\end{equation}
where, for all $z\in\cset$, $\Re_+(z)=\max\lr{0,(z+\bar{z})/2}$
denotes the non-negative real part of $z$.
\end{enumerate}
\end{theorem}
In the following theorem, we
specify \Cref{thm:cns-fi-operator-functional-fiarma} in the case where
$X$ is an ARMA process, as in  \cite{Li-long-memory-2020}, but we let
$D$ be any normal operator and not necessarily a scalar one. Our
necessary and sufficient condition relies on the following definition:
  \begin{equation}
    \label{eq:kernel_arma_around_zero-new-fiarma-defk_n}
    \aop_n\lr{\arpol,\mapol} =
    \lrcb{\lrb{[\arpol]^{-1}\,\mapol}\circ\exp}^{(n)}(0) \;,\quad n\in\nset\;,\arpol\in\mathcal{P}_p^\ast\;,\mapol\in\mathcal{P}_q\;,
  \end{equation}
where the exponent $^{(n)}$ here denotes the $n$-th derivative of the
mapping $z\mapsto [\arpol(\rme^{z})]^{-1}\,\mapol(\rme^z)$, which is
infinitely differentiable in a neighborhood of $z=0$ as a
$\cL_b(\cH_0)$-valued function, since
$\arpol\in\mathcal{P}_p^\ast(\cH_0)$. In fact we have
$$
\aop_0\lr{\arpol,\mapol}=
[\arpol(1)]^{-1}\mapol(1)\quad\text{and}\quad
\aop_n\lr{\arpol,\mapol} =\sum_{k=1}^n
b_{n,k}\;\lrb{[\arpol]^{-1}\,\mapol}^{(k)}(1)\;,\quad n\geq1\;,
$$
where $b_{n,k}$ are known positive rational coefficients obtained by taking the exponential
Bell $n$-order Bell polynomial at $(1,\dots,1)$. 

We have the following result.
\begin{theorem}\label{cor:cns-arma-fi-operator-functional-fiarma}
  Let $\cH_0$ be a separable Hilbert space.
  Let $X$ be an $\cH_0$-valued ARMA($p,q)$ process defined by
  $\hat
  X(\rmd\lambda)=\left[\arpol(\rme^{-\rmi\lambda})\right]^{-1}\mapol(\rme^{-\rmi\lambda})\hat
  Z(\rmd\lambda)$ with $\mapol\in\mathcal{P}_q(\cH_0)$,
  $\arpol\in\mathcal{P}_p^\ast(\cH_0)$ and $Z$ a white noise with
  covariance operator $\Sigma$ defined on $(\Omega,\cF,\PP)$.
  Let $D \in \cN(\cH_0)$, with singular value function
  $\mathrm{d}$ on $\cG_0:=L^2(\Vset, \Vsigma, \xi)$ and decomposition
  operator $U$.   
  Let $\sigma_n : v \mapsto \varsqrt{\PEarg{\abs{W_n(v,\cdot)}^2}}$
  where $W_n = U\,\aop_n\lr{\arpol,\mapol}\,Z_0$ is seen as a
  $\cset$-valued function defined on $\Vset\times\Omega$.
  Then, the two following assertions are equivalent:
  \begin{enumerate}[label=(\roman*),series=cns-fractional]
  \item\label{ass:cns-arma-fi-operator-functional-cond0-fiarma} We
    have
    $X\in\processtransfer{\fracintoptransfer{D}}(\Omega,\cF,\PP)$.
  \item\label{ass:cns-arma-fi-operator-functional-cond1-fiarma} For
    all $n\in\nset$, we have
    $\Re(\mathrm{d}) < n+1/2$, $\xi\text{-}\mae$ on $\{\sigma_n >
    0\}$ and
    \begin{equation}
      \label{eq:cns-arma-sigma_n}
      \int_{\{\Re(\mathrm{d}) < n+1/2\}} \frac{\sigma_n^2(v)}{1 + 2n - 2 \Re(\mathrm{d}(v))} \, \xi(\rmd v) < \infty\;.
    \end{equation}
  \end{enumerate}
\end{theorem}
Note that there is a slight abuse of notation in the definition of
$\sigma_n$ since the definition of measurability for a
$L^2(\Vset,\cV,\xi)$-valued random variable does not necessarily
ensure measurability of $W_n(v,\cdot)$ as a $\cset$-valued random
variable. This abuse of notation is common in the literature on
functional data analysis and is harmless because we can always find a
version of $W$ which is jointly measurable on
$(\Vset\times\Omega,\cV\otimes\cF)$, see
\Cref{prop:joint-meas-version-fiarma} in \Cref{sec:functional-time-series-fiarma}.
\begin{remark}\label{ex:rem-existence-fiarma-particular-cases}
Let us briefly comment on
\Cref{cor:cns-arma-fi-operator-functional-fiarma}.
\begin{enumerate}[label=(\arabic*)]
\item First, by definition of $W_n$ and
$\sigma_n$, we have $\sigma_n\in L^2(\Vset, \cV,\xi)$ and thus
Assertion~\ref{ass:cns-arma-fi-operator-functional-cond1-fiarma}
always holds
for $n$ large enough, namely, for all $n>\sup(\Re(\rmd))-1/2$
(recall that the singular value function of a normal operator is
bounded).
\item\label{item:WN-case} If $\mapol=\arpol$ is the unit polynomial, the ARMA process $X$ equals the white noise
  process $Z$, $\aop_n(\arpol,\mapol)=\Id_{\cH_0}$, and
  $\sigma_n=\sigma_0$ for all $n\in\nset$, so that
  Condition~\ref{ass:cns-arma-fi-operator-functional-cond1-fiarma}
  only needs to be verified for $n=0$. 
\item \label{item:finite-dim-case}
  In the
  $N$-dimensional case with $n$ finite, we have $\Vset=\{1,\dots,N\}$,
  $\xi$ is the counting measure on $\Vset$, and $U$ can be interpreted
  as a $n\times n$ unitary matrix, and $\mathrm{d}$ and $\sigma_n$ as
  $N$-dimensional
  vectors. Condition~\ref{ass:cns-arma-fi-operator-functional-cond1-fiarma}
  then says that $\Re(\mathrm{d}(k)) < n+1/2$ for all $n\in\nset$
  and 
  $k\in\{1,\dots,N\}$ such that $\sigma_n(k) > 0$. 
\item For the
  real univariate case ($N=1$, $D=d\in\rset$ in~\ref{item:finite-dim-case}),
  Condition~\ref{ass:cns-arma-fi-operator-functional-cond1-fiarma}
  says that $d < n_0+1/2$, where $n_0$ is the smallest $n$ such that $\sigma_n>0$.
  Ruling out the case where $Z$ is the null process (in which case  $\Sigma=0$
  and $\sigma_n=0$ for all $n\in\nset$), one can see that
  $n_0$ equals 0 if $\mapol(1)\neq0$ and $n_0$ equals
  the order of multiplicity of 1 as a root of $\mapol$ otherwise (in
  other words, it corresponds to the difference
  operator largest order contained in the  MA operator). 
  In particular, we find the usual $d<1/2$ condition for
  the existence of a weakly stationary ARFIMA$(p,d,q)$ model in the
  case where the underlying ARMA($p,q$) process is canonical
  ($\arpol$ and $\mapol$ do not vanish on the unit disk). If $n_0\geq1$, the
  usual convention is to include the difference operator as a negative
  exponent of the fractional integration operator hence leading to an
  ARFIMA$(p,d-n_0,q-n_0)$ with $d-n_0<1/2$. 
\item\label{item:li-robinson-case} We already
  mentioned in~(\ref{eq:farima-robinson-fiarma}) the case treated in
  \cite[Section~4]{Li-long-memory-2020}. In the setting
  of \Cref{cor:cns-arma-fi-operator-functional-fiarma}, it 
  corresponds to the case where $D=d\times\Id_{\cH_0}$ is a scalar operator on
  $\cH_0=\cG_0=L^2(\Vset,\borel(\Vset),\xi)$ for a compact subset $\Vset$ of
  $\rset$, $\xi$ being the Lebesgue measure on $\Vset$ and $-1/2<d<1/2$ (thus
  $\rmd(v)\equiv d$ and $U=\Id_{\cH_0}$). Under this assumption,
  Condition~\ref{item:finite-dim-case} trivially holds since
  $1+2n-2d>0$ and $\sigma_n\in L^2(\Vset,\borel(\Vset),\xi)$ for all
  $n\in\nset$.
\end{enumerate}
\end{remark}

\subsection{Other long memory processes}\label{sec:other-long-memory-fiarma}
  Several non-equivalent definitions of
  long rang dependence, or long memory, are available in the literature
  for time series. Some approaches focus on the behavior of the
  autocovariance function at large lags, others on the spectral
  density at low frequencies (see
  \cite[Section~2.1]{pipiras_taqqu_2017} and the references
  therein). Separating short range from long range dependence is often
  made more natural within a particular class of models. For instance,
  for a Hilbert space-valued process $Y = (Y_t)_{t \in \zset}$, one may rely
  on a causal linear representation, namely
\begin{equation}\label{eq:linear-process-fiarma}
  Y_t = \sum_{k=0}^\infty\aop_k \epsilon_{t-k}\, , \quad t \in \zset\;, \quad
 \end{equation}
where $\epsilon = (\epsilon_t)_{t \in \zset}$ is a centered white
noise valued in the separable Hilbert space $\cH_0$ and
$(\aop_k)_{k \in \zset}$ is a sequence of $\cL_b(\cH_0)$
operators.
A sufficient condition for 
convergence of this series in $\cM(\Omega, \cF, \cH_0, \PP)$ is that 
$\sum_{k=0}^\infty \norm{\aop_k}_{\infty} < \infty$, and this
assumption is referred to as the \emph{short range dependence} (or
short memory) case, in contrast to
\emph{long range dependence} (long memory) case, for which
$\sum_{k=0}^\infty \norm{\aop_k}_{\infty} = \infty$, under
which the convergence in~(\ref{eq:linear-process-fiarma}) is no
longer granted. The case where
$\aop_k = (k+1)^{-N}$ for some 
$N \in \cN(\cH_0)$ is investigated in \cite{duker_limit_2018}.
More precisely, let  $\mathrm{n}$ and $U$ be the singular value function and decomposition operator of $N$ on $\cG_0:=L^2(\Vset, \Vsigma, \xi)$. Assume  that 
  \begin{equation}\label{eq:ass-LRD-fiarma}
\mathrm{h} > \frac{1}{2} \text{ $\xi$-a.e.}\quad\text{and}\quad    \int_\Vset \frac{\sigma_W^2(v)}{2 \mathrm{h}(v) - 1} \, \xi(\rmd v) < \infty \; ,
\end{equation}
where $\mathrm{h} : v \mapsto \Re(\mathrm{n}(v))$ and $\sigma_W^2 : s \mapsto \PEarg{\abs{W(v,\cdot)}^2}$ with $W=U \epsilon_0$. Then, using the arguments of the proof of \cite[Lemma~A.1]{duker_limit_2018}, one can show that, for all $t \in \zset$, 
\begin{equation}\label{eq:long-memory-time-domain-duker-fiarma}
  Y_t = \sum_{k=0}^{\infty} (k+1)^{-N} \epsilon_{t-k}
\end{equation}
converges in  $\cM(\Omega,\cF, \cH_0, \PP)$.

In \cite[Theorem 2.1]{duker_limit_2018}, the author also studies the partial sums
of the process \eqref{eq:long-memory-time-domain-duker-fiarma} and exhibits asymptotic properties
which naturally extend the usual behavior observed for univariate long memory processes. 
In the following, we explain how the process (\ref{eq:long-memory-time-domain-duker-fiarma}) can be
related to a \fiarma(D,0,0) process. First, we prove 
that Condition~(\ref{eq:ass-LRD-fiarma}) also implies the existence of this
\fiarma\ process.
\begin{lemma}\label{lem:duker-fiarma-fiarma}
  Condition~(\ref{eq:ass-LRD-fiarma}) implies $\epsilon\in\processtransfer{\fracintoptransfer{D}}(\Omega,\cF,\PP)$ with $D=\Id_{\cH_0}-N$. 
\end{lemma}
We can now state a result which shows that the process defined by
(\ref{eq:long-memory-time-domain-duker-fiarma}) is closely related to a \fiarma\ process up to a bounded operator $C$ and to an additive short memory process.
\begin{proposition}\label{prop:longmemory-fiarma}
  Assume that (\ref{eq:ass-LRD-fiarma}) holds and define $Y=(Y_t)_{t\in\zset}$
  by~(\ref{eq:long-memory-time-domain-duker-fiarma}). Then, there exists $C \in \cL_b(\cH_0)$ and
$(\Delta_k)_{k \in \nset} \in \cL_b(\cH_0)^\nset$ with
$\sum_{k=0}^{\infty} \norm{\Delta_k}_{\infty} < \infty$ such that
$$
\filterprocess{\fracintoptransfer{D}}(\epsilon)=C\,Y+Z\;,
$$
where $Z$ is the  short memory process defined, for all $t \in \zset$, by $Z_t = \sum_{k= 0}^\infty \Delta_k \epsilon_{t-k}$.
\end{proposition}

\section{Prediction and estimation}
\label{sec:estimation}
\subsection{Main assumptions and preliminary result}
\label{sec:est-main-assumpt-main}
We denote by $\lebtore$ the Lebesgue measure on $\lr{\tore,\btore}$
divided by $2\pi$, so that for any locally integrable $2\pi$-periodic
function $g$,
$$
\int g\;\rmd \lebtore=(2\pi)^{-1}\int_{\tore} g(x)\;\rmd x=(2\pi)^{-1}\int_{-\pi}^{\pi} g(x)\;\rmd x\;.
$$
Let $\cH_0$ be a separable Hilbert space, $X=(X_t)_{t\in\zset}$ be a
process defined on $(\Omega,\cF,\PP)$, and valued in $\cH_0$, and
consider the following assumptions.
\begin{enumerate}[label=(\textbf{A-\arabic*}),series=est-assump]
\item\label{item:ergo-assump}   The process $X$ is
  stationary and ergodic.
\item \label{item:sp_density-assump} The process $X$ is weakly stationary.
\end{enumerate}
Under~\ref{item:sp_density-assump}, we always denote by $\nu_X$ the
spectral operator measure of $X$. Denote the discrete Fourier
coefficients of $X_1,\dots,X_n$ by
\begin{equation}
    \label{eq:def-dnX}
d_n^X(\lambda)=\frac1{\sqrt{n}}\sum_{k=1}^n X_k\,\rme^{-\rmi\lambda k}\;,\quad\lambda\in\tore\;,
\end{equation}
and the periodogram by
$$
I^X_n(\lambda)= d^X_n(\lambda)\otimes d^X_n(\lambda)\;,\quad\lambda\in\tore\;.
$$
If $X$ is not a centered process, one can use the empirical mean to
center it, that is, in~(\ref{eq:def-dnX}), replace $X_k$ by
$X^c_{n,k}$ as defined in~(\ref{eq:centering}), in which case we
denote the corresponding discrete Fourier coefficients and the corresponding
periodogram by $d_n^{X^c_n}$ and $I_n^{X^c_n}$, respectively.

The periodogram is related to the empirical covariance estimators
through the following identity, for all $s,t\in\zset$,
  \begin{equation}
    \label{eq:emp-cov-centered}
  \hat{\Gamma}_n(s-t)=\int
  I_n^{X^c_n}(\lambda)\,\rme^{\rmi\,(s-t)\,\lambda}\;\lebtore(\rmd\lambda)=\frac1n\sum_{\stackrel{1\leq k,k'\leq
      n}{k-k'=(s-t)}}
  X^c_{n,k}\otimes X^c_{n,k'}\;.    
  \end{equation}
The integral in this equation can be interpreted as a
sesquilinear functional $\cQ_{I_n^{X^c_n}}$ applied to exponential functions
$\lambda\mapsto\rme^{\rmi\,s\,\lambda}$ on the left and
$\lambda\mapsto\rme^{\rmi\,t\,\lambda}$ on the right, where, for any operator functions $L,g$ and $R$ defined on
$\tore$, we set 
\begin{align}\label{eq:cQ-def-Leb}
  \cQ_g(L,R)= \int  L\, g \,R^\adjoint\;\rmd\lebtore \;.
\end{align}
Similarly, if $\nu$ is a trace-class \povm\ defined on
$(\tore,\btore)$ and valued in
$\cS_1^+(\cH_0)$, we set 
\begin{align}\label{eq:cQ-def-general}
  \cQ_\nu(L,R)= \int  L\, \rmd\nu \,R^\adjoint\;. 
\end{align}
To ensure that these integrals are well defined, we assume that $L$
and $R$ are measurable bounded functions valued in
$\cL_b(\cH_0,\cG_0)$, with $\cG_0$ an additional separable Hilbert
space. Namely, for any Banach space $\lr{\cE,\norm{\cdot}_{\cE}}$, we
further denote, by $\mathbb{F}_b\lr{\tore,\btore,\cE}$ the set of
bounded measurable functions from $\lr{\tore,\btore}$ to
$\lr{\cE,\borel(\cE)}$, and we endow
$\mathbb{F}_b\lr{\tore,\btore,\cE}$ with the sup norm, which, for all
$L\in\mathbb{F}_b\lr{\tore,\btore,\cE}$, we denote by
$$
\sup\lr{L}=\sup_{\lambda\in\tore}\norm{L(\lambda)}_{\cE}\;.
$$
Then, for all $g$ valued in $\cS_1(\cH_0)$ and
$L,R\in\mathbb{F}_b\lr{\tore,\btore,\cL_b(\cH_0,\cG_0)}$, we have
$\cQ_g(L,R)\in\cS_1(\cG_0)$ with
\begin{equation}
  \label{eq:firsCQbound}
\norm{\cQ_g(L,R)}_1\leq \sup(L)\,\sup(R)\,\int \norm{g}_1 \rmd\lebtore
\;,  
\end{equation}
and similarly, for  a trace-class \povm\ $\nu$ valued in
$\cS_1^+(\cH_0)$, we have
$\cQ_\nu(L,R)\in\cS_1(\cG_0)$ with
\begin{equation}
  \label{eq:firsCQbound-nuX}
\norm{\cQ_\nu(L,R)}_1\leq \sup(L)\,\sup(R)\,\norm{\nu}_1(\tore)\;.
\end{equation}
We denote by $\mathbb{F}_{b,b}(\cH_0,\cG_0)$ the product vector space
$\mathbb{F}_b\lr{\tore,\btore,\cL_b(\cH_0,\cG_0)}\times\mathbb{F}_b\lr{\tore,\btore,\cL_b(\cH_0,\cG_0)}$,
endowed with the max norm defined by
$$
\norm{(L,R)}_{b,b}:=
\max(\sup(L),\sup(R))\quad\text{for all}\quad (L,R)\in\mathbb{F}_{b,b}(\cH_0,\cG_0)\;.
$$
For any two metric spaces $\lr{\cE_1,d_1}$ and $\lr{\cE_2,d_2}$,
$\cC\lr{\cE_1,\cE_2}$ denotes the space of continuous functions from
$\cE_1$ to $\cE_2$. If $F$ and $F_n$ are in $\cC\lr{\cE_1,\cE_2}$ for all $n\in\nset$ and
$$
\lim_{n\to\infty}\sup_{x\in\cE_1}d_2\lr{F_n(x),F(x)} = 0 \;,
$$
we say that $(F_n)_{n\in\nset}$ converges to $F$ uniformly in
$\cC\lr{\cE_1,\cE_2}$.

Using these definitions, immediate properties of the quadratic
functionals $\cQ_{I_n^{X^c_n}}$ and $\cQ_{\nu_X}$ are summarized in
the following proposition.
\begin{proposition}
  \label{prop:cQ-well-def} For any $X_1,\dots,X_n$ in $\cH_0$, the
  mapping $\cQ_{I_n^{X^c_n}}$ is well defined and belongs to
  $\cC(\mathbb{F}_{b,b}(\cH_0,\cG_0),\cS_1\lr{\cG_0})$. If~\ref{item:sp_density-assump}
  holds, we also have
  $\cQ_{\nu_X}\in\cC(\mathbb{F}_{b,b}(\cH_0,\cG_0),\cS_1\lr{\cG_0})$.
\end{proposition}
\begin{proof}
  Observe that, under the given assumptions,
  $X_{n,1}^c,\dots,X_{n,n}^c$ all are in $\cH_0$, and so is
  $d_n^{X^c_n}(\lambda)$. Moreover, $\norm{d_n^{X^c_n}(\lambda)}_{\cH_0}$ is bounded
  independently of $\lambda$. Consequently, $I_n^{X^c_n}$ is valued
  in $\cS^{1}(\cH_0)$ and its trace-norm is integrable over
  $\tore$. The result on $\cQ_{I_n^{X^c_n}}$ thus follows
  from~(\ref{eq:firsCQbound}). Under~\ref{item:sp_density-assump},
  $\nu_X$ is a trace-class \povm\ $\nu$ valued in $\cS_1^+(\cH_0)$,
  with $\norm{\nu}_1(\tore)=\PEarg{\norm{X_0}_{\cH_0}^2}<\infty$. The
  result on $\cQ_{\nu_X}$ thus follows
  from~(\ref{eq:firsCQbound-nuX}).
\end{proof}
Our next result only exploits~\ref{item:ergo-assump}
and~\ref{item:sp_density-assump}, and is thus of independent
interest. It is a uniform convergence result for integral quadratic
functionals based on the periodogram. It applies to a parameterized
pair of bounded operators functions defined on $(\tore,\btore)$ and
valued in $\cL_b(\cH_0,\cG_0)$. More precisely, for $L$ and $R$ in
$\cC\lr{\Theta\times\tore,\cL_b(\cH_0,\cG_0)}$, for all $n\geq1$, we
define
$\tilde{\cQ}_{I_n^{X^c_n}}^{(L,R)}:\theta\mapsto\cQ_{I_n^{X^c_n}}\lr{L(\theta,\cdot),R(\theta,\cdot)}$.
We also define
$\tilde{\cQ}_{\nu_X}^{(L,R)}:\theta\mapsto\cQ_{\nu_X}(L(\theta,\cdot),R(\theta,\cdot))$,
with $\cQ_{I_n^{X^c_n}}$ and $\cQ_{\nu_X}$ as in \Cref{prop:cQ-well-def}.

We can now state a first result on the convergence of the periodogram
quadratic functional, in the case where the left and right operator
functions' arrival spaces are finite-dimensional.
\begin{theorem}
\label{thm:periodo-func-fidi}
Let $\cH_0$ be a separable Hilbert space and $\cG_0$ be a finite
dimensional space. Let $X=(X_t)_{t\in\zset}$ be a
process defined on $(\Omega,\cF,\PP)$ and valued in $\cH_0$ satisfying~\ref{item:ergo-assump}
and~\ref{item:sp_density-assump} and let $(\Theta,\Delta)$ be
a compact metric space. Let $L$ and $R$ in
$\cC\lr{\Theta\times\tore,\cL_b(\cH_0,\cG_0)}$.  Then, we have
  \begin{equation}
    \label{eq:unif-periodo-mapping-convergence-fidi}
\lim_{n\to\infty}\tilde{\cQ}_{I_n^{X^c_n}}^{(L,R)}=\tilde{\cQ}_{\nu_X}^{(L,R)}\quad\text{uniformly
  in}\quad\cC\lr{\Theta,\cS_1\lr{\cG_0}}\;,\quad\as
\end{equation}
\end{theorem}
We now consider the case where $\cG_0=\cH_0$ and $\cH_0$ is an
infinite-dimensional separable Hilbert space. In order to obtain the
convergence in $\cS_1(\cH_0)$, we will rely on an 
additional assumption  in this case. To this end, for any sequence
$s=(s_k)_{k\in\nset}\in[1,\infty)^\nset$ and any orthonormal sequence
$\lr{\phi_k}_{k\in\nset}$ of $\cH_0$, we set
\begin{equation}
  \label{eq:defHOs}
\cH_0^s=\set{x\in\cspanarg[\cH_0]{\phi_k, k\in\nset}}{\sum_{k\in\nset}s_k^2\,\lrav{\pscal{x}{\phi_k}_{\cH_0}}^2<\infty}\;.
\end{equation}
A typical example of such spaces are the Sobolev spaces with index $\alpha>0$ where
$\lr{\phi_k}_{k\in\nset}$ is a well chosen Hilbert basis (i.e. orthonormal
  and complete in $\cH_0$)  and $s_k=(1+k)^{\alpha}$. 
The space $\cH_0^s$ is a subspace of $\cH_0$ and is itself a
separable Hilbert space endowed with the inner product
\begin{equation}
  \label{eq:pscalHOs}
\pscal{x}{y}_{\cH_0^s}=\sum_{k\in\nset}s_k^2\,\pscal{x}{\phi_k}_{\cH_0}
\overline{\pscal{y}{\phi_k}_{\cH_0}}\;.  
\end{equation}
Setting $\xi_k=s_k^{-1}\,\phi_k$ for all $k\in\nset$, we note that
$\lr{\xi_k}_{k\in\nset}$ is a Hilbert basis of $\cH_0^s$.

Using the space $\cH_0^s$ that we have just introduced, 
we have the following result for the infinite-dimensional case.
\begin{theorem}
\label{thm:periodo-func}
Let $\cH_0$ be an infinite-dimensional separable Hilbert space. Let
$X=(X_t)_{t\in\zset}$ be a process defined on $(\Omega,\cF,\PP)$ and
valued in $\cH_0$ satisfying~\ref{item:ergo-assump}
and~\ref{item:sp_density-assump} and let $(\Theta,\Delta)$ be a
compact metric space. Let $L$ and $R$ in
$\cC\lr{\Theta\times\tore,\cL_b(\cH_0)}$. Let  $s=(s_k)_{k\in\nset}\in[1,\infty)^\nset$
and  $(\phi_k)_{k\in\nset}$ be an orthonormal sequence of $\cH_0$. 
Define the Hilbert space  $\cH_0^s$ by~(\ref{eq:defHOs})
and~(\ref{eq:pscalHOs}). We suppose that the three following
assertions hold.
\begin{enumerate}[label=(\roman*)]
\item\label{item:s-sob-assump} The sequence $s=(s_k)_{k\in\nset}\in[1,\infty)^\nset$ is non-decreasing and
  goes to $\infty$.
  \item\label{item:X-sob-assump} We have $X_0\in\cH_0^s$
    $\as$ with $\PE\lrb{\norm{X_0}_{\cH_0^s}^2}<\infty$.
  \newcounter{primcount}
  \setcounter{primcount}{\value{enumi}} 
    \item\label{item:LR-sob-assump} Defining $L_s$ and $R_s$ by
$L_s(\theta,\lambda)=L(\theta,\lambda)_{|\cH_0^s}$ and
$R_s(\theta,\lambda)=R(\theta,\lambda)_{|\cH_0^s}$ for all
$(\theta,\lambda)\in\Theta\times\tore$,  we have $L_s$ and $R_s$
in $\cC\lr{\Theta\times\tore,\cL_b(\cH_0^s)}$.
\end{enumerate}
Then, the following convergence holds.
  \begin{equation}
    \label{eq:unif-periodo-mapping-convergence-infidi}
\lim_{n\to\infty}\tilde{\cQ}_{I_n^{X^c_n}}^{(L,R)}=\tilde{\cQ}_{\nu_X}^{(L,R)}\quad\text{uniformly
  in}\quad\cC\lr{\Theta,\cS_1\lr{\cH_0}}\;,\quad\as
\end{equation}
\end{theorem}
In fact, as shown by \Cref{lem:A1A2implyA4}
  in \Cref{sec:proof-estimation}, Assumptions~\ref{item:ergo-assump}
and~\ref{item:sp_density-assump} imply
Conditions~\ref{item:s-sob-assump} and~\ref{item:X-sob-assump} of
\Cref{thm:periodo-func} for a well chosen $s$. This fact is useful to
prove in \Cref{thm:1}, by applying \Cref{thm:periodo-func}  for a specific choice
of $L$ and $R$, for which \ref{item:LR-sob-assump} holds for any
sequence $s$, see the proof of  \Cref{thm:1} in
\Cref{sec:proof-thm:1}.

\subsection{FIARMA prediction and estimation}
\label{sec:examples-est-application}

A common tool for $M$-estimation for finite dimensional time series is
the Whittle contrast, which relies on a Gaussian approximation of
$d_n^X$ as $n\to\infty$, hence suggesting to use a Gaussian likelihood
contrast for $d_n^X$, based on its asymptotic covariance operator.  We
still have such an approximation for time series valued in a Hilbert
space, see \cite[Theorem~1]{CEROVECKI2017282}. However, using the
Whittle approach directly in infinite dimension does not seem to be
directly applicable. Indeed, Gaussian distributions are generally
singular to each other in infinite dimension. In particular, the 
log determinant term of the noise covariance matrix appearing in the
Whittle contrast (for example, see the first term in $\bar{L}_N$ on Page~344
of \cite{dunsmuir_hannan_1976}), is not well defined when this matrix
is replaced by an infinite dimensional covariance operator.

There are two possible ways of circumventing this issue. The first one
is to project the data on some finite-dimensional subspace for
statistical inference and then study the behavior of estimators as the
dimension of this subspace diverges. The second one is to work on a least square criterion, which does
not include the optimization of the noise covariance operator.  Here
we investigate this second approach.  We first derive the best
one-step ahead predictor of a FIARMA process (see
\Cref{thm:fiarma-predictor}). We then show that, under some condition,
such a predictor can be estimated from the data (see
\Cref{thm:fiarma-predictor-esimation}).

Recall that, for any integer $d\in\nset$, $\mathcal{P}_d(\cH_0)$
denotes the set of polynomials $\apol$ of degree $d$ with coefficients
in $\cL_b(\cH_0)$, and such that $\apol(0)=\Id_{\cH_0}$. In the
following, we further denote by $\mathcal{P}_d^{\dagger}(\cH_0)$, the
set of all $\apol\in\mathcal{P}_d(\cH_0)$, which are invertible on the
closed unit disk $\overline{\unitdisk}$.

\begin{theorem}
  \label{thm:fiarma-predictor}
  Let $\cH_0$ be a separable Hilbert space and $p,q$ be two
  non-negative integers.  Let $Y$ be an $\cH_0$-valued \fiarma\
  process, as in \Cref{def:fiarma-fiarma}, with long-memory operator
  $D\in\cN(\cH_0)$, MA polynomial
  $\mapol\in\mathcal{P}_q^\dagger(\cH_0)$, AR polynomial
  $\arpol\in\mathcal{P}_p^\dagger(\cH_0)$ and an $\cH_0$-valued centered
  white noise $Z$.   Define, for all $z\in\cset\setminus[1,\infty)$,
  \begin{equation}
    \label{eq:best-pred-fiarma-def}
    \fiarmapred\lr{z}=
      \Id_{\cH_0}-\lrcb{\mapol(z)}^{-1}\,\arpol(z)\,(1-z)^{D}\;.
    \end{equation}
  Then, $\lambda\mapsto\fiarmapred\lr{\rme^{\-\rmi\lambda}}$ belongs
  to $\hat{\cH}^Y$, and, for all $t\in\zset$, the best linear predictor of $Y_t$
  given its past $\set{Y_s}{s\leq t-1}$ is given by
  \begin{equation}
    \label{eq:best-pred-fiarma}
    \proj{Y_t}{\cH^Y_{t-1}} = \int_\tore \rme^{i\lambda t} \; \fiarmapred[\mapol,\arpol,D]\lr{\rme^{-\rmi\lambda}}\;\hat{Y}(\rmd\lambda)\;,
  \end{equation}
  where $\proj{Y_t}{\cH^Y_{t-1}}$ denotes the orthogonal projection of
  $Y_t$ onto the closed space
  $$
  \cH^Y_{t-1}=\cspanarg[\cM(\Omega,\cF, \cH_0, \PP)]{\aop\,Y_s\;,\,s=t-1,t-2,\dots,\;\aop\in\cL_b(\cH_0)}\;.
  $$
\end{theorem}
We now derive the best predictor among a collection of
FIARMA predictors from a finite sample $X_1,\dots,X_n$. We will
consider ARMA predictors or positive long memory FIARMA predictors. More precisely, Define
  \begin{equation}
    \label{eq:Dond-positive-long-memory}
    \cN^\dagger(\cH_0) :=\set{D\in\cN(\cH_0)}{\Re(D)\text{ is invertible}}\;,
\end{equation}
where $\Re(D)=(D+D^\adjoint)/2$ denote the real part of $D$. As a
subset of $\cL_b(\cH_0)$, we endow $\cN^\dagger(\cH_0)$ with the
topology inherited from the operator norm $\norm{\cdot}_\infty$. Our main
assumption on the FIARMA parameters is the following.
\begin{enumerate}[resume*=est-assump]
\item\label{item:fiarma-mode} The FIARMA parameters
  $\fiarmacol=\lr{D_{\theta},\arpol_\theta,\mapol_\theta}_{\theta\in\Theta}$
  are valued in
  $\cN(\cH_0)\times\mathcal{P}_p^\dagger(\cH_0)\times\mathcal{P}_q^\dagger(\cH_0)$
  for some $p,q\in\nset^2$ and indexed by a compact metric space
  $(\Theta,\Delta)$. Moreover, the mappings
  $(\theta,\lambda)\mapsto\arpol_\theta\lr{\rme^{-\rmi\lambda}}$ and
  $(\theta,\lambda)\mapsto\mapol_\theta\lr{\rme^{-\rmi\lambda}}$
  belong to
  $\cC\lr{\Theta\times\tore,\cL_b(\cH_0)}$ and one
  of the following two assertions hold.
  \begin{enumerate}[label=(\roman*)]
  \item\label{item:arma-pred} For all
  $\theta\in\Theta$, $D_\theta=0$.  
\item\label{item:long-memory-fiarma-pred} The mapping
  $\theta\mapsto D_\theta$ belongs to $\cC\lr{\Theta,\cN^\dagger(\cH_0)}$.
  \end{enumerate}
\end{enumerate}
In~\ref{item:fiarma-mode}, Condition~\ref{item:arma-pred} corresponds
to using an ARMA predictor. Therefore, we will call $\fiarmacol$ an ARMA
predictor model in this case. 
Condition~\ref{item:long-memory-fiarma-pred} corresponds to using a
FIARMA predictor with  \emph{positive} long-memory. Therefore, we will call $\fiarmacol$ a positive FIARMA
predictor model in this case.

Our goal is now to derive, based on a finite sample $X_1,\dots,X_n$,
an approximation of the best possible $\fiarmacol$-prediction of a
weakly stationary process $X$ taken among the ARMA or FIARMA
predictors defined by a collection $\fiarmacol$
satisfying~\ref{item:fiarma-mode}. We precise what we mean by this best
prediction in the following result, for a centered weakly stationary
process $Y$. We treat both the case where the model
is well specified and the case where it is not. Recall that we say
that the model $\fiarmacol$ of~\ref{item:fiarma-mode} is
well-specified for $Y$ when $Y$ is indeed a FIARMA process with a
FIARMA parameter $(\mapol,\arpol,D)$ among $\fiarmacol$.

\begin{proposition}[Definition of $\fiarmacol$-best prediction]\label{prop:firama-pred-def}
  Let $\cH_0$ be an infinite-dimensional separable Hilbert space and
  $p,q$ be two non-negative integers. Let $Y=(Y_t)_{t\in\zset}$ be a
  centered weakly stationary process defined on $(\Omega,\cF,\PP)$ and
  valued in $\cH_0$ and let
  $\fiarmacol$
  be a model satisfying~\ref{item:fiarma-mode}. Then, we have the following
  facts and definitions.
  \begin{enumerate}[label=(\roman*)]
  \item\label{item:prop:firama-pred-def1} For all
  $\theta\in\Theta$, there exists an absolutely summable $\cL_b(\cH_0)$-valued sequence
  $\lr{\fiarmapredcoef[\theta]{k}}_{k\geq1}$ such that, for all
  $\lambda\in\tore\setminus\{0\}$,
  \begin{equation}
    \label{eq:fiarmapred-def2}
    \fop^{\dagger}(\theta,\lambda) :=
    \Id_{\cH_0}-\sum_{k=1}^\infty
    \fiarmapredcoef[\theta]{k}\,\rme^{-\rmi\lambda\,k}
    =    \lrb{\mapol _{\theta} (\rme^{-\rmi\lambda})}^{-1}\,\arpol _{\theta} (\rme^{-\rmi\lambda})\,\lr{1-\rme^{-\rmi\lambda}}^{D _{\theta}} \;.
  \end{equation}
\item\label{item:prop:firama-pred-def2} For all $t\in\zset$ and $\theta\in\Theta$, we can define 
  \begin{equation}
    \label{eq:fiarmacol-pred}
  \hat{Y}_t(\theta):=\sum_{k=1}^\infty\fiarmapredcoef[\theta]{k}\,Y_{t-k}\in\cH^Y_{t-1}\;.    
  \end{equation}
\item\label{item:prop:firama-pred-def3}  The \emph{best $\fiarmacol$-prediction quadratic risk} of $Y$ is defined as
  \begin{equation}
    \label{eq:fiarmacol-pred-best-error}
  \mathbb{E}^2\lr{Y,\fiarmacol}=\inf_{\theta\in\Theta}\PEarg{\norm{Y_t-\hat{Y}_t(\theta)}_{\cH_0}^2}\;.
\end{equation}
which does not depend on $t$ by weak stationarity of $Y$.
\item\label{item:prop:firama-pred-def4} The inf in~(\ref{eq:fiarmacol-pred-best-error})  is attained in $\Theta$ (hence is a minimum) and we
call the argmin set   the \emph{set of best $\fiarmacol$-predictors} for
$Y$, denoted by
  \begin{equation}
    \label{eq:fiarmacol-pred-best-set}
    \Theta^\ast_Y:=\set{\theta\in\Theta}{\PEarg{\norm{Y_t-\hat{Y}_t(\theta)}_{\cH_0}^2}=  \mathbb{E}^2\lr{Y,\fiarmacol}}
  \end{equation}
  Then, $\Theta^\ast_Y$ is a compact subset of $\Theta$.
\item\label{item:prop:firama-pred-def5} 
If there
  exists $\hat{Y}^\ast_t\in\cH_{t-1}^Y$ such that the subset
  $\set{\hat{Y}_t(\theta)}{\theta\in\Theta^\ast_Y}$ of $\cH_{t-1}^Y$
  reduces to the singleton $\hat{Y}^\ast_t$, we call $\hat{Y}^\ast_t$
  the \emph{best $\fiarmacol$-predictor} of $Y_t$. Otherwise we say that the
  best $\fiarmacol$-predictor of $Y_t$ is \emph{not well defined}.
\item\label{item:prop:firama-pred-def6} When the best  $\fiarmacol$-predictor of $Y_t$ is well defined for one
  $t$ it is well defined for all $t$. Moreover, in this case, there
  exists a set of probability one on which, for all $t\in\zset$ and
  $\theta\in\Theta^\ast_Y$, $\hat{Y}^\ast_t=\hat{Y}_t(\theta)$.
\item\label{item:prop:firama-pred-def7} 
  In
  the well-specified case,  the
  best $\fiarmacol$-predictor  $\hat{Y}^\ast_t$ is always well defined
  and coincides with the best
  predictor in  $\cH_{t-1}^Y$, that is,
  \begin{align}
    \label{eq:fiarmacol-pred-best-error-well-specified}
&    \mathbb{E}^2\lr{Y,\fiarmacol}=\inf_{V\in\cH^Y_{t-1}}\PEarg{\norm{Y_t-V}_{\cH_0}^2}\;,\\
      \label{eq:fiarmacol-pred-best-error-well-specified2}
& \hat{Y}_t^\ast= \proj{Y_t}{\cH^Y_{t-1}}\;.
  \end{align}
  \end{enumerate}
  \end{proposition}
The next results shows how to estimate a predictor which converges to
the best predictor that we have just introduced. We now introduce our
estimation procedure.

Using $\fop^{\dagger}_{\theta}$ defined
by~(\ref{eq:fiarmapred-def2}) and the periodogram  $I_n^{X^c_n}$
defined in \Cref{sec:est-main-assumpt-main}, we consider a sequence of
estimators $(\hat\theta_n)_{n\in\nset}$ satisfying
\begin{align}
  \label{eq:theta-hat-def}
\limsup_{n\to\infty}  \lr{\Lambda_n(\hat{\theta}_n) - \inf_{\theta\in\Theta}
    \Lambda_n(\theta)} = 0\;,
\end{align}
where, for all $n\in\nset$ and $\theta\in\Theta$,
\begin{align}
\label{eq:theta-hat-contrast-def}
\Lambda_n(\theta) :=\tr\lr{\tilde{Q}_{I_n^{X^c_n}}^{(\fop^\dagger,
    \fop^\dagger)}(\theta)} = 
    \tr\lr{\int \fop^{\dagger}(\theta,\lambda)\,I_n^{X^c_n}(\lambda)\,\lr{\fop^{\dagger}(\theta,\lambda)}^\adjoint\;\rmd\lebtore}\;.
\end{align}
Let $Y$ be the centered version of $X$, $Y=X-\PEarg{X_0}$.  Using that
$I_n^{X^c_n}$ approximates $\nu_X=\nu_Y$,
with~(\ref{eq:fiarmapred-def2}) and~(\ref{eq:fiarmacol-pred}),
$\Lambda_n(\theta)$ can be seen as an approximation of
$\PEarg{\norm{Y_t-\hat{Y}_t(\theta)}_{\cH_0}^2}$, and $\hat\theta_n$
as an attempt to minimize this risk in $\theta$, mimicking what is
done in~(\ref{eq:fiarmacol-pred-best-error}).  Then, to take onto
account the unknown mean of $X$ and since we can only use the
observations $X_1,\dots,X_n$ to predict $X_{n+1}$, we truncate the
series defining the predictor in~(\ref{eq:fiarmacol-pred}) to keep its
$n$ first terms only, apply it to the empirically centered
observations $\lr{X^c_{n,n+1-k}}_{1\leq k\leq n}$ and add the
empirical mean to approximate $\PEarg{X_0}$. This lead us to define
the predictor of $X_{n+1}$ from the sample $X_1,\dots,X_n$ associated
to the estimator $\hat{\theta}_n$ by
\begin{equation}
  \label{eq:our-priedictor-from-sample}
\hat{X}_{n+1,n} = \frac1n\sum_{k=1}^nX_k
+\sum_{k=1}^n\fiarmapredcoef[\hat{\theta}_n]{k}\;X^c_{n,n+1-k} \;,
\end{equation}
where $\fiarmapredcoef[\theta]{k}$ is defined in
\Cref{prop:firama-pred-def}. Note that the predictor $\hat{X}_{n+1,n}$
can be written as
\begin{equation}
  \label{eq:our-priedictor-from-sample-general-form}
m+\sum_{k=1}^n\fiarmapredcoef[\theta]{k}\;\lr{X_{n+1-k}-m}
\end{equation}
by taking $m\in\cH_0$ and $\theta\in\Theta$ equal to
$\frac1n\sum_{k=1}^nX_k$ and $\hat\theta_n$, respectively.
In the following theorem, defining, for all $n\geq1$, $m\in\cH_0$ and
$\theta\in\Theta$, the quadratic prediction risk of a predictor of
this form by
\begin{equation}
  \label{eq:pred-finite-obs-with-emp-mean}
\mathbb{E}^{2}_{X,n}\lr{m,\theta}= \PEarg{\norm{X_{n+1}-\lr{m+\sum_{k=1}^n\fiarmapredcoef[\theta]{k}\;\lr{X_{n+1-k}-m}}}_{\cH_0}^2}  \;,
\end{equation}
we show that, as a predictor of $X_{n+1}$, $\hat{X}_{n+1,n}$
asymptotically achieves the same 
prediction risk as the optimal risk for predicting the centered
process $Y=X-\PEarg{X_0}$ from its past.
\begin{theorem}
  \label{thm:fiarma-predictor-esimation}
  Let $\cH_0$ be an infinite-dimensional separable Hilbert space and
  $p,q$ be two non-negative integers. Let $X=(X_t)_{t\in\zset}$ be a
  process defined on $(\Omega,\cF,\PP)$ and valued in $\cH_0$
  satisfying~\ref{item:ergo-assump} and~\ref{item:sp_density-assump}.
  Let $\fiarmacol$ be a model satisfying~\ref{item:fiarma-mode} with
  compact parameter metric space $\lr{\Theta,\Delta}$. Let
  $s=(s_k)_{k\in\nset}\in[1,\infty)^\nset$ and $(\phi_k)_{k\in\nset}$
  be an orthonormal sequence of $\cH_0$.  Define the Hilbert space
  $\cH_0^s$ by~(\ref{eq:defHOs}) and~(\ref{eq:pscalHOs}). We suppose
  that~\ref{item:s-sob-assump} and~\ref{item:X-sob-assump} of
  \Cref{thm:periodo-func} hold as well as the following condition. 
\begin{enumerate}[label=(\roman*)]
  \newcounter{backcount}
  \setcounter{backcount}{\value{enumi}}
  \setcounter{enumi}{\value{primcount}}
\item\label{item:LR-sob-assump-est} Defining, for all $z\in\cset$,
$\arpol_{\theta,s}(z)=\arpol_{\theta}(z)_{|\cH_0^s}$ and
  $\mapol_{\theta,s}(z)=\mapol_{\theta}(z)_{|\cH_0^s}$, we have that
  $(\theta,\lambda)\mapsto\arpol_{\theta,s}\lr{\rme^{-\rmi\lambda}}$
  and
  $(\theta,\lambda)\mapsto\mapol_{\theta,s}\lr{\rme^{-\rmi\lambda}}$
  belong to $\cD\lr{\Theta\times\tore,\cL_b(\cH_0^s)}$. Under~\ref{item:fiarma-mode}\ref{item:long-memory-fiarma-pred},
  defining $D_{\theta,s}=D_{\theta|\cH_0^s}$, assume, in addition,
  that
  $\theta\mapsto D_{\theta,s}$ belongs to $\cC\lr{\Theta,\cN^\dagger(\cH_0^s)}$.
  \setcounter{enumi}{\thebackcount}
\end{enumerate}
Finally, let $(\hat\theta_n)_{n\in\nset}$ be a sequence of estimators
satisfying~(\ref{eq:theta-hat-def}). Then, we have
\begin{align}
  \label{eq:consitency-not-well-specified}
  \lim_{n\to\infty}\Delta\lr{\hat\theta_n, \Theta^*_Y}
  =0\;,\quad\as\;,\\
  \label{eq:best-fiarmapred-not-well-specified-error}
\lim_{n\to\infty}\mathbb{E}^{2}_{X,n}\lr{\frac1n\sum_{k=1}^nX_k,\hat\theta_n}  =   \mathbb{E}^2\lr{Y,\fiarmacol}\;,\quad\as\;.
\end{align}
where $Y=(Y_t)_{t\in\zset}$ denotes the centered process defined by
$Y_t=X_t-\PEarg{X_t}$, $\Theta^*_Y$ is defined
in~(\ref{eq:fiarmacol-pred-best-set}),
$\mathbb{E}^{2}_{X,n}$
in~(\ref{eq:pred-finite-obs-with-emp-mean}), and
$\mathbb{E}^2\lr{Y,\fiarmacol}$
in~(\ref{eq:fiarmacol-pred-best-error}).

Moreover, if $\hat{Y}^\ast_t$ is well defined (as in
\Cref{prop:firama-pred-def}~\ref{item:prop:firama-pred-def5}), we further have
\begin{align}
  \label{eq:best-fiarmapred-not-well-specified-error-well-spec}
\limsup_{n\to\infty}\PEarg{\norm{X_{n+1}-\hat{X}_{n+1,n}}_{\cH_0}^2}
  \leq   \mathbb{E}^2\lr{Y,\fiarmacol}\;,
\end{align}
where is  $\hat{X}_{n+1,n}$
defined by~(\ref{eq:our-priedictor-from-sample}).
\end{theorem}
Let us briefly comment the conclusions
of \Cref{thm:fiarma-predictor-esimation}. Equation~(\ref{eq:consitency-not-well-specified})
says that $\hat\theta_n$ is consistent for
estimating the optimal $\theta$ up to the equivalence relationship
$\theta\sim\theta'$ defined by
$\PEarg{\norm{Y_t-\hat{Y}_t(\theta)}_{\cH_0}^2}=\PEarg{\norm{Y_t-\hat{Y}_t(\theta')}_{\cH_0}^2}$.
Equation~(\ref{eq:best-fiarmapred-not-well-specified-error}) says that
the rsik of an estimator of the
form~(\ref{eq:our-priedictor-from-sample-general-form}) for predicting
$X_{n+1}$ is asymptotically minimal with $m$ and $\theta$ replaced by
the empirical mean and $\hat\theta_n$. Finally ,
while~(\ref{eq:consitency-not-well-specified}) and
(\ref{eq:best-fiarmapred-not-well-specified-error}) hold in the $\as$
sense,
Equation~(\ref{eq:best-fiarmapred-not-well-specified-error-well-spec})
says that the prediction risk directly defined with the predictor 
$\hat{X}_{n+1,n}$, that is, in contrast
to~(\ref{eq:best-fiarmapred-not-well-specified-error}), with the
empirical mean and $\hat\theta_n$ inside the expectation, is indeed
asymptotically optimal.
 
\section{Postponed proofs}\label{sec:proofs-fiarma}
\subsection{Proofs of \Cref{sec:hilbert-valued-fiarma}}
\subsubsection{Proofs of  \Cref{prop:existence-joint-kernelope-fiarma}
  and \Cref{thm:cns-fi-operator-functional-fiarma}}
\Cref{prop:existence-joint-kernelope-fiarma} is used to introduce the
definition of joint kernels as in \Cref{def:joint-kernels}.
\begin{proof}[\bf Proof of \Cref{prop:existence-joint-kernelope-fiarma}]
  Let $(\phi_i)_{0\leq i<N}$ denote a Hilbert basis of
  $L^2(\Vset,\Vsigma,\xi)$, assumed to be of dimension $N\in\{1,2,\dots,\infty\}$.
  Define $\kernelope{K}_n : (v,v';\lambda) \mapsto \sum_{0\leq i,j\leq n}\phi_i^\adjoint
  K(\lambda)\phi_{j}\;\phi_i(v)\bar\phi_{j}(v')$ on  $\Vset^2 \times
  \tore$ and, for all $\epsilon>0$, $N_\epsilon : \lambda \mapsto
  \inf\set{n<N}{\sum_{i\text{ or }j> n}\left|\phi_i^\adjoint
      K(\lambda)\phi_{j}\right|^2\leq\epsilon}$ on $\tore$. Note that,
  for all $\lambda \in\tore$, $N_\epsilon(\lambda)$ is well defined
  and finite since $\sum_{0\leq i,j<N}\left|\phi_i^\adjoint K(\lambda)\phi_{j}\right|^2=\norm{K(\lambda)}_2<\infty$. Now let us define,  for all $v,v'\in\Vset$ and $\lambda\in\tore$, 
$
    \kernelope{K}(v,v';\lambda):=\lim_{n\to\infty}   \kernelope{K}_{N_{2^{-n}}(\lambda)}(v,v';\lambda)
 $
whenever this limit exists in $\cset$ and set
$\kernelope{K}(v,v';\lambda)=0$ otherwise. 
Since
$(\phi_k\otimes\bar\phi_{k'})_{0\leq k,k'<N}$ is a Hilbert basis of
$L^2(\Vset^2, \Vsigma^{\otimes 2}, \xi^{\otimes 2})$, we immediately
have that, for any $\lambda\in\Lambda$,
$\kernelope{K}_{N_{2^{-n}}(\lambda)}(\cdot,\cdot;\lambda)$ converges in the
sense of this $L^2$ space to $\sum_{0\leq i,j<N}\phi_i^\adjoint
    K(\lambda)\phi_{j}\;\phi_i\otimes\bar\phi_{j}$, and thus, this limit
    must be equal to 
$\kernelope{K}(\cdot,\cdot;\lambda)$, $\xi^{\otimes2}$-$\mae$ It follows that, that for any
$\lambda\in\Lambda$, for all $i,j\in\nset$,
$
\int \kernelope{K}(v,v';\lambda)\bar\phi_i(v)\phi_{j}(v')\;\xi(\rmd v)\xi(\rmd
v')= \phi_i^\adjoint K(\lambda)\phi_{j}\;,
$
which gives that $K(\lambda)$ is an integral operator associated to
the kernel $\kernelope{K}(\cdot,\cdot;\lambda)$. Since
$(v,v',\lambda)\mapsto\kernelope{K}(v,v';\lambda)$ is measurable by
definition, this concludes the proof of the existence of the $\Lambda$-joint kernel of $K$. If, moreover, $K\in\cL^2(\Lambda,\calA,\cS_2(\cH_0),\mu)$, then $\kernelope{K}_n$ converges in $L^2(\Vset^2\times\Lambda,\Vsigma^{\otimes2}\otimes\calA,\xi^{\otimes2}\otimes\mu)$ and the limit must be equal to $\kernelope{K}$,
$\xi^{\otimes2}\otimes\mu\text{-}\mae$ since for each $\lambda\in\Lambda$,
$\kernelope{K}_n(\cdot,\cdot;\lambda)$ converges to
$\kernelope{K}(\cdot,\cdot;\lambda)$ in
$L^2(\Vset^2,\Vsigma^{\otimes2},\xi^{\otimes2})$. Hence, we get that $\kernelope{K} \in L^2(\Vset^2\times\Lambda,\Vsigma^{\otimes2}\otimes\calA,\xi^{\otimes2}\otimes\mu)$. 
\end{proof}
We now prove  \Cref{thm:cns-fi-operator-functional-fiarma}. 
\begin{proof}[\bf Proof of
  \Cref{thm:cns-fi-operator-functional-fiarma}]
  We assume without loss of generality that $\mu(\{0\})=0$ (since it
  affects none of the given assertions).
  By \Cref{prop:carac-processtransfer-discrete}, Assertion~\ref{item:cns-fi-operator-functional1-fiarma} is equivalent to
   \begin{align}\label{eq:X-in-S-fiarma}
     \int_{\tore} \norm{(1 - \rme^{-\rmi \lambda})^{-D} g_X(\lambda) \left[(1 - \rme^{-\rmi \lambda})^{-D}\right]^\adjoint}_1 \, \mu(\rmd \lambda) < \infty \; .
   \end{align}
   Using the singular values decomposition \eqref{eq:SVD-normal-op-fiarma}, and since $U$ is unitary from $\cH_0$ to $L^2(\Vset,\Vsigma,\xi)$, we get that, for all $\lambda\in\tore\setminus\{0\}$,  $\norm{(1 - \rme^{-\rmi \lambda})^{-D} g_X(\lambda) \left[(1 - \rme^{-\rmi \lambda})^{-D}\right]^\adjoint}_1 =   \norm{U^\adjoint M_{(1 - \rme^{-\rmi \lambda})^{-d}} U g_X(\lambda) \U^\adjoint M_{(1 - \rme^{-\rmi \lambda})^{-d}}^\adjoint U}_1  = \norm{M_{(1 - \rme^{-\rmi \lambda})^{-\mathrm{d}}} h(\lambda)}_2^2$.
Hence~(\ref{eq:X-in-S-fiarma}) holds if and only if 
$\displaystyle\int_{\tore} \norm{M_{(1 - \rme^{-\rmi \lambda})^{-\mathrm{d}}} h(\lambda)}_2^2 \, \mu(\rmd \lambda) < \infty$, which, using the $\tore$-joint kernel $\kernelope{h}$ of $h$, reads
$\displaystyle
\int \left|(1 -
\rme^{-\rmi \lambda})^{-\mathrm{d}(v)} \kernelope{h}(v,v';\lambda)\right|^2\;\xi(\rmd
v)\xi(\rmd v') \, \mu(\rmd \lambda) < \infty
$.
Applying~\Cref{lem:integral-1minusexp-fiarma} to $z =-\mathrm{d}(v)$, since $\mathrm{d}$ is a $\mu$-essentially
bounded function, we get that
Assertion~\ref{item:cns-fi-operator-functional1-fiarma} is equivalent to
\begin{equation}
  \label{eq:int-xiximu-small-freq}
\int_{\Vset^2\times(-\pi,\pi]}
|\lambda|^{-2\Re(\mathrm{d}(v))} \left|\kernelope{h}(v,v';\lambda)\right|^2\;\xi(\rmd
v)\xi(\rmd v')\mu(\rmd\lambda)<\infty\;.  
\end{equation}
Using that $|\lambda|^{2\Re_{-}(\mathrm{d}(v))}$ is
  bounded over $\lambda\in(-\pi,\pi]$ and that
$$
\int \left|\kernelope{h}(v,v';\lambda)\right|^2\;\xi(\rmd v)\xi(\rmd
v')\mu(\rmd\lambda)=\int\norm{h(\lambda)}_2^2\;\mu(\rmd\lambda)<\infty\;,
$$
Condition~(\ref{eq:int-xiximu-small-freq}) is equivalent to
Assertion~\ref{item:cns-fi-operator-functional3-fiarma}. Using that,
for any $\eta\in(0,\pi)$, $|\lambda|^{-2\Re(\mathrm{d}(v))}$ is bounded
independently of $v$ on $\lambda\in(-\pi,\pi]\setminus(-\eta,\eta)$,
we also get that Condition~(\ref{eq:int-xiximu-small-freq}) is equivalent to
Assertion~\ref{item:cns-fi-operator-functional2-fiarma}. 
\end{proof}
\subsubsection{Proof of \Cref{cor:cns-arma-fi-operator-functional-fiarma}}
We start with a useful result on ARMA processes. 
\begin{lemma}\label{prop:arma-processes-kernel-fiarma}
  Let $\cH_0$ be a separable Hilbert space and $X$ be an ARMA($p,q$) process
  defined by
  $\hat
  X(\rmd\lambda)=[\arpol(\rme^{-\rmi\lambda})]^{-1}\mapol(\rme^{-\rmi\lambda})\hat
  Z(\rmd\lambda)$ with $\mapol\in\mathcal{P}_q(\cH_0)$,
  $\arpol\in\mathcal{P}_p^\ast(\cH_0)$ and $Z$ an $\cH_0$-valued white noise with
  covariance operator $\Sigma$. 
  Then, there exists $\eta\in(0,\pi)$
  such that
\begin{equation}
  \label{eq:kernel_arma_around_zero-new-fiarma-geom-decreasing-k_n}
  \sum_{n\in\nset}\frac{\eta^n}{n!}\,\norm{\aop_n\lr{\arpol,\mapol}}_{\infty}<\infty\;,
\end{equation}
where $\aop_n\lr{\arpol,\mapol}$ is defined in~(\ref{eq:kernel_arma_around_zero-new-fiarma-defk_n}).
Moreover, for $\leb\text{-}\mae$ $\lambda\in(-\eta,\eta)$, we have
  \begin{equation}
    \label{eq:kernel_arma_around_zero-new-fiarma}
    g_{X}(\lambda)=h(\lambda)\,[h(\lambda)]^\adjoint\quad\text{with}\quad
    h(\lambda)=\sum_{n\in\nset}\frac{(-\rmi\lambda)^n}{n!}\,\aop_n\lr{\arpol,\mapol}\,\Sigma^{1/2}\;,
  \end{equation}
where $h$ can be seen as a power series valued in $\cS_2(\cH_0)$ with a
convergence radius at least equal to $\eta$. 
\end{lemma}
\begin{proof}
  Since $z\mapsto[\arpol(z)]^{-1}\,\mapol(z)$ is
  holomorphic in an open ring containing $\unitcircle$ and the
  exponential function is holomorphic on $\cset$, by
  \cite[Theorem~1.8.5]{gohberg-leiterer09}, there exists $\eta>0$ such
  that~(\ref{eq:kernel_arma_around_zero-new-fiarma-geom-decreasing-k_n})
  holds and 
  $[\arpol(\rme^{z})]^{-1}\,\mapol(\rme^z)$ coincides with the $\cL_b(\cH_0)$-valued power
  series $\sum_{n=0}^\infty\aop_n\lr{\arpol,\mapol} z^n/n!$ on the set
  $\set{z\in\cset}{|z|\leq\eta}$. 
  
  Finally, we observe that $g_X=hh^H$ with
  $h(\lambda)=[\arpol(\rme^{-\rmi\lambda})]^{-1}\,\mapol(\rme^{-\rmi\lambda})\,\Sigma^{1/2}$
  and the given expression of $h$
  in~(\ref{eq:kernel_arma_around_zero-new-fiarma}) follows
  from~(\ref{eq:kernel_arma_around_zero-new-fiarma-geom-decreasing-k_n})
  and the usual bound 
$$
\norm{\aop_n\lr{\arpol,\mapol}\,\Sigma^{1/2}}_2\leq\norm{\aop_n\lr{\arpol,\mapol}}_{\infty}\,\norm{\Sigma^{1/2}}_2=\norm{\aop_n\lr{\arpol,\mapol}}_{\infty}\,\norm{\Sigma}_1^{1/2}\;.
$$
This concludes the proof. 
\end{proof}
\begin{proof}[\bf Proof of  \Cref{cor:cns-arma-fi-operator-functional-fiarma}]
  Before proving the claimed implications, we start with some
  preliminary facts that follow from
  \Cref{prop:arma-processes-kernel-fiarma},
  \Cref{lem:kernel-cov-fiarma} and
  \Cref{thm:cns-fi-operator-functional-fiarma}.  First observe that
  the process $UX=(UX_t)_{t\in\zset}$ is the $\cG_0$-valued
  ARMA($p,q$) process defined by
  $\widehat
  {UX}(\rmd\lambda)=[\tilde{\arpol}(\rme^{-\rmi\lambda})]^{-1}\,\tilde\mapol(\rme^{-\rmi\lambda})\,\widehat{UZ}(\rmd\lambda)$,
  where $\tilde{\mapol}:=U\mapol U^{\adjoint}\in\mathcal{P}_q(\cG_0)$
  and
  $\tilde{\arpol}:=U\arpol U^{\adjoint}\in\mathcal{P}_p^\ast(\cG_0)$,
  and $UZ=(UZ_t)_{t\in\zset}$ is a $\cG_0$-valued white noise.  Then,
  applying \Cref{prop:arma-processes-kernel-fiarma} with $\mu$ as the
  Lebesgue measure, we get that, for some $\eta>0$, $\nu_{UX}$ has
  density $h(\lambda)[h(\lambda)]^\adjoint$ on $(-\eta,\eta)$ with $h$
  a power series valued in $\cS_2(\cG_0)$ with radius of convergence
  at least $\eta>0$,
  \begin{equation}\label{eq:h-proof-fiarma-cns--fiarma}
    h(\lambda)=\sum_{n\in\nset}\frac{(-\rmi\lambda)^n}{n!}\,U\,\aop_n\lr{\arpol,\mapol}\,\Sigma^{1/2}\,U^{\adjoint}\;.
  \end{equation}
  Now, define, for any $\eta'\in(0,\eta)$,
$$
I(\eta'):=\int_{\Vset^2\times(-\eta',\eta')}
|\lambda|^{-2\Re(\mathrm{d}(v))} \abs{\kernelope{h}(v,v';\lambda)}^2\;\xi(\rmd
v)\xi(\rmd v')\frac{\rmd\lambda}{2\pi} \;,
$$
where $\kernelope{h}$ is the $\tore$-joint kernel of $h$
in~(\ref{eq:h-proof-fiarma-cns--fiarma}). By
\Cref{thm:cns-fi-operator-functional-fiarma},
Assertion~\ref{ass:cns-arma-fi-operator-functional-cond0-fiarma} holds
if,
and only if, there exists $\eta'\in(0,\eta)$ such that
$I(\eta')<\infty$ which is itself equivalent to having
$I(\eta')<\infty$ for all $\eta'\in(0,\eta)$. Using
\eqref{eq:h-proof-fiarma-cns--fiarma}, we have
\begin{equation}
  \label{eq:Ieta-priem-def}
I(\eta')
=\int_{\Vset^2\times(-\eta',\eta')}
|\lambda|^{-2\Re(\mathrm{d}(v))} \abs{\sum_{n\in\nset}\frac{(-\rmi\lambda)^n}{n!}\,\kernelope{k}_n(v,v')}^2\;\xi(\rmd
v)\xi(\rmd v')\frac{\rmd\lambda}{2\pi}\;,  
\end{equation}
where $\kernelope{k}_n$ denotes the kernel of
$U\,\aop_n\lr{\arpol,\mapol}\,\Sigma^{1/2}\,U^{\adjoint}\in\cS_2(\cG_0)$. In
particular we have by \Cref{lem:kernel-cov-fiarma} that
\begin{equation}
  \label{eq:sigma_n-k_n-relation}
\sigma_n(v)= \varsqrt{\PEarg{\abs{W_n(v,\cdot)}^2}}=\norm{\kernelope{k}_n(v,\cdot)}_{\cG_0}\;.
\end{equation}
Denote
\begin{align}
  \label{eq:I_n-eta-priem-def}
I_n(\eta')&:=\int_{\Vset^2\times(-\eta',\eta')}
|\lambda|^{2n-2\Re(\mathrm{d}(v))}\,\abs{\kernelope{k}_n(v,v')}^2\;\xi(\rmd
v)\xi(\rmd v')\frac{\rmd\lambda}{2\pi}\;,\\
  \nonumber
  \overline{\rmd}&=\sup(\Re(\rmd))\quad\text{and}\quad \overline{m}:=\inf\set{m\in\nset}{m>\overline{\rmd}-1/2}\;.
\end{align}
Note that $\overline{\rmd}$ and $\overline{m}$ are finite since $\rmd$ is bounded. 
Defining, moreover,
\begin{align*}
&  \overline{I}(\eta'):=  \int_{\Vset^2\times(-\eta',\eta')}
|\lambda|^{-2\Re(\mathrm{d}(v))} \abs{\sum_{0\leq
                 n<\overline{m}}\frac{(-\rmi\lambda)^n}{n!}\,\kernelope{k}_n(v,v')}^2\;\xi(\rmd
v)\xi(\rmd v')\frac{\rmd\lambda}{2\pi}\\
&\text{and}\quad            R(\eta'):=
            \sum_{n=\overline{m}}^\infty\frac{I_n^{1/2}(\eta')}{n!}\;,
\end{align*}
and applying the Minkowski
inequality in~(\ref{eq:Ieta-priem-def}), for any $\eta'\in(0,\eta)$,
we have that, if  $R(\eta')<\infty$,
$$
I(\eta')<\infty\Leftrightarrow   \overline{I}(\eta')<\infty \;.
$$
Let us pick
$\eta'\in(0,1\wedge\eta)$.
Then, for all
$\lambda\in(-\eta',\eta')$ and $n\in\nset$, we have
$|\lambda|^{2n-2\Re(\mathrm{d}(v))}\leq|\lambda|^{2n-2\overline{\rmd}}$
and thus, for all $n\geq\overline{m}$,
\begin{align*}
I_n(\eta')&\leq \frac{\eta^{\prime(1+2n-2\overline{\rmd})}}{\pi(1+2n-2\overline{\rmd})}
\int_{\Vset^2}\abs{\kernelope{k}_n(v,v')}^2\;\xi(\rmd
  v)\xi(\rmd v')\\
  &=
    \frac{\eta^{\prime(1+2n-2\overline{\rmd})}}{\pi(1+2n-2\overline{\rmd})}
    \norm{\aop_n\lr{\arpol,\mapol}\,\Sigma^{1/2}}_{2}^2\;.
\end{align*}
Using that
$\norm{\aop_n\lr{\arpol,\mapol}\,\Sigma^{1/2}}_{2}\leq\norm{\aop_n\lr{\arpol,\mapol}}_{\infty}\norm{\Sigma}_{1}^{1/2}$
and the bound~(\ref{eq:kernel_arma_around_zero-new-fiarma-geom-decreasing-k_n})
of \Cref{prop:arma-processes-kernel-fiarma}, we get that $R(\eta')<\infty$.
We thus conclude that
Assertion~\ref{ass:cns-arma-fi-operator-functional-cond0-fiarma} is
equivalent to have, for some $\eta'\in(0,1\wedge\eta)$,
\begin{equation}
  \label{eq:Ieta-priem-cond-equiv}
\overline{I}(\eta') < \infty \;.
\end{equation}
Next, we show that, for any $\eta'\in(0,1)$,
Condition~\ref{ass:cns-arma-fi-operator-functional-cond1-fiarma} is, in fact, equivalent to have
\begin{equation}
  \label{eq:cns-arma-fi-operator-functional-cond1-fiarma-equiv}
I_n(\eta')<\infty\quad\text{for all}\quad n\in\nset  \;.
\end{equation}
Indeed, integrating w.r.t. $v'$ and $\lambda$ in the definition of
$I_n(\eta')$ in~(\ref{eq:I_n-eta-priem-def}), we have that, for all $n\in\nset$,
$$
I_n(\eta')
=\int_{\lrcb{\Re(\mathrm{d})<n+1/2}}
\frac{\eta^{\prime\,(1+2n-2\Re(\mathrm{d}(v)))}\,\sigma_n^2(v)}{1+2n-2\Re(\mathrm{d}(v))}\;\frac{\xi(\rmd
v)}{\pi}\;,
$$
if $\Re(\mathrm{d})<n+1/2$ $\xi$-a.e. on $\lrcb{\sigma_n>0}$, or
is equal to $\infty$ otherwise. Since $\mathrm{d}$ is bounded on
$\Vset$, so is $\eta^{\prime\,(1+2n-2\Re(\mathrm{d}(v)))}$ on
$v\in\Vset$ and we conclude that
Condition~\ref{ass:cns-arma-fi-operator-functional-cond1-fiarma} is equivalent to~(\ref{eq:cns-arma-fi-operator-functional-cond1-fiarma-equiv}).

We are now ready to prove each implication of the claimed equivalence
successively.

\noindent\textbf{Proof of \ref{ass:cns-arma-fi-operator-functional-cond0-fiarma}$\Leftarrow$\ref{ass:cns-arma-fi-operator-functional-cond1-fiarma}.}
This is now trivial, since applying the Minkowski inequality
in the integral defining $\overline{I}$
in~(\ref{eq:Ieta-priem-cond-equiv}) and the definition of $I_n$
in~(\ref{eq:I_n-eta-priem-def}), we immediately see that
Condition~(\ref{eq:Ieta-priem-cond-equiv}) is implied
by~(\ref{eq:cns-arma-fi-operator-functional-cond1-fiarma-equiv}).

\noindent\textbf{Proof of \ref{ass:cns-arma-fi-operator-functional-cond0-fiarma}$\Rightarrow$\ref{ass:cns-arma-fi-operator-functional-cond1-fiarma}.}
  The proof of this implication is a bit more complex. The first step is
  to prove that
  Assertion~\ref{ass:cns-arma-fi-operator-functional-cond0-fiarma}
  implies
  \begin{equation}
    \label{eq:cns-fiarma-existence-1}
\text{for all}\quad n\in\nset,\;\Re(\mathrm{d}) < n+1/2\quad \xi\text{-}\mae\quad\text{on}\quad\{\sigma_n >
    0\}\;.
  \end{equation}
Then, we show that it must also imply ~(\ref{eq:cns-arma-sigma_n}) for all $n\in\nset$
in a second and last step. 

\noindent\textbf{Step~1.} Suppose
that~(\ref{eq:cns-fiarma-existence-1}) does not hold; let us show
that
Assertion~\ref{ass:cns-arma-fi-operator-functional-cond0-fiarma}
cannot hold. Since it is equivalent
to~(\ref{eq:Ieta-priem-cond-equiv}), it is sufficient to show that
$\overline{I}(\eta')=\infty$ for any arbitrary $\eta'>0$.
Let $\underline{m}$ be the smallest $n\in\nset$ for which
$\xi\lr{\lrcb{\Re(\mathrm{d}) \geq n+1/2}\cap\lrcb{\sigma_n >0}}>0$. Note that by their mere definitions, we have  $\underline{m}<\overline{m}$.
In addition, for all $0\leq  n <\underline{m}$, we have
$$
\xi \lr{ \lrcb{\Re(\mathrm{d}) \geq \underline{m}+1/2}\cap\lrcb{\sigma_n
    > 0} } \leq \xi(\lrcb{\Re(\mathrm{d}) \geq n+1/2}\cap\lrcb{\sigma_n >   0}) = 0 \;.
$$
Hence, $\kernelope{k}_n(v,v')=0$ for $\xi^{\otimes2}$-$\mae$ $(v,v')\in\lrcb{\Re(\mathrm{d}) \geq \underline{m}+1/2}\times\Vset$.
Now we have, using the definition of $\overline{I}(\eta')$
in~(\ref{eq:Ieta-priem-cond-equiv}) and what we just deduced,
\begin{align*}
\overline{I}(\eta') & \geq  \int_{\lrcb{\Re(\mathrm{d}) \geq \underline{m}+1/2}\times\Vset\times(-\eta',\eta')}
|\lambda|^{-2\Re(\mathrm{d}(v))} \abs{\sum_{0\leq n<\overline{m}}\frac{(-\rmi\lambda)^n}{n!}\,\kernelope{k}_n(v,v')}^2\;\xi(\rmd
                      v)\xi(\rmd v')\frac{\rmd\lambda}{2\pi}  \\
  & =  \int_{\lrcb{\Re(\mathrm{d}) \geq \underline{m}+1/2}\times\Vset\times(-\eta',\eta')}
|\lambda|^{-2\Re(\mathrm{d}(v))} \abs{\sum_{\underline{m} \leq n<\overline{m}}\frac{(-\rmi\lambda)^n}{n!}\,\kernelope{k}_n(v,v')}^2\;\xi(\rmd
v)\xi(\rmd v')\frac{\rmd\lambda}{2\pi} \;. 
\end{align*}
Note that, for all $(v,v')\in\Vset^2$ and $\lambda\in\rset$, we have
$$
|\lambda|^{-2\Re(\mathrm{d}(v))} \abs{\sum_{\underline{m} \leq
    n<\overline{m}}\frac{(-\rmi\lambda)^n}{n!}\,\kernelope{k}_n(v,v')}^2=|\lambda|^{2\underline{m}-2\Re(\mathrm{d}(v))}
\lr{\frac{\abs{\kernelope{k}_{\underline{m}}(v,v')}^2}{(\underline{m}!)^2}+o(1)}\;,
$$
where $o$-term tends to 0 as $\lambda\to0$. We get that, for all $(v,v')\in\lr{\lrcb{\Re(\mathrm{d}) \geq
  \underline{m}+1/2}\times\Vset}\cap \lrcb{\abs{\kernelope{k}_{\underline{m}}}^2>0}$
and $\eta'>0$,
$$
\int_{(-\eta',\eta')}
|\lambda|^{-2\Re(\mathrm{d}(v))} \abs{\sum_{\underline{m} \leq
    n<\overline{m}}\frac{(-\rmi\lambda)^n}{n!}\,\kernelope{k}_n(v,v')}^2\;\frac{\rmd\lambda}{2\pi} =\infty\;.
$$
With the previous lower bound on $\overline{I}(\eta')$, we deduce that
$\overline{I}(\eta')=\infty$ if
$$
\xi^{\otimes 2}\lr{\lr{\lrcb{\Re(\mathrm{d}) \geq
  \underline{m}+1/2}\times\Vset}\cap \lrcb{\abs{\kernelope{k}_{\underline{m}}}^2>0} }
>0 \;.
$$
By definition of $\sigma_n$ in~(\ref{eq:sigma_n-k_n-relation}), we
have, for all $v\in\Vset$,
$$
g(v):=\int_{\Vset}
\1_{\lrcb{\abs{\kernelope{k}_{\underline{m}}(v,v')}^2>0}}\;\xi(\rmd v')
>0\Leftrightarrow \sigma_{\underline{m}}(v)>0 \;.
$$
Hence,
$$
\xi^{\otimes 2}\lr{\lr{\lrcb{\Re(\mathrm{d}) \geq
  \underline{m}+1/2}\times\Vset}\cap \lrcb{\abs{\kernelope{k}_{\underline{m}}}^2>0}
} = \int_{\lrcb{\Re(\mathrm{d}) \geq
  \underline{m}+1/2}} g \;\rmd\xi 
$$
is positive if and only if $\xi\lr{ \lrcb{\Re(\mathrm{d}) \geq
  \underline{m}+1/2} \cap \lrcb{\sigma_{\underline{m}}>0} } >0$, which
is true by definition of $\underline{m}$. This concludes the first
step.

\noindent\textbf{Step~2.} Suppose now
that~(\ref{eq:cns-fiarma-existence-1}) does hold
but~(\ref{eq:cns-arma-sigma_n}) does not hold for all $n\in\nset$ and
let us show again that
Assertion~\ref{ass:cns-arma-fi-operator-functional-cond0-fiarma}
cannot hold.  Let us define $\tilde{m}$ as the smallest $n\in\nset$
such that~(\ref{eq:cns-arma-sigma_n}) does not hold. Again by definition of $\overline{m}$, we must
have $\tilde{m}<\overline{m}$, since~(\ref{eq:cns-arma-sigma_n}) holds
for $n=\overline{m}$ by definition of $\overline{m}$.  Take now an
arbitrary $\eta'\in(0,1\wedge\eta)$. We have shown in the preamble of
the proof that if~(\ref{eq:cns-fiarma-existence-1}) is satisfied,
then~(\ref{eq:cns-arma-sigma_n}) is equivalent to
$I_{n}(\eta')<\infty$. Hence, we can also see  $\tilde{m}$ as the smallest $n\in\nset$
such that $I_{n}(\eta')=\infty$. Thus, we have
$I_k(\eta')<\infty$ for all $0\leq k <\tilde{m}$ and $I_{\tilde{m}}(\eta')=\infty$. It follows that
Assertion~\ref{ass:cns-arma-fi-operator-functional-cond0-fiarma} is
not only equivalent to having $\overline{I}(\eta')<\infty$ as
in~(\ref{eq:Ieta-priem-cond-equiv}) but also to the condition
\begin{equation}
  \label{eq:Ieta-priem-cond-equiv-bis}
\widetilde{I}(\eta'):=  \int_{\Vset^2\times(-\eta',\eta')}
|\lambda|^{-2\Re(\mathrm{d}(v))} \abs{\sum_{\tilde{m}\leq n<\overline{m}}\frac{(-\rmi\lambda)^n}{n!}\,\kernelope{k}_n(v,v')}^2\;\xi(\rmd
v)\xi(\rmd v')\frac{\rmd\lambda}{2\pi} < \infty \;.
\end{equation}
Now, we observe that
\begin{align*}
  \widetilde{I}(\eta')& \geq
                        \int_{\lrcb{\Re(\mathrm{d}) < \tilde{m}+1/2}\times\Vset\times(-\eta',\eta')}
|\lambda|^{-2\Re(\mathrm{d}(v))} \abs{\sum_{\tilde{m}\leq n<\overline{m}}\frac{(-\rmi\lambda)^n}{n!}\,\kernelope{k}_n(v,v')}^2\;\xi(\rmd
v)\xi(\rmd v')\frac{\rmd\lambda}{2\pi} \;.
\end{align*}
Therefore, to conclude that $\widetilde{I}(\eta')=\infty$ (implying
that Assertion~\ref{ass:cns-arma-fi-operator-functional-cond0-fiarma}
does not hold),  by the
Minkowski inequality, it is sufficient to show that
\begin{align}\label{eq:fiarma-ex-step2-cond-finale}
\tilde{I}_{\tilde{m}}(\eta') =\infty
\quad\text{and, for all $n>\tilde{m}$,}\quad
\tilde{I}_{n}(\eta')  <\infty\;,
\end{align}
where, for all $n\in\nset$, we denoted
$$
\tilde{I}_{n}(\eta'):=
\int_{\lrcb{\Re(\mathrm{d}) < \tilde{m}+1/2}\times\Vset\times(-\eta',\eta')}
|\lambda|^{2n-2\Re(\mathrm{d}(v))}\,\abs{\kernelope{k}_n(v,v')}^2\;\xi(\rmd
v)\xi(\rmd v')\frac{\rmd\lambda}{2\pi} \;.
$$
For all $n\geq \tilde{m}$ we have, as in the previous computation of
$I_n$ that
$$
\tilde{I}_{n}(\eta')<\infty\Leftrightarrow\int_{\lrcb{\Re(\mathrm{d}) < \tilde{m}+1/2}}
\frac{\sigma_n^2(v)}{1+2n-2\Re(\mathrm{d}(v))}\;\xi(\rmd v)<\infty \;.
$$
For an integer $n>\tilde{m}$, we have $1+2n-2\Re(\mathrm{d}(v))\geq 2$
on $\lrcb{\Re(\mathrm{d}) < \tilde{m}+1/2}$, hence the right-hand
side of~(\ref{eq:fiarma-ex-step2-cond-finale}) follows as a consequence
of $\int\sigma_n^2\;\rmd\xi<\infty$. For $n=\tilde{m}$, the left-hand
side of~(\ref{eq:fiarma-ex-step2-cond-finale}) follows as a consequence
of~(\ref{eq:cns-arma-sigma_n}) not being satisfied for $n=\tilde{m}$
by definition of $\tilde{m}$.
\end{proof}
\subsection{Proofs of \Cref{sec:other-long-memory-fiarma}}
\begin{proof}[\bf Proof of \Cref{lem:duker-fiarma-fiarma}]
  Since $\epsilon$ is a white noise, as explained in
  \Cref{ex:rem-existence-fiarma-particular-cases}~\ref{item:WN-case},
  Assertion~\ref{ass:cns-arma-fi-operator-functional-cond1-fiarma} of
  \Cref{cor:cns-arma-fi-operator-functional-fiarma} only needs to be
  checked for $n=0$. The result follows since this case precisely
  corresponds to the conditions in~(\ref{eq:ass-LRD-fiarma}) with
  $D=\Id_{\cH_0}-N$.
\end{proof}
The proof of \Cref{prop:longmemory-fiarma} relies on the following
lemma where we recall that the open and closed complex unit discs of
$\cset$ are denoted by $\unitdisk$ and $\overline{\unitdisk}$,
respectively. This lemma will also be useful in the proofs of
\Cref{sec:examples-est-application}.
\begin{lemma}\label{lem:dev-1minusz-fiarma}
  Let $\cH_0$ be a separable Hilbert space and $N$ be in $\cN(\cH_0)$. Let $\varrho$ and $\varsigma$ such that
  \begin{equation}
    \label{eq:Ncond-fiarma-infsup}
\varrho\leq\inf\set{\pscal{\Re(N)\,x}{x}_{\cH_0}}{x\in\cH_0\,,\;\norm{x}_{\cH_0}=1}\leq\norm{N}_{\infty}\leq \varsigma\;,
\end{equation}
where $\Re(N)=(N^\adjoint+N)/2$.
  Then, there exist $\bop \in \cL_b(\cH_0)$ and
  $(\aop^{\ast}_k)_{k \in \nset} \in \cL_b(\cH_0)^\nset$ such that, for all $z \in \unitdisk$,
  \begin{align}
    \label{eq:dev-1minusz-fiarma}
      (1 - z)^{N-\Id} = \bop \left( \sum_{k= 0}^\infty (k+1)^{-N} z^k
      \right) + \sum_{k= 0}^\infty \aop^{\ast}_k z^k \;,
  \end{align}
  where the two infinite sums on the right-hand side are
  $\cL_b(\cH_0)$-valued power series with a convergence radius at least
  equal to 1.

  There further exist  $C_\varsigma, C_\varrho>0$ only depending on
  $\varsigma$ and $\varrho$ respectively such that 
    \begin{align}
\label{eq:N-fiarma-boundDelta_bop}
    &    \norm{\bop}_\infty\leq C_\varsigma\quad\text{and, for all
      $k\geq0$,}\quad    \norm{\aop^{\ast}_k}_\infty\leq C_\varsigma C_\varrho\,(k+1)^{-1-\varrho}\;.
    \end{align}
  Moreover, if $\varrho>0$, then
  (\ref{eq:dev-1minusz-fiarma}) continues to hold for all
  $z \in \overline{\unitdisk} \setminus \{1\}$ with the two infinite
  sums converging in $\cL_b(\cH_0)$.
\end{lemma}
\begin{proof}
  In the following, we denote by $\mathrm{n}$ the singular value
  function of $N$, defined on
  $\cG_0:=L^2(\Vset, \Vsigma, \xi)$ with decomposition operator
  $U$. Then, Condition~(\ref{eq:Ncond-fiarma-infsup}) car be rewritten
  as
    \begin{equation}
    \label{eq:Ncond-fiarma-infsup-bis}
\varrho\leq\essinfarg{\xi}_{v \in \Vset} \Re(\mathrm{n}(v))
\leq\essuparg{\xi}_{v \in \Vset} \lrav{\mathrm{n}(v)} \leq \varsigma\;.
  \end{equation}
  We now proceed in three steps.  We first show Relation
  \eqref{eq:dev-1minusz-fiarma} for all $z \in \unitdisk$, then that
  the bounds~(\ref{eq:N-fiarma-boundDelta_bop}) hold, and, finally, we
  extend \eqref{eq:dev-1minusz-fiarma} to
  $z \in \overline{\unitdisk} \setminus \{1\}$ when $\varrho > 0$.

  \noindent{\bf Step~1.} Let $z \in \unitdisk$, then
      \begin{equation}
        \label{eq:1-z-N-first-expansion}
  (1 - z)^{N-\Id} = \Id + \sum_{k \geq 1} N_k z^k \quad \text{with} \quad N_k = \prod_{j=1}^k \left( \Id - \frac{N}{j} \right)\;,\;\; \text{for all $k \geq 1$} \;.
\end{equation}
Define the integer $k_0 \geq 1$ by the condition $\varsigma< k_0\leq\varsigma+1$. Then, for
  all $j\geq k_0$,
  $\Id - \frac{N}{j} = \exp \left( \ln \left( \Id - \frac{N}{j}
    \right) \right) = \exp \left( - \sum_{\ell \geq 1} \frac{N^\ell}{\ell\,j^\ell}
  \right)$ and therefore, for all $k\geq k_0$,
  \begin{align}
\nonumber
  N_k &= \prod_{j = 1}^{k_0-1} \left( \Id - \frac{N}{j} \right) \exp
  \left( - \sum_{\ell \geq 1} \frac{N^\ell}{\ell} \sum_{j = k_0}^k
    \frac{1}{j^\ell} \right) \\
        \label{eq:Nk-decomp}
      &= \prod_{j = 1}^{k_0-1} \left( \Id - \frac{N}{j} \right) \;
        \exp
  \left( - N \sum_{j = k_0}^k
    \frac{1}{j} \right) \;\exp
  \left( - \sum_{\ell \geq 2} \frac{N^\ell}{\ell} \sum_{j = k_0}^k
    \frac{1}{j^\ell} \right)\;.    
  \end{align}
  Moreover, we have the following
  asymptotic expansions. For all $k\geq k_0$,
      \begin{align*}
      \sum_{j = k_0}^k \frac{1}{j} &= \sum_{j=1}^k \frac{1}{j} - \sum_{j=1}^{k_0-1} \frac{1}{j} = \ln(k+1) + \gamma_e - \sum_{j=1}^{k_0 - 1} \frac{1}{j} + \frac{\alpha_k}{k}\;, \\
      \sum_{j = k_0}^k \frac{1}{j^\ell}  &= \sum_{j=k_0}^{\infty} \frac{1}{j^\ell} - \sum_{j=k+1}^{\infty} \frac{1}{j^\ell} = \frac{\beta_\ell}{k_0^\ell} + \frac{\eta_{k,\ell}}{(\ell-1)k^{\ell-1}}\;,\;\; \text{for all $\ell \geq 2$}\;,
      \end{align*}
      where $\gamma_e$ is Euler's constant,  $\beta_\ell := \sum_{k=k_0}^{\infty}
      \left(\frac{k_0}{k}\right)^\ell$, and
      $(\alpha_k)_{k \geq 1}$ and
      $(\eta_{k,\ell})_{k \geq 1, \ell \geq 2}$ are some universal
      constants satisfying
      \begin{equation}
        \label{eq:boun-alpha-eta}
       \sup_{k \geq 1} \abs{\alpha_k} < \infty 
      \quad\text{and}\quad\sup_{k \geq 1, \ell \geq 2} \abs{\eta_{k,\ell}} < \infty\;.
      \end{equation}
      Also note that
      \begin{equation}
        \label{eq:bound-beta-Nfiarma}
        \sup_{\ell \geq 2} \beta_\ell =\beta_2 < \infty\;.
      \end{equation}
      Using these definitions in~(\ref{eq:Nk-decomp}), we obtain, for all
      $k \geq k_0$,
      $$
      N_k = \bop (k+1)^{-N} \exp\left(-N\frac{\alpha_k}{k} - \sum_{\ell \geq 2} \frac{N^\ell \eta_{k,\ell} }{(\ell-1) k^{\ell-1}}\right) 
      $$
      where
      \begin{equation}
        \label{eq:def-bop-Nfiarma}
      \bop = \prod_{j = 1}^{k_0-1} \left( \Id - \frac{N}{j} \right) \exp \left(-N \left(\gamma_e - \sum_{t=1}^{k_0-1} \frac{1}{t} \right)\right) \exp\left( - \sum_{\ell \geq 2} \left(\frac{N}{k_0}\right)^\ell \frac{\beta_\ell}{\ell} \right)\;.
    \end{equation}
    Using the previous equations in~(\ref{eq:1-z-N-first-expansion}),
    for all $z\in\unitdisk$, we can write $(1 - z)^{N - \Id}$ as 
      \begin{align*}
        \Id + \sum_{k=1}^{k_0 - 1} \prod_{j=1}^k \left( \Id - \frac{N}{j} \right) z^k 
         + \bop \sum_{k \geq k_0}  (k+1)^{-N} \exp\left(-N\frac{\alpha_k}{k} - \sum_{\ell \geq 2} \frac{N^\ell \eta_{k,\ell} }{(\ell-1) k^{\ell-1}}\right)  z^k\;.
      \end{align*}
      Thus Relation \eqref{eq:dev-1minusz-fiarma} follows by setting
      \begin{align}
        \label{eq:def-Delta0}
      \aop^{\ast}_0& := \Id - \bop \;,& \\
        \label{eq:def-Deltakinfk0}
 \aop^{\ast}_k &:=\prod_{j=1}^k \left( \Id - \frac{N}{j} \right) - \bop \; (k+1)^{-N} \;,      &\text{for $1 \leq k \leq k_0 -1$,} \\
        \label{eq:def-Deltaksupk0}
 \aop^{\ast}_k &:= \bop (k+1)^{-N} \left[ \exp\left(-N \frac{\alpha_k}{k} -
        \sum_{\ell \geq 2} \frac{N^\ell \eta_{k,\ell} }{(\ell-1)
        k^{\ell-1}}\right) - \Id \right]&\text{for $k \geq k_0$.}
      \end{align}
      
      \noindent{\bf Step~2.} By~(\ref{eq:def-bop-Nfiarma}),
      (\ref{eq:Ncond-fiarma-infsup})
      and~(\ref{eq:bound-beta-Nfiarma}), first note that 
      $\norm{\bop}_{\infty}$ can be bounded by a constant only
      depending on $\varsigma$ (since $k_0$ only depends on
      $\varsigma$ as well). Hence, by~(\ref{eq:def-Delta0}) and~(\ref{eq:def-Deltakinfk0}), for
      $k<k_0\leq\varsigma+1$ we can again bound $\norm{\aop^{\ast}_k}$ by a
      constant only depending on $\varsigma$. Now
      let $k \geq k_0$, defining
      $$
      \Phi_k :=  -N \frac{\alpha_k}{k}  - \sum_{\ell \geq 2}
      \frac{N^\ell \eta_{k,\ell} }{(\ell-1) k^{\ell-1}}\;,
      $$
      Relation~(\ref{eq:def-Deltaksupk0}) yields
      $$
      \norm{\aop^{\ast}_k}_{\infty}
      \leq \norm{\bop}_{\infty} \norm{(k+1)^{-N}}_{\infty} \sum_{t \geq 1} \frac{\norm{\Phi_k}_{\infty}^t}{t!}\;.
      $$
      Using the singular value function $\mathrm{n}$ with~(\ref{eq:Ncond-fiarma-infsup-bis}), we have 
      $$
      \norm{(k+1)^{-N}}_{\infty} =
      \norm{(k+1)^{-M_{\mathrm{n}}}}_{\infty} =
      \norm{M_{(k+1)^{-\mathrm{n}}} }_{\infty} = \essuparg{\xi}_{v \in
        \Vset}\abs{(k+1)^{-\mathrm{n}(v)}} \leq (k+1)^{- \varrho} \;.
      $$
      Using the upper bound of operator norm of $N$ in~(\ref{eq:Ncond-fiarma-infsup}), we have
      \begin{align*}
      \norm{\Phi_k}_{\infty}
      &\leq \varsigma \frac{\abs{\alpha_k}}{k}+ \sum_{\ell \geq 2} \frac{\varsigma^\ell \eta_{k,\ell}}{(\ell-1) k^{\ell-1}}\\ 
      &= \frac{\varsigma}k \left(\abs{\alpha_k} + \sum_{\ell
        \geq 1} \frac{\varsigma^\ell}{\ell}
        \eta_{k,\ell+1}k^{1-\ell}\right)\\
        &\leq \frac{\varsigma}k \left(\abs{\alpha_k} + \lr{\sup_{k
          \geq 1, \ell \geq 2} \abs{\eta_{k,\ell}}}
          \,\varsigma \lr{1-\frac{\varsigma}{k_0}}^{-1}\right) \;.
      \end{align*}
      By~(\ref{eq:boun-alpha-eta}), this bound only depend on
      $\varsigma$ (since $k_0$ does as well). Gathering the obtained
      bounds we get~(\ref{eq:N-fiarma-boundDelta_bop}).

      \noindent{\bf Step~3.} We now assume $\varrho > 0$ and extend
      \eqref{eq:dev-1minusz-fiarma} to $\overline{\unitdisk} \setminus
      \{1\}$, that is to the case $z=\rme^{-\rmi \lambda}$ for some
      $\lambda \in \tore \setminus \{0\}$. For such a $\lambda$, we already have, for all $0 < a < 1$, 
      $$
      (1- a \rme^{-\rmi \lambda})^{N - \Id}  = \bop \sum_{k \geq 0} (k+1)^{-N} a^k \rme^{-\rmi \lambda k} + \sum_{k \geq 0} \aop^{\ast}_k a^k \rme^{-\rmi \lambda k} \; .
      $$
      Moreover, $(1- \rme^{-\rmi \lambda})^{N - \Id} = \lim_{a
        \uparrow 1} (1- a \rme^{-\rmi \lambda})^{N - \Id}$ by
      continuity of $z \mapsto (1-z)^{N- \Id}$ in
      $\overline{\unitdisk} \setminus \{1\}$ and $\sum_{k \geq 0}
      \aop^{\ast}_k \rme^{-\rmi \lambda k} = \lim_{a \uparrow 1} \sum_{k
        \geq 0} \aop^{\ast}_k a^k \rme^{-\rmi \lambda k}$  because $\sum_{k
        \geq 0} \norm{\aop^{\ast}_k}_{\infty} < \infty$. It remains
      to show that $\sum_{k \geq 0} (k+1)^{-N} z$ is well defined on
      $\unitcircle\setminus\{1\}$ and that, for $\lambda\in\tore\setminus\{0\}$,
      $\sum_{k \geq 0} (k+1)^{-N} a^k \rme^{-\rmi
        \lambda k}$ converges to $\sum_{k \geq 0} (k+1)^{-N} \rme^{-\rmi
        \lambda k}$  as $a\uparrow1$. We prove these facts at once by applying
      \Cref{lem:powerseries-vector-fiarma} with $a_k=(k+1)^{-N}$. We already used in Step~2
      that, for all $k \in \nset$, we
      have $\norm{(k+1)^{-N}}_{\infty} \leq (k+1)^{-\varrho}$.   Since $\varrho>0$, we get that
      $\norm{(k+1)^{-N}}_{\infty}\to0$ as $k\to\infty$. 
      Hence, to apply \Cref{lem:powerseries-vector-fiarma} it
      only remains to show
      \begin{align}
      \label{eq:last-of-the-last-fiarma}
              \sum_{k \in \nset} \norm{(k+1)^{-N} -
        (k+2)^{-N}}_{\infty} <\infty\;.
      \end{align}        
      Note that we have, for all $k\in\nset$,
      \begin{align}
      \label{eq:last-of-the-last-1-fiarma}
              \norm{(k+1)^{-N} -
        (k+2)^{-N}}_{\infty} =  \essuparg{\xi}_{v \in
          \Vset}   \abs{(k+1)^{-\mathrm{n}(v)} - (k+2)^{-\mathrm{n}(v)}}\;.
      \end{align}        
      Moreover, for all $k\in\nset$, and $\xi-\mae\,v\in\Vset$,
      since $\Re(\mathrm{n}(v))\geq\varrho>0$, we have
      \begin{align*}
      \abs{(k+1)^{-\mathrm{n}(v)} -
        (k+2)^{-\mathrm{n}(v)}}&=\abs{k+1}^{-\Re(\mathrm{n}(v))}\abs{1-\exp\left(-\ln\left(1+\frac1{k+1}\right)\,\mathrm{n}(v)\right)}\\
&        \leq \varsigma\;\alpha(\varsigma\ln(2))\;(k+1)^{-\varrho}\;\ln\left(1+\frac{1}{k+1}\right)\;,
      \end{align*}
      where, here, we set, for
      any $r>0$,
      $
      \alpha(r):=\sup\set{\abs{\frac{1-\rme^{-z}}{z}}}{z\in\cset\,\,0<\abs{z}\leq r}
      $.
      This leads
      to the asymptotic bound, as $k\to\infty$,
      $
       \essuparg{\xi}_{v \in
          \Vset}   \abs{(k+1)^{-\mathrm{n}(v)} - (k+2)^{-\mathrm{n}(v)}}=\bigo{(k+1)^{-\varrho-1}}
      $. 
       Hence, with~(\ref{eq:last-of-the-last-1-fiarma}) and the
       assumption $\varrho>0$, we obtain~(\ref{eq:last-of-the-last-fiarma}).
  \end{proof}

\begin{proof}[\bf Proof of \Cref{prop:longmemory-fiarma}]
  The first condition in~(\ref{eq:ass-LRD-fiarma}) gives
  that~(\ref{eq:Ncond-fiarma-infsup}) holds with
  $\varrho\geq1/2$. Applying \Cref{lem:dev-1minusz-fiarma},
  there exists $\bop \in \cL_b(\cH_0)$ and
  $(\aop^{\ast}_k)_{k \in \nset} \in \cL_b(\cH_0)^\nset$ with
  $\norm{\aop^{\ast}_k}_{\infty}=O(k^{-3/2})$ such that
  $(1 - \rme^{- \rmi \lambda})^{N - \Id} = \bop \sum_{k=0}^{\infty}
  (k+1)^{-N} \rme^{- \rmi \lambda k} + \sum_{k = 0}^\infty \aop^{\ast}_k
  \rme^{- \rmi \lambda k}$ in $\cL_b(\cH_0)$ for all
  $\lambda \in \tore\setminus\{0\}$, thus concluding the proof.
\end{proof}

\subsection{Proofs of \Cref{sec:est-main-assumpt-main}}
\label{sec:proof-estimation}
\subsubsection{Preliminary results}
In the following, for two separable Hilbert spaces $\cH_0$ and $\cI_0$
and a finite non-negative measure $\mu$ on $\lr{\tore,\btore}$, we
denote by $\norm{\cdot}_{1,1}$ the natural norm of the Bochner space
$L^{1,1}(\cH_0,\cI_0,\mu):=L^1\lr{\tore,\btore,\cS_1(\cH_0,\cI_0),\mu}$,
that is,
$$
\norm{g}_{1,1}:=\int\norm{g(\lambda)}_1\;\rmd\mu\;.
$$
We use the notation
$$
B_{1,1}(r,\cH_0,\cI_0,\mu)=\set{g\in
  L^1\lr{\tore,\btore,\cS_1(\cH_0,\cI_0),\mu}}{\norm{g}_{1,1}\leq r}\;,
$$
for the ball of radius $r$ in the Banach space
$L^{1,1}(\cH_0,\cI_0,\mu)$. As usual, if $\cI_0=\cH_0$, we drop $\cI_0$ in
the notation, thus writing $L^{1,1}(\cH_0,\mu)$ and $B_{1,1}(r,\cH_0,\mu)$ in
this case. Also if $\mu=\lebtore$ we drop the measure in the notation,
thus writing $L^{1,1}(\cH_0,\cI_0)$ and $B_{1,1}(r,\cH_0,\cI_0)$, or
$L^{1,1}(\cH_0)$ and $B_{1,1}(r,\cH_0)$ if $\cI_0=\cG_0$, in
this case.

These definitions and those introduced in
\Cref{sec:est-main-assumpt-main} (such as $\cQ$ and $\tilde{\cQ}$)
will be useful in the following. Recall, in particular, that
$\mathbb{F}_b\lr{\tore,\btore,\cE}$ denotes the set of bounded
measurable functions from $\lr{\tore,\btore}$ to
$\lr{\cE,\borel(\cE)}$. We will also need to define $\mathbb{F}_{b,b}\lr{\lr{\cH_0,\cG_0},\lr{\cI_0,\cJ_0}}$
as the product vector space
$\mathbb{F}_b\lr{\tore,\btore,\cL_b(\cI_0,\cJ_0)}\times\mathbb{F}_b\lr{\tore,\btore,\cL_b(\cH_0,\cG_0)}$,
endowed with the max norm 
$$
\norm{(L,R)}_{b,b}:=
\max(\sup(L),\sup(R))\quad\text{for all}\quad
(L,R)\in\mathbb{F}_{b,b}\lr{\lr{\cH_0,\cG_0},\lr{\cI_0,\cJ_0}}
\;.
$$
This simply extends the definition of $\mathbb{F}_{b,b}(\cH_0,\cG_0)$
already introduced in \Cref{sec:est-main-assumpt-main}, which can be
seen as a short-hand notation for
$\mathbb{F}_{b,b}\lr{\lr{\cH_0,\cG_0},\lr{\cH_0,\cG_0}}$.

We now derive a series of useful lemmas. 
  \begin{lemma}
  \label{lem:cont-quad-func-general}
Let  $\mu$ be a finite non-negative measure on $\lr{\tore,\btore}$ and $\cH_0$, $\cG_0$,  $\cI_0$ and $\cJ_0$ be four separable Hilbert
spaces and $g\in L^{1,1}(\cH_0,\cI_0,\mu)$.
Then the mapping $\cQ_{g,\mu}$ defined by
\begin{align}\label{eq:cQ-def-general-plus}
\cQ_{g,\mu}(L,R)= \int  L\, g \,R^\adjoint\;\rmd\mu 
\end{align}
is a continuous sesquilinear mapping from $\mathbb{F}_{b,b}\lr{\lr{\cH_0,\cG_0},\lr{\cI_0,\cJ_0}}$ to $\cS_1\lr{\cG_0,\cJ_0}$
satisfying
$$
\sup\set{\norm{\cQ_{g,\mu}(L,R)}_1}{\norm{(L,R)}_{b,b}\leq1}\leq\norm{g}_{1,1}\;.
$$
Consequently, for any positive radius $r$, the set
$\set{\cQ_{g,\mu}}{g\in
  B_{1,1}(r,\cH_0,\cI_0,\mu)}$
  is equicontinuous in
$\cC\lr{\mathbb{F}_{b,b}\lr{\lr{\cH_0,\cG_0},\lr{\cI_0,\cJ_0}},\cS_1\lr{\cG_0,\cI_0}}$.
\end{lemma}
\begin{proof}
For all $(L,R)\in\mathbb{F}_{b,b}\lr{\lr{\cH_0,\cG_0},\lr{\cI_0,\cJ_0}}$, $g\in L^{1,1}(\cH_0,\cI_0)$ and $\lambda\in\tore$, we have
$$
\norm{L(\lambda)\, g(\lambda)
  \,R(\lambda)^\adjoint}_1\leq \norm{L(\lambda)}_\infty\, \norm{g(\lambda)}_1
\,\norm{R(\lambda)^\adjoint}_\infty\,\;,
$$
which is thus integrable with respect to $\mu$. Moreover, we obtain that
$$
\norm{\cQ_{g,\mu}(L,R)}_1\leq \norm{g}_{1,1}\,\sup(L)\,\sup(R)\;.
$$
The given claim immediately follows as well as its consequence. 
\end{proof}
We also derive the following lemma, which will be useful in the following.
\begin{lemma}
  \label{lem:rel-comp-general1}
  Let $\cH_0$, $\cG_0$ and  $\cI_0$ be three separable Hilbert spaces,
  $\Theta$ be a compact metric space, and let $L$ and $R$ be two
  continuous mappings from $\Theta\times\tore$ into
  $\cL_b(\cH_0,\cG_0)$ and   $\cL_b(\cH_0,\cI_0)$, respectively.
  Then, for any positive radius $r$, the set
  $\cR(r):=\set{\tilde{\cQ}_g^{(L,R)}}{g\in B_{1,1}(r,\cH_0)}$
  is equicontinuous in  $\cC\lr{\Theta,\cS_1\lr{\cI_0,\cG_0}}$.
\end{lemma}
  \begin{proof}
    Let $r>0$. Since $\Theta\times\tore$ is compact, $R$ and $L$ are
    uniformly continuous on $\Theta\times\tore$ and we get that
    $\theta \mapsto
      (L(\theta,\cdot),R(\theta,\cdot)) \in
      \cC\lr{\Theta,\mathbb{F}_{b,b}\lr{\lr{\cH_0,\cG_0},\lr{\cH_0,\cI_0}}}$.
    By \Cref{lem:cont-quad-func-general}, we get that $\cR(r)$ is
    equicontinuous in
    $\cC\lr{\Theta,\cS_1\lr{\cI_0,\cG_0}}$.
  \end{proof}
We next provide four last preliminary lemmas, one about the
non-centered periodogoram, and two dealing with the centering term. 
  \begin{lemma}
    \label{lem:periodo-bounded-general}
    Let $\cH_0$ and $\cI_0$ be two separable Hilbert
    spaces such that $\cI_0$ is continuously embedded in
    $\cH_0$. Assume~\ref{item:ergo-assump} and suppose moreover that
    $X_0\in\cI_0$ $\as$ with
    $\PE\lrb{\norm{X_0}_{\cI_0}^2}<\infty$. Then, for all $n\geq1$,
    $I_n^X\in L^{1,1}(\cH_0,\cI_0)$  $\as$ and we have
    $$
    \sup_{n\geq1}\norm{I_n^X}_{1,1}<\infty\qquad\as
    $$
  \end{lemma}
  \begin{proof}
    Note that, for all $n\geq1$, we have,
  using the
  definition of $I_n^X$, and then that of $d^X_n$,
  \begin{align*}
    \norm{I_n^X}_{1,1} & =  \int\norm{I_n^X}_1\;\rmd\lebtore\\
    &\leq\int
      \norm{d_n^X(\lambda)}_{\cI_0}\norm{d_n^X(\lambda)}_{\cH_0}\;\rmd\lebtore\\
                       &\leq\lr{\int \norm{d_n^X(\lambda)}_{\cI_0}^2\;\rmd\lebtore}^{1/2}\;\lr{\int\norm{d_n^X(\lambda)}_{\cH_0}^2\;\rmd\lebtore}^{1/2}\\
&      = \lr{\frac1{n}\sum_{k=1}^n\norm{X_k}_{\cI_0}^2}^{1/2}\;\lr{\frac1{n}\sum_{k=1}^n\norm{X_k}_{\cH_0}^2}^{1/2}\;,
  \end{align*}
  where in the right-hand side of the first line $\norm{\cdot}_1$
  denotes the $\cS_1(\cH_0,\cI_0)$-norm. By~\ref{item:ergo-assump},
  with the Birkhoff ergodic theorem, we get that the right-hand of the
  previous bound converges $\as$ The claim  \as uniform bound of $\lr{\norm{I_n^X}_{1,1}}_{n\geq1}$ follows.
  \end{proof}
\begin{lemma}
  \label{lem:periodo-centering}
  Recall that $I_n^X$ and $I_n^{X^c_n}$ denote the periodograms respectively computed
  from $X_1,\dots,X_n$ and from $X_{n,1}^{c},\dots,X^c_{n,n}$, as
  defined in~(\ref{eq:centering}). Suppose that $X_1,\dots,X_n\in\cH_0$.
  Then, $I_n^X$ and $I_n^{X^c_n}$ belong to $\cS_1(\cH_0)$ and we have, for all $\lambda\in\tore$,
  $$
  \norm{I_n^X(\lambda)-I_n^{X^c_n}(\lambda)}_1\leq
  \norm{\frac1n\sum_{j=1}^nX_j}_{\cH_0}^2\,F_n(\lambda) +2   \norm{\frac1n\sum_{j=1}^nX_j}_{\cH_0}\,\norm{d_n^X(\lambda)}_{\cH_0}\,\lr{F_n(\lambda)}^{1/2}\;,
  $$
  where $F_n$ denotes the Fejér kernel defined by
    \begin{equation}
    \label{eq:def-fejer}
  F_n(\lambda)=\frac1n\lrav{\sum_{k=1}^n\rme^{-\rmi\,\lambda\,k}}^2\;.    
  \end{equation}
\end{lemma}
\begin{proof}
  By~(\ref{eq:centering}) and~(\ref{eq:def-dnX}), we have, for all $\lambda\in\tore$,
  $$
  d^{X_n^c}_n(\lambda)=d_n^X(\lambda)-\lr{\frac1n\sum_{j=1}^nX_j}\,\lr{\frac1{\sqrt{n}}\sum_{k=1}^n\rme^{-\rmi\,\lambda\,k}}\;.
  $$
  Computing
  $I_n^{X^c_n}(\lambda)=d^{X_n^c}_n(\lambda)\otimes d^{X_n^c}_n(\lambda)$
  and using that $\norm{x\otimes
    y}_1=\norm{x}_{\cH_0}\,\norm{y}_{\cH_0}$ for all $x,y\in\cH_0$, we
  easily get the result. 
\end{proof}
We have the following result on the process $X$. 
\begin{lemma}
  \label{lem:A1A2implyA4}
  Let $\cH_0$ be a separable Hilbert
  space. Assume~\ref{item:ergo-assump}
  and~\ref{item:sp_density-assump}. Then, the process $X$ is valued in
  a finite-dimensional space $\cG_0\subset\cH_0$ or, if it is not the
  case, there always exists an orthonormal sequence
  $(\phi_k)_{k\in\nset}$ of $\cH_0$ and a sequence
  $s=(s_k)_{k\in\nset}\in[1,\infty)^\nset$ such that
  Conditions~\ref{item:s-sob-assump} and~\ref{item:X-sob-assump} in
  \Cref{thm:periodo-func} hold.
\end{lemma}
\begin{proof}
  Define $\Sigma=\PEarg{X_0\otimes X_0}$. Since
  $\Sigma\in\cS_1^+(\cH_0)$, there exists a finite or countable
  non-increasing sequence
  $\lr{\sigma_k}_{0\leq k<K}\in(0,\infty)^{\nset}$ and an orthonormal
  sequence $\lr{\phi_k}_{0\leq k<K}\in\cH_0^{\nset}$ such that
  $$
  \Sigma=\sum_{0\leq
    k<K}\sigma^2_k\,\phi_k\otimes\phi_k\quad\text{and}\quad\sum_{0\leq
    k<K}\sigma_k^2<\infty\;.
  $$
  In particular, we have that $\as$, $X_0$ is valued in
  $\cspanarg[\cH_0]{\phi_k\,,\,0\leq k<K}$.  If $K$ is finite, then
  $X_0$ is valued in the finite-dimensional space
  $\cG_0=\lspan{\phi_k\,,\,0\leq k<K}$.

  From now on, we take $K=\infty$. By \Cref{lem:1}, we can find
  $s=(s_k)_{k\in\nset}\in[1,\infty)^{\nset}$, non-decreasing and going
  to $\infty$ (hence, satisfying Condition~\ref{item:s-sob-assump} in
  \Cref{thm:periodo-func}), such that
  $\sum_{k\in\nset}s_k^2\sigma^2_k<\infty$. Defining $\cH_0^s$
  by~(\ref{eq:defHOs}) and its inner product by~(\ref{eq:pscalHOs}),
  we get that
  $$
  \PEarg{\norm{X_0}_{\cH_0^s}^2}=\sum_{k\in\nset}s_k^2\sigma^2_k<\infty\;.
  $$
  We thus have Condition~\ref{item:X-sob-assump} in \Cref{thm:periodo-func}.
\end{proof}
The following lemma is used to treat the centering term in the next
result, and also to prove \Cref{thm:1}.
\begin{lemma}
\label{lem:as-mean-cv}  Let $\cH_0$ be a separable Hilbert
  space. Assume~\ref{item:ergo-assump}
  and~\ref{item:sp_density-assump}. Then, 
  \begin{equation}
    \label{eq:cv-as-emp-mean}
\lim_{n\to\infty} \frac1n\sum_{j=1}^nX_j = \PEarg{X_0}  \quad \text{in $\cH_0$}\;,\quad\as    
  \end{equation}
\end{lemma}
\begin{proof}
  If we can find a finite-dimensional space $\cG_0\subset\cH_0$ such
  that $X$ is valued in $\cG_0$, then~(\ref{eq:cv-as-emp-mean})
  follows straightforwardly from the Birkhoff ergodic theorem. If not,
  by \Cref{lem:A1A2implyA4}, we can find an orthonormal sequence
  $(\phi_k)_{k\in\nset}$ of $\cH_0$ such that
  Conditions~\ref{item:s-sob-assump} and~\ref{item:X-sob-assump} in
  \Cref{thm:periodo-func} hold. As a consequence, by the Birkhoff
  ergodic theorem, we get that
  $$
\lim_{n\to\infty}  \frac1n\sum_{j=1}^n\norm{X_j}_{\cH_0^s}= \PEarg{\norm{X_0}_{\cH_0^s}}\quad\as
  $$
  Since
  $\norm{\frac1n\sum_{j=1}^nX_j}_{\cH_0^s}\leq\frac1n\sum_{j=1}^n\norm{X_j}_{\cH_0^s}$
  for all $n\geq1$, we get that, $\as$, there exists $r>0$ such that
  $\sup_{n\geq1}\norm{\frac1n\sum_{j=1}^nX_j}_{\cH_0^s}\leq r$.  Using
  that $(s_k)_{k\in\nset}$ is going to infinity, we get that the
  operator $\sum_{k\in\nset}s_k^{-1}\phi_k\otimes\phi_k$ belongs to
  $\cS_\infty(\cH_0)$. Since the image by this operator of the unit
  $\cH_0$-ball is the unit $\cH_0^s$-ball, we get that all
  $\cH_0^s$-balls are compact in $\cH_0$. In particular, we have that
  $\lr{\frac1n\sum_{j=1}^nX_j}_{n\geq1}$ is $\as$ valued in a compact
  subset of $\cH_0$.  Therefore, it only remains to
  show that, $\as$, the only possible accumulation point of this
  sequence is $\PEarg{X_0}$. This fact follows by using the Birkhoff
  ergodic theorem again, which gives us, for all $x\in\cH_0$, 
  $$
  \lim_{n\to\infty}\pscal{\frac1n\sum_{j=1}^nX_j}{x}_{\cH_0}=
  \lim_{n\to\infty}\frac1n\sum_{j=1}^n\pscal{X_j}{x}_{\cH_0} = \PEarg{\pscal{X_0}{x}_{\cH_0}}= \pscal{\PEarg{X_0}}{x}_{\cH_0}\quad\as
  $$
  Hence we get the
  convergence~(\ref{eq:cv-as-emp-mean}).
\end{proof}
  We can now state a result which applies both in the context of
  \Cref{thm:periodo-func-fidi} and \Cref{thm:periodo-func}.
  \begin{theorem}
\label{thm:periodo-func-general}
Let $\cH_0$ and $\cG_0$ be two separable Hilbert
spaces. Assume~\ref{item:ergo-assump}
and~\ref{item:sp_density-assump}, and suppose that $(\Theta,\Delta)$
is a compact metric space. Let $L$ and $R$ in
$\cC\lr{\Theta\times\tore,\cL_b(\cH_0,\cG_0)}$. Moreover, suppose
that, $\as$, there exists, for all $\theta\in\Theta$,  a compact
subset $B\subset\cS_1\lr{\cG_0}$ such that, for all $n\geq1$
$\tilde{\cQ}_{I_n^X}^{(L,R)}(\theta)\in B$.  Then, we have
  \begin{equation}
    \label{eq:unif-periodo-mapping-convergence}
\lim_{n\to\infty}\tilde{\cQ}_{I_n^{X^c_n}}^{(L,R)}=\tilde{\cQ}_{\nu_X}^{(L,R)}\quad\text{uniformly
  in}\quad\cC\lr{\Theta,\cS_1\lr{\cG_0}}\;,\quad\as
\end{equation}
\end{theorem}
\begin{proof}
  Thanks to the centering~(\ref{eq:centering}), we can replace all
  $X_k$'s by $X_k-\PEarg{X_0}$ without modifying $X^c_{n,k}$
  for any $n,k\geq1$. Hence,  without loss of generality, from now on
  in this proof, we assume that $\PEarg{X_0}$ is
  zero, that is, the process $X$ is centered.
  
  By \Cref{lem:periodo-centering} and using that $\int F_n\;\rmd\lebtore=1$, first note that for all
  $\theta\in\Theta$ we have, for all $n\geq1$,
  $$
  \norm{\tilde{\cQ}_{I_n^{X^c_n}}(\theta)-\tilde{\cQ}_{I_n^{X}}(\theta)}_1\leq
  A_n\,\lr{A_n+B_n}\,\sup_{\stackrel{\theta\in\Theta}{\lambda\in\tore}}
  \norm{L(\theta,\lambda)}_\infty\, \sup_{\stackrel{\theta\in\Theta}{\lambda\in\tore}}\norm{R(\theta,\lambda)}_\infty\;,
  $$
  where we set
  $$
  A_n:= \norm{\frac1n\sum_{j=1}^nX_j}_{\cH_0}\quad\text{and}\quad B_n=2   \,
  \int \norm{d_n^X}_{\cH_0}\,\lr{F_n}^{1/2}\;\rmd\lebtore\;.
  $$
  By the Cauchy-Schwartz inequality and then the Parseval identity, we have
  $$
  B_n\leq 2\,\lr{  \int \norm{d_n^X}^2_{\cH_0}\;\rmd\lebtore}^{1/2}=2\,\lr{\frac1n\sum_{k=1}^n\norm{X_k}_{\cH_0}^2}^{1/2}\;,
  $$
  which, by the Birkhoff ergodic theorem
  and~\ref{item:sp_density-assump}, converges $\as$. On the other hand,
  using \Cref{lem:as-mean-cv}, we have that, $\as$,
  $\lim_{n\to\infty}A_n=0$. Hence, we finally get that
  $$
  \lim_{n\to\infty}\sup_{\theta\in\Theta}\norm{\tilde{\cQ}_{I_n^{X^c_n}}(\theta)-\tilde{\cQ}_{I_n^{X}}(\theta)}_1
  = 0\quad\as
  $$
  Therefore, to prove~(\ref{eq:unif-periodo-mapping-convergence}), we
  can replace $I_n^{X^c_n}$ by  $I_n^{X}$, that is, it only remains to
  prove that 
  \begin{equation}
    \label{eq:unif-periodo-mapping-convergence-noncentered}
    \lim_{n\to\infty}\tilde{\cQ}_{I_n^{X}}^{(L,R)}=\tilde{\cQ}_{\nu_X}^{(L,R)}\quad\text{uniformly
  in}\quad\cC\lr{\Theta,\cS_1\lr{\cG_0}}\;,\quad\as
\end{equation}
  We immediately have from \Cref{prop:cQ-well-def} that
  $\tilde{\cQ}_{I_n^X}^{(L,R)}$ and $\tilde{\cQ}_{\nu_X}^{(L,R)}$ belong
  to $\cC\lr{\Theta,\cS_1\lr{\cG_0}}$. The rest of the proof is now in three steps. First, we show that, $\as$, every sequence valued in
  $\set{\tilde{\cQ}_{I_n^X}^{(L,R)}}{n\geq1}$ admits a subsequence
  which converges uniformly in 
  $\cC\lr{\Theta,\cS_1\lr{\cG_0}}$. Second, we show that, for
  all $x,y\in\cG_0$ and $\theta\in\Theta$, we have 
  \begin{equation}
    \label{eq:step1-pointwise-cv}
  \lim_{n\to\infty} x^\adjoint\tilde{\cQ}_{I_n^X}^{(L,R)}(\theta) y = x^\adjoint\tilde{\cQ}_{\nu_X}^{(L,R)}(\theta) y\;,\quad\as
  \end{equation}
  We then conclude in Step~3 from these two results.

  \noindent\textbf{Step~1}.   By~\ref{item:sp_density-assump}, we have
  $\PEarg{\norm{X_0}_{\cH_0}^2}<\infty$, and we can apply
  \Cref{lem:periodo-bounded-general} with $\cI_0=\cH_0$ and get that, $\as$, there exists $C$ such
  that, for all $n\in\nset$, $\norm{I_n^X}_{1,1}\leq C$, where here
  $\norm{\cdot}_{1,1}$ denotes the norm in $L^{1,1}(\cH_0)$.  We conclude
  with \Cref{lem:rel-comp-general1} that, $\as$,
  $\set{\tilde{\cQ}_{I_n^X}^{(L,R)}}{n\geq1}$ is equicontinuous in
  $\cC\lr{\Theta,\cS_1(\cG_0)}$.  Using the last assumption of the theorem,
  we also have that, $\as$, for all $\theta\in\Theta$, there is a compact subset $B$ of
  $\cS_1(\cG_0)$ such that $\set{\tilde{\cQ}_{I_n^X}^{(L,R)}(\theta)}{n\geq1}$
  is included in $B$. Thus, $\as$, the
  Ascoli-Arzelà theorem (see \cite[Section~7.10]{royden:1988}) applies
  and any sequence in  $\set{\tilde{\cQ}_{I_n^X}^{(L,R)}}{n\geq1}$ admits a
  uniformly convergent subsequence in $\cC\lr{\Theta,\cS_1(\cG_0)}$, which
  concludes the proof of Step~1.
  
  \noindent\textbf{Step~2}. This step is similar to the scalar
  case and we follow the ideas of \cite{hannan1973}. Let $x,y\in\cG_0$, $\theta\in\Theta$,
  and denote $\bx=x^\adjoint L(\theta,\cdot)$ and $\by=y^\adjoint R(\theta,\cdot)$, so that
  $\bx$ and $\by$ are continuous functions from $\tore$ to $\cH_0^\adjoint=\cL_b(\cH_0,\cset)$,
  and~(\ref{eq:step1-pointwise-cv}) can be written as
  \begin{equation}
    \label{eq:step1-pointwise-cv2}
  \lim_{n\to\infty} \cQ_{I_n^X}\lr{\bx,\by} =\cQ_{\nu_X}\lr{\bx,\by}\;, \quad\as
  \end{equation}
 We can write, for any $N\in\nset^*$,
  $$
  \bx= F_N\star \bx + \lr{\bx-F_N\star \bx}\;,
  $$
  where $\star$ denotes the convolution of locally integrable $2\pi$-periodic
  functions,
  $$
  f\star g(\lambda)=\int f(\lambda')\,g(\lambda-\lambda')\;\lebtore(\rmd\lambda')\;,\quad\lambda\in\tore\;,
  $$
  and $F_N$ is the Fejér kernel defined by~(\ref{eq:def-fejer}).
  Using standard properties of Fejér's kernel and the fact that
  $\lambda\mapsto \bx(\lambda)$ is continuous on $\rset$, we have, denoting
  $\mathbf{e}_\ell(\lambda)=\rme^{\rmi\lambda\ell}$ so that
  $(\mathbf{e}_\ell)_{\ell\in\zset}$ is a Hilbert basis of $\cL_0:=L^2\lr{\tore,\btore,\lebtore}$,
  \begin{align}\label{eq:fejer1}
&    F_N\star
    \bx=\sum_{\ell=-N}^N\alpha_{\ell}(N)\;c_\ell(\bx)\;\mathbf{e}_{\ell}\;,\\
    \nonumber
&    \text{with}\quad\alpha_{\ell}(N)=1-\frac{|\ell|}{N}\quad\text{and}\quad
    c_\ell(\bx)  = \int
                                          \bx(\lambda)\rme^{-\rmi\ell\lambda}\;\lebtore(\rmd\lambda)\;,\\
    \label{eq:fejer2}
&\lim_{N\to\infty}\sup_{\lambda\in\rset}    \norm{\bx(\lambda)-F_N\star \bx(\lambda)}_{\infty}=0
\;.
  \end{align}
  Eq~(\ref{eq:fejer2}) can be interpreted as saying that $F_N\star \bx$ converges to
  $\bx$ in
  $\mathbb{F}_b\lr{\tore,\btore,\cL_b(\cH_0,\cset)}$. The same holds
  with $\by$ replacing $\bx$ and, applying
  \Cref{lem:cont-quad-func-general}
  with $\cI_0=\cH_0$ and $\cG_0=\cJ_0=\cset$ and $\mu=\lebtore$, since by Step~1, $\set{I_n^X}{n\geq1}$
  remains in a ball of
  $L^{1,1}$ $\as$, we have, $\as$,
  \begin{align}
    \label{eq:fejer-remainder-part}
    \lim_{N\to\infty}
\sup_{g\in\set{I_n^X}{n\geq1}}\lrav{\cQ_g(F_N\star
    \bx,F_N\star
    \by) - \cQ_g(\bx,\by)}\;.
  \end{align}
  Similarly by the continuity of $\cQ_{\nu_X}=\cQ_{f_X,\mu}$ (with
  $f_X$ the density of $\nu_X$ with respect to $\mu=\norm{\nu_X}_1$) established in
  \Cref{lem:cont-quad-func-general}, we have
  \begin{align}
    \label{eq:fejer-remainder-part-nuX}
    \lim_{N\to\infty} \cQ_{\nu_X}(F_N\star
    \bx,F_N\star
    \by) = \cQ_{\nu_X}(\bx,\by)\;.
  \end{align}   
  Next, using~(\ref{eq:fejer1}), we have  
  \begin{align}
    \nonumber
    \cQ_{I_n^X}(F_N\star
    \bx,F_N\star
    \by)&=\sum_{\ell,\ell'=-N}^N\alpha_{\ell}(N)\,\alpha_{\ell'}(N)\,c_{\ell}(\bx)\,\cQ_{I_n^X}\lr{\mathbf{e}_{\ell},\mathbf{e}_{\ell'}}\,c_{\ell'}(\by)^\adjoint\\
    \label{eq:fejer-part}    &=\sum_{\ell,\ell'=-N}^N\alpha_{\ell}(N)\,\alpha_{\ell'}(N)\,c_\ell(\bx)\;\tilde{\Gamma}_n(\ell-\ell')\,c_{\ell'}(\by)^\adjoint\;,
  \end{align}
  where $\tilde{\Gamma}_n$ denotes the empirical covariance defined as
  in~(\ref{eq:emp-cov-centered}), but with $X_n^c$ replaced by $X$,
  that is,
  $$
  \tilde{\Gamma}_n(s-t)=\int
  I_n^{X}(\lambda)\,\rme^{\rmi\,(s-t)\,\lambda}\;\lebtore(\rmd\lambda)=\frac1n\sum_{\stackrel{1\leq k,k'\leq
      n}{k-k'=(s-t)}}
  X_{k}\otimes X_{k'}\;.    
  $$
  In particular, we have for any $\ell,\ell'\in\lrcb{-N,\dots,N}$,
  $$
  c_\ell(\bx)\;\tilde{\Gamma}_n(\ell-\ell')\,c_{\ell'}(\by)^\adjoint
  = \frac1n\sum_{1\vee(1+\ell-\ell')\leq k\leq n\wedge(n+\ell-\ell')}c_\ell(\bx)\,X_k\,X_{k-(\ell-\ell')}^\adjoint\,c_{\ell'}(\by)^\adjoint\,
  $$
  By~\ref{item:ergo-assump}, with the Birkhoff ergodic theorem, we
  get, for any $\ell,\ell'\in\lrcb{-N,\dots,N}$, $\as$,
  \begin{align*}
    \lim_{n\to\infty}
    c_\ell(\bx)\;\tilde{\Gamma}_n(\ell-\ell')\,c_{\ell'}(\by)^\adjoint
    &
      =\PE\lrb{c_\ell(\bx)\,X_{\ell-\ell'}\,X_0^\adjoint\,c_{\ell'}(\by)^\adjoint}\\
    &=c_\ell(\bx)\Cov{X_{\ell-\ell'}}{X_0}c_{\ell'}(\by)^\adjoint\\
    &=c_\ell(\bx)\cQ_{\nu_X}\lr{\mathbf{e}_{\ell},\mathbf{e}_{\ell'}}\,c_{\ell'}(\by)^\adjoint \;.
  \end{align*}
  where we used that $X$ is centered and that $\nu_X$ is the spectral
  operator measure of $X$.  From~(\ref{eq:fejer-part}),
  using~(\ref{eq:fejer1}) again, we get that, $\as$, for any $N\geq1$,
  $$
  \lim_{n\to\infty}\cQ_{I_n^X}(F_N\star
    \bx,F_N\star
    \by) =\cQ_{\nu_X}(F_N\star
    \bx,F_N\star
    \by)  \;.
$$
This, with~(\ref{eq:fejer-remainder-part}) and~(\ref{eq:fejer-remainder-part-nuX}), concludes Step~2.

  \noindent\textbf{Step~3}. From Step~1, there exists $\Omega'\in\cF$
  with probability 1 such that on $\Omega'$, any sequence valued in
  $\set{\tilde{\cQ}_{I_n^X}^{(L,R)}}{n\geq1}$ admits a subsequence
  uniformly converging in $\cC\lr{\Theta,\cS_1(\cH_0)}$.  To
  obtain~(\ref{eq:unif-periodo-mapping-convergence-noncentered}), we will exhibit $\Omega''\subset\Omega'$
  with probability one such that, on $\Omega''$,
  $\tilde\cQ_{\nu_X}^{(L,R)}$ is the only possible accumulation point
  of the sequence $\lr{\tilde{\cQ}_{I_n^X}^{(L,R)}}_{n\geq1}$. Let
  $E_0$ be a countable linearly dense subset of $\cG_0$ and let
  $(\theta_j)_{j\in\nset}$ be a dense sequence in $\Theta$, which
  exists since $\Theta$ is compact. Then, from Step~2, we have,
  $\as$,
  $$
\forall  j\in\nset\,,\forall x,y\in E_0\,,\;\lim_{n\to\nset}x^\adjoint\tilde{\cQ}_{I_n^X}^{(L,R)}\,y=
  x^\adjoint\,\tilde{\cQ}_{\nu_X}^{(L,R)}\,y\;.
  $$
  We can thus take $\Omega''\subset\Omega'$ with probability one, on
  which the previous display holds. Let $\omega\in\Omega''$ and take
  an accumulation point $\tilde{\cQ}_\infty$ of
  $\lr{\tilde{\cQ}_{I_n^{X(\omega)}}^{(L,R)}}_{n\geq1}$ in
  $\cC\lr{\Theta,\cS_1(\cG_0)}$.  Then, for all $j\in\nset$, using the
  previous display and the fact that $\tilde{\cQ}_\infty$ must also be
  an accumulation point for the weak operator topology, we get that,
  for all $x,y\in E_0$,
  $x^\adjoint\tilde{\cQ}_\infty(\theta_j)
  y=x^\adjoint\tilde{\cQ}_{\nu^{X}}^{(L,R)}(\theta_j)y$, which implies
  $\tilde{\cQ}_\infty(\theta_j)=\tilde{\cQ}_{\nu^{X}}^{(L,R)}(\theta_j)$. Since
  $\tilde{\cQ}_\infty$ and $\tilde{\cQ}_{\nu^{X}}^{(L,R)}$ are
  continuous on $\Theta$, and $(\theta_j)_{j\in\nset}$ is dense in
  $\Theta$, we get that $\tilde{\cQ}_\infty$ and
  $\tilde{\cQ}_{\nu^{X}}^{(L,R)}$ coincide, which concludes the proof.
\end{proof}
\subsubsection{Proof of \Cref{thm:periodo-func-fidi}}
We can now prove \Cref{thm:periodo-func-fidi} as a direct application
of \Cref{thm:periodo-func-general}. 
\begin{proof}[\bf Proof of \Cref{thm:periodo-func-fidi}]
  \Cref{thm:periodo-func-fidi} directly follows from
  \Cref{thm:periodo-func-general}, if we can prove that, $\as$, there
  exists $B$, a compact subset of $\cS_1(\cG_0)$, such that 
  $\tilde{\cQ}_{I_n^X}^{(L,R)}(\theta)\in B$ for all $n\geq1$ and
  $\theta\in\Theta$. Because $\cG_0$ is finite-dimensional, so is
  $\cS_1(\cG_0)$, and we only need to show that, $\as$, there exists $C>0$ such that
  $\norm{\tilde{\cQ}_{I_n^X}^{(L,R)}(\theta)}_1\leq C$ for all $n\geq1$ and
  $\theta\in\Theta$. By \Cref{lem:cont-quad-func-general}, this
  follows from the fact that, $\as$, there exists $r>0$ such that
  $\norm{I_n^X}_1\leq r$ for all $n\geq1$, which has already been used
  in the proof of \Cref{thm:periodo-func-general} and is a consequence
  of \Cref{lem:periodo-bounded-general} with
  \ref{item:sp_density-assump} in the case $\cH_0=\cI_0$. This
  concludes the proof.
\end{proof}
\subsubsection{Proof of \Cref{thm:periodo-func}}
The proof of \Cref{thm:periodo-func} essentially follows the same path
as that of \Cref{thm:periodo-func-fidi}. However, in the
infinite-dimensional case, we will need an additional result
(\Cref{prop:compactS1ballinS1}) to prove the assumption involving the
set $B$ in \Cref{thm:periodo-func-general}.  The result relies on the
space $\cH_0^s$ introduced in~\Cref{sec:est-main-assumpt-main}. In
this section, we will make extensive use of partial isometries as (see
\cite[Definition~ 3.8]{conwaycourseope}). We recall that a partial
isometry $U$ on the Hilbert space $\cH_0$ onto another Hilbert space
$\cG_0$ is a bounded operator which is an isometry on
$(\ker(U))^\perp$. The subspaces $(\ker(U))^\perp$ and $\range(U)$ are
respectively called the \emph{initial space} and \emph{final space} of
$U$. We recall that, if $U$ is a partial isometry, then $U^\adjoint U$
and $U U^\adjoint$ are the orthogonal projections onto the initial and
the final space of $U$ respectively.

Let us start with the following lemma, whose proof is straightforward,
but which contains some important definitions 
\begin{lemma}
  \label{lem:JsLemma}
  Let $\cH_0$ be a separable Hilbert space and let
  $(\phi_k)_{k\in\nset}$ be an orthonormal sequence in $\cH_0$.
  Let
  $s=(s_k)_{k\in\nset}\in[1,\infty)^\nset$ and define
  $\lr{\cH_0^s,\pscal{\cdot}{\cdot}_{\cH_0^s}}$ by~(\ref{eq:defHOs})
  and~(\ref{eq:pscalHOs}). Denote by $\mathrm{J}_s$ the continuous
  $\cH_0^s\hookrightarrow\cH_0$-inclusion map defined on $\cH_0^s$
  onto $\cH_0$ by $x\mapsto x$. Further denote by $U_s$ the partial
  isometry with initial space
  $\cspanarg[\cH_0]{\phi_k,k\in\nset}$ and final space $\cH_0^s$ such that,
  for all $k\in\nset$, $U_s\phi_k=s_k^{-1}\,\phi_k$. Finally,
  let us set $\mathbb{J}_s=\mathrm{J}_s\,U_s\in\cL_b(\cH_0)$. Then, for all
  $x\in\cH_0$, we have 
\begin{equation}
    \label{eq:yy-sob-bound}
    \mathbb{J}_s=\sum_{k\in\nset} s_k^{-1}\;\phi_k\otimes\phi_k \;.
 \end{equation}
Moreover, suppose that $s$ is non-decreasing and going  to
$\infty$. Then, we have $\mathbb{J}_s\in\cS_{\infty}\lr{\cH_0}$. 
\end{lemma}
We now have the following result.
\begin{proposition}
  \label{prop:compactS1ballinS1}
  Let $\cH_0$ be a separable Hilbert space and let
  $(\phi_k)_{k\in\nset}$ be a an orthonormal sequence in $\cH_0$.
  Define $\mathbb{J}_s$ as in \Cref{lem:JsLemma} for
  $s=(s_k)_{k\in\nset}\in[1,\infty)^\nset$ non-decreasing and going to
  $\infty$.  Define
  $\mathbb{B}_1=\set{\aop\in\cS_1(\cH_0)}{\norm{\aop}_1\leq1\quad\text{and}\quad\lr{\mathbb{J}_s\,\aop}^\adjoint=\mathbb{J}_s\,\aop}$.
  Then the set $B_1=\set{\mathbb{J}_s\,\aop}{\aop\in\mathbb{B}_1}$ is
  compact in $\cS_1(\cH_0)$.
\end{proposition}
\begin{proof}
  Let $(\aop_n)_{n\in\nset}$ be a sequence valued in $B_1$ and let us
  prove that it admits a subsequence which converges in $B_1$. By
  definition, we can write, for all $n\in\nset$,
  $\aop_n=\mathbb{J}_s\,\tilde{\aop}_n$ with
  $\tilde{\aop}_n\in\mathbb{B}_1$. We use that $\cS_1(\cH_0)$ is
  isometric to the dual of the space $\cS_\infty(\cH_0)$ (see
  \cite[Theorem~19.1]{conwaycourseope}. Then, by the Banach-Alaoglu
  Theorem (see Theorem~3.1 in \cite[Chapter V]{conway1994course}), we
  get that the unit ball of $\cS_1(\cH_0)$ is compact for the
  weak-star topology, that is the topology generated by the family of
  semi-norms
  $\set{\aop\mapsto\abs{\tr(\cop\aop)}}{\cop\in
    \cS_\infty(\cH_0)}$. This implies that
  $(\tilde\aop_n)_{n\in\nset}$ admits a subsequence
  $(\tilde\aop_{a_n})_{n\in\nset}$ converging to an element
  $\tilde{\bop}$ in the unit ball of $\cS_1(\cH_0)$ in the
  sense of the weak-star topology, that is, for all
  $\cop\in \cS_\infty(\cH_0)$, we have
  $$
  \lim_{n\to\infty}\tr\lr{\cop\tilde{\aop}_{a_n}}=\tr\lr{\cop\tilde{\bop}}\;.
  $$
  Observe that for all $x,y\in\cH_0$, the operator
  $\cop = x \otimes y \mathbb{J}_s$ is a rank-one (hence compact)
  linear operator on $\cH_0$ onto $\cH_0$. The last display thus gives
  that $\tilde{\aop}_{a_n}$ converges to $\tilde{\bop}$
  in weak operator topology (that is, for all $x,y\in\cH_0$,
  $\pscal{\tilde{\aop}_{a_n}x}{y}_{\cH_0}$ converges to
  $\pscal{\bop\,x}{y}_{\cH_0}$). Since
  $\mathbb{J}_s\,\tilde{\aop}_{a_n}$ is hermitian for all $n$, we get
  that $\mathbb{J}_s\,\tilde{\bop}$ is
  hermitian as well and we finally get that $\tilde{\bop}$ must be in
  $\mathbb{B}_1$. (In fact we have shown that $\mathbb{B}_1$ is
  compact for the weak-star topology).

  Let us set $\bop = \mathbb{J}_s\tilde{\bop} \in B_1$ and for all
  $n\in\nset$, $\Delta_n=\aop_{a_n}-\bop=\mathbb{J}_s\tilde\Delta_n$
  with $\tilde{\Delta}_n = \tilde\aop_{a_n} - \tilde\bop$, and let us
  summarize our findings so far.  We already know that $\aop_{a_n}$
  and $\bop$ are in $B_1$ (hence are hermitian and so is $\Delta_n$),
  that $\norm{\tilde{\Delta}_n}_1\leq2$ and that
  $\lr{\tilde{\Delta}_n}_{n\to\infty}$ converges to zero in
  $\cS_1(\cH_0)$ for the weak-star topology, which also implies the
  convergence in weak operator topology. To conclude, we now proceed in
  two steps. First, we show that
  $(\Delta_n)_{n\in\nset}$ converges to $0$ for the strong operator
  topology. Second, we use the first step to show that $\lim_{n\to+\infty}\norm{\Delta_{n}}_1=0$.

\noindent\textbf{Step~1}.
Let $x\in\cH_0$, then, for all $n\in\nset$,
  using~(\ref{eq:yy-sob-bound}), we have
\begin{align*}
  \norm{\Delta_n\,x}_{\cH_0}^2=
                              \sum_{k\in\nset}\lrav{\pscal{\Delta_n\,x}{\phi_k}_{\cH_0}}^2=
                              \sum_{k\in\nset}s_k^{-2}\,\lrav{\pscal{\tilde\Delta_n\,x}{\phi_k}_{\cH_0}}^2\;.
\end{align*}
Since $\lr{\tilde\Delta_n}_{n\in\nset}$ converges to 0 for the weak operator topology,  we have, for all $m\geq1$,
$$
\lim_{n\to\infty}
\sum_{k=0}^{m-1}s_k^{-2}\,\lrav{\pscal{\tilde\Delta_n\,x}{\phi_k}_{\cH_0}}^2=0\;.
$$
On the other hand, for all $m,n\geq1$, using the fact that
$\norm{\tilde\Delta_n}_\infty\leq\norm{\tilde\Delta_n}_1\leq 2$ and
that $s$ is non-decreasing, we get that
\begin{align*}
  \sum_{k=m}^\infty s_k^{-2}\,\lrav{\pscal{\tilde\Delta_n\,x}{\phi_k}_{\cH_0}}^2
 & \leq s_m^{-2}\,\sum_{k=0}^\infty\lrav{\pscal{\tilde\Delta_n\,x}{\phi_k}_{\cH_0}}^2
  \\
  &=    s_m^{-2}\,\norm{\tilde\Delta_n\,x}_{\cH_0^s}^2 \leq 4\,s_m^{-2}\,\norm{x}_{\cH_0}^2\;,
\end{align*}
hence converges to 0 independently of $n$ as $m\to\infty$, by
assumption on $s$. With the
two previous displays, we conclude that $\lr{\Delta_n\,x}_{n\in\nset}$
converges to 0 in $\cH_0$.  Hence, $(\Delta_n)_{n\in\nset}$ converges
to $0$ for the strong operator topology.

\noindent\textbf{Step~2}. Let $n\in\nset$. Since $\tilde{\Delta}_n\in\cS_1(\cH_0)$ and
$\mathbb{J}_s\in\cL_b(\cH_0)$, we have
$\Delta_n\in\cS_1(\cH_0)$. Consider the polar decomposition of
$\Delta_n$, that is $\Delta_n = V_n \abs{\Delta_n}$ where $V_n$ is a
partial isometry with initial space
$\ker(\Delta_n)^\perp = \crange(\abs{\Delta_n})$ and final space
$\crange(\Delta_n)$ (see \S 3.9 in \cite{conwaycourseope}). Since
$\Delta_n$ is autoadjoint we have
$\ker(\Delta_n)^\perp = \crange(\Delta_n)$ and we get that
$\crange(\abs{\Delta_n}) = \crange(\Delta_n) \subset
\crange(\mathbb{J}_s) = \cspan{\phi_k, k \in \nset}$. Hence
\begin{equation}\label{eq:trace-norm-delta}
\norm{\Delta_n} = \tr(\abs{\Delta_n}) = \sum_{k\in\nset}
\pscal{\abs{\Delta_n}\phi_k}{\phi_k}_{\cH_0} = \sum_{k\in\nset}
\pscal{\Delta_n\phi_k}{V_n\phi_k}_{\cH_0} \;,
\end{equation}
where the last equality comes from the fact that 
$\abs{\Delta_n} = V_n^\adjoint \Delta_n$.

Now, note that for all $m\geq 1$,
$$
  \sum_{k=0}^{m-1}
\lrav{  \pscal{\Delta_n\phi_k}{V_n\phi_k}_{\cH_0}}
  \leq   \sum_{k=0}^{m-1} \norm{\Delta_n\phi_k}_{\cH_0}\;,
$$
which converges to zero by Step~1.
Thus,
  \begin{equation}\label{eq:trace-m}
\text{for all $m\geq1$,}\;\lim_{n\to\infty} \sum_{k=0}^{m-1}
\pscal{\Delta_n\phi_k}{V_n\phi_k}_{\cH_0} = 0\;.
\end{equation}
On the other hand, using the fact that $\Delta_n$ is hermitian and (\ref{eq:yy-sob-bound}),
we have, for all $n,k\in\nset$,
\begin{align*}
\pscal{\Delta_n\phi_k}{V_n\phi_k}_{\cH_0}
=\pscal{\phi_k}{\Delta_nV_n\phi_k}_{\cH_0}
=\pscal{\phi_k}{\mathbb{J}_s
  \tilde\Delta_n V_n\phi_k}_{\cH_0}
&=
                                         s_{k}^{-1}\,\pscal{\phi_k}{\tilde\Delta_nV_n\phi_k}_{\cH_0} \\
&= s_{k}^{-1}\, \pscal{\tilde\Delta_n \phi_k}{V_n\phi_k}_{\cH_0}\;.
\end{align*}
It follows that, for all $m\geq 1$,
$$
\sum_{k\geq m}
\lrav{ \pscal{\Delta_n\phi_k}{V_n\phi_k}_{\cH_0}}
\leq  s_m^{-1}\,  \sum_{k\in\nset}\lrav{  \pscal{
    \tilde\Delta_n\phi_k}{V_n\phi_k}_{\cH_0} }
\leq s_m^{-1}\, \norm{\tilde\Delta_n}_1 \;,
$$
where we used \cite[Corollary~18.12]{conwaycourseope}. Since
$\norm{\tilde\Delta_n}_1\leq1$ and $s^{-1}_m$ converges to 0, we obtain
that
$$
\lim_{m\to\infty}\sup_{n\in\nset}\sum_{k\geq m}
\lrav{ \pscal{\Delta_n\phi_k}{V_n\phi_k}_{\cH_0} }=0\;.
$$
This with \eqref{eq:trace-norm-delta} and \eqref{eq:trace-m} concludes the second and final step.
\end{proof}
We can now prove \Cref{thm:periodo-func}.
\begin{proof}[\bf Proof of \Cref{thm:periodo-func}]
  By the polarization formula we can write
  $\tilde{\cQ}_{I_n^{X_n^c}}^{(L,R)}$ as a linear combination of
  $\tilde{\cQ}_{I_n^{X_n^c}}^{(W,W)}$ with $W$ in
  $\lrcb{L+R,L-R,L+\rmi R,L-\rmi R}$. The same formula holds for
  expressing $\tilde{\cQ}_{\nu_X}^{(L,R)}$ using
  $\tilde{\cQ}_{\nu_X}^{(W,W)}$ with the same $W$'s. Hence, to obtain
  the claimed result, it suffices to show that, for all $W\in\lrcb{L+R,L-R,L+\rmi R,L-\rmi R}$, we have
  \begin{equation}
    \label{eq:unif-periodo-mapping-convergence-D}
\lim_{n\to\infty}\tilde{\cQ}_{I_n^{X_n^c}}^{(W,W)}=\tilde{\cQ}_{\nu_X}^{(W,W)}\quad\text{uniformly
  in}\quad\cC\lr{\Theta,\cS_1\lr{\cH_0}}\;,\quad\as
\end{equation}
So, take $W\in\lrcb{L+R,L-R,L+\rmi R,L-\rmi R}$, and let us
show~(\ref{eq:unif-periodo-mapping-convergence-D}).

By assumption on $L$ and $R$, we have
$W\in\cC\lr{\Theta\times\tore,\cL_b(\cH_0)}$ and, using Condition~\ref{item:LR-sob-assump},
$W_s\in\cC\lr{\Theta\times\tore,\cL_b(\cH_0^s)}$ where
$W_s(\theta,\lambda)=W(\theta,\lambda)_{|\cH_0^s}$ for all
$(\theta,\lambda)\in\Theta\times\tore$. By Condition~\ref{item:X-sob-assump}, we can
apply \Cref{lem:periodo-bounded-general} with $\cI_0=\cH_0^s$, and obtain
that, $\as$,
\begin{equation}
  \label{eq:InX-L11-compact}
  \text{there exists $r_1>0$ such that}\quad
  \set{I_n^{X_n^c}}{n\geq1}\subset B_{1,1}(r_1,\cH_0,\cH_0^s)\;.
\end{equation}
Now let us define $\mathrm{J}_s$, $U_s$ and $\mathbb{J}_s$ as in
  \Cref{lem:JsLemma}. Applying these definitions carefully and using
  the fact that $UU^\adjoint$ is the orthogonal projection onto
  $\range(U_s) = \cH_0^s$, it straightforwardly 
  yields that,  for all
  $(\theta,\lambda)\in\Theta\times\tore$, and $x\in\cH_0^s$,
  $$
  \mathbb{J}_s\,U_s^\adjoint\, W_s(\theta,\lambda)\, x=\mathrm{J}_s\,W_s(\theta,\lambda)\, x=W(\theta,\lambda)\, x\;.
  $$
  Thus, for all $n\in\nset$, $\as$, $I_n^X\in L^{1,1}\lr{\cH_0,\cH_0^s}$,
  and for all $\theta\in\Theta$,
  \begin{equation}
    \label{eq:QI_n-Hos}
\tilde{\cQ}_{I_n^X}^{(W,W)}(\theta)=\mathbb{J}_s\,\tilde{\cQ}_{I_n^X}^{(U_s^\adjoint\,
  W_s,W)}(\theta)\;.    
  \end{equation}
  Observe that $W_s\in\cC\lr{\Theta\times\tore,\cL_b(\cH_0^s)}$
  immediately implies that
  $U_s^\adjoint\,W_s(\theta,\cdot)\in\mathbb{F}_b\lr{\cH_0^s,\cH_0}$. Thus,
  for all $\theta$, we can apply \Cref{lem:cont-quad-func-general}
  with $\cI_0=\cH_0^s$ and $\cG_0=\cJ_0=\cH_0$,
  $L=U_s^\adjoint\,W_s(\theta,\cdot)$, $R=W$, $g=I_n^X$ and
  $\mu=\lebtore$, which, with~(\ref{eq:InX-L11-compact}), gives us
  that, $\as$: for all $\theta\in\Theta$, there exists $r>0$ such that
  $\norm{\tilde{\cQ}_{I_n^X}^{(U_s^\adjoint\,W_s,W)}(\theta)}_1\leq
  r$, where, here, $\norm{\cdot}_1$ denotes the trace-class norm in
  $\cS_1\lr{\cH_0}$. With~(\ref{eq:QI_n-Hos}) and the fact that
  $\tilde{\cQ}_{I_n^X}^{(W,W)}(\theta)$ is an hermitian operator for all $\theta$,
  we get that, $\as$: for all $\theta\in\Theta$, there exists $r>0$
  such that $\tilde{\cQ}_{I_n^X}^{(W,W)}(\theta)$ belongs to the set
  $B=r\,B_1$, with $B_1$ defined as in
  \Cref{prop:compactS1ballinS1}. Since $B$ is compact by
  \Cref{prop:compactS1ballinS1}, we can apply
  \Cref{thm:periodo-func-general} with $L=R:=W$ and we
  obtain~(\ref{eq:unif-periodo-mapping-convergence-D}), which
  concludes the proof.
\end{proof}

\subsection{Proof of \Cref{thm:1}}
\label{sec:proof-thm:1}

Hereafter we prove~(\ref{eq:cv-as-emp-cov}) by relying on
\Cref{thm:periodo-func}. Of course the uniform convergence in $\theta$
does not matter in this case and it could be proven by relying
directly on \Cref{prop:compactS1ballinS1}, by mimicking the argument
used in the proof of \Cref{thm:periodo-func} to show that
$\lr{\hat{\Gamma}_n(h)}_{n\in \nset}$ remains in a compact subset of
$\cS_1(\cH_0)$, $\as$ (since the convergence of $\hat{\Gamma}_n(h)$ to
$\Gamma(h)=\cov{X_h}{X_0}$ in weak operator topology is obvious under
the assumptions of \Cref{thm:1}).

\begin{proof}[\bf Proof of \Cref{thm:1}]
  Let $h$ be a given lag in
  $\zset$. The claimed result follows from \Cref{thm:periodo-func}
  with $L(\theta,\lambda)=\rme^{\rmi\,h\,\lambda}\Id_{\cH_0}$ and
  $R(\theta,\lambda)=\Id_{\cH_0}$, with $\theta$ being an arbitrary
  point and $\Theta$ the corresponding singleton. Indeed, with these
  definitions, the convergence~(\ref{eq:cv-as-emp-cov}) can be
  rewritten as~(\ref{eq:unif-periodo-mapping-convergence}). Thus, we
  only to show that $(X_t)_{t\in\zset}$ and the above defined $L$ and
  $R$ satisfy the assumptions of
  \Cref{thm:periodo-func}. Obviously~\ref{item:ergo-assump}
  and~\ref{item:sp_density-assump} are satisfied. As
  for Conditions~\ref{item:s-sob-assump} and~\ref{item:X-sob-assump},
  they follow from ~\ref{item:ergo-assump}
  and~\ref{item:sp_density-assump} by \Cref{lem:A1A2implyA4}, for well
  chosen $s$ and $\lr{\phi_k}_{k\in\nset}$.  Now, defining $L_s$ and
  $R_s$ as in \Cref{thm:periodo-func} with $L$ and $R$ as above, we
  get $L_s(\theta,\lambda)=\rme^{\rmi\,h\,\lambda}\Id_{\cH_0^s}$ and
  $R_s(\theta,\lambda)=\Id_{\cH_0^s}$, and $L,R,L_s$ and $R_s$
  obviously satisfy the assumptions of the theorem (including
  Condition~\ref{item:LR-sob-assump}). Hence \Cref{thm:periodo-func}
  applies and the proof is finished.
\end{proof}

\subsection{Proofs of \Cref{sec:examples-est-application}}
\subsubsection{Preliminary results}
\begin{lemma}
  \label{lem:lem-predictor}
  Let $\cH_0$ be a separable Hilbert space and $p,q$ be two
  non-negative integers.  Let $D\in\cL_b(\cH_0)$,
  $\mapol\in\mathcal{P}_q^\dagger(\cH_0)$ and
  $\arpol\in\mathcal{P}_p^\dagger(\cH_0)$. Define
  $\fiarmapred:\cset\setminus[1,\infty)\to\cL_b(\cH_0)$
  by~(\ref{eq:best-pred-fiarma-def}). Then,
  $\mapol^{-1}:z\mapsto\lrb{\mapol(z)}^{-1}$, $\arpol^{-1}:z\mapsto\lrb{\arpol(z)}^{-1}$ and
  $\fiarmapred$ are all holomorphic functions on the open unit disk $\unitdisk$ onto
  $\cL_b(\cH_0)$. Moreover $\mapol^{-1}$ and $\arpol^{-1}$ are
  continuous on the closed unit disk $\overline{\unitdisk}$ and
  $\fiarmapred$ is continuous over  $\overline{\unitdisk}\setminus\{1\}$. 
\end{lemma}
\begin{proof}
  By holomorphic in \Cref{lem:lem-predictor} we mean
  the same as in \cite[Definition~1.1.1]{gohberg-leiterer09}.  
  Since $\arpol$ and $\mapol$ are polynomials they are holomorphic in
  $\cset$. Because we assumed that they belong to
  $\mathcal{P}_p^\dagger(\cH_0)$, we further have that they
  are valued in the space of invertible operators on
  $\overline{\unitdisk}$ and so the inverted polynomials
  $z\mapsto\lrb{\mapol(z)}^{-1}$ and $z\mapsto\lrb{\arpol(z)}^{-1}$
  are holomorphic on $\unitdisk$ and continuous on
  $\overline{\unitdisk}$. Since the principal logarithmic function is
  holomorphic on $\cset\setminus[1,\infty)$, so is
  $z\mapsto(1-z)^D$. The result follows.
\end{proof}

\begin{lemma}
  \label{lem:lem-predictor2}
  Let $\cH_0$ be a separable Hilbert space and $p,q$ be two
  non-negative integers.  Let $D\in\cN(\cH_0)$
  $\mapol\in\mathcal{P}_q^\dagger(\cH_0)$ and
  $\arpol\in\mathcal{P}_p(\cH_0)$. Define
  $\fiarmapred:\cset\setminus[1,\infty)\to\cL_b(\cH_0)$
  by~(\ref{eq:best-pred-fiarma-def}). Let moreover
  $X\in\processtransfer{\fracintoptransfer{D}}(\Omega,\cF,\PP)$ and
  denote by $g_X$ its spectral operator density with respect to a
  non-negative measure $\mu$ on $\lr{\tore,\btore}$. 
  We assume that
  $\mu$ has no mass at the origin, $\mu\lr{\lrcb{0}}=0$.
  Then, we have
  \begin{align}
    \label{eq:uniform-cv-small-freq1}
&    \int_{-\pi/3}^{\pi/3}\;\sup_{0\leq\rho\leq1}\norm{\Psi(\rho,\lambda)g_X(\lambda)\Psi^\adjoint(\rho,\lambda)}_1\;\mu(\rmd\lambda)<\infty\;,
  \end{align}
  in the two following cases:
  \begin{align}
    \label{eq:uniform-cv-small-freq11}
  \text{if}\quad  & \Psi(\rho,\lambda):=\lr{1-\rho\rme^{-\rmi\lambda}}^{-D}-\lr{1-\rme^{-\rmi\lambda}}^{-D}\;.\\
  \label{eq:uniform-cv-small-freq2}
  \text{or if}\quad  & \Psi(\rho,\lambda):=\lr{\fiarmapred\lr{\rho\rme^{-\rmi\lambda}}-\fiarmapred\lr{\rme^{-\rmi\lambda}}}\,\lr{1-\rme^{-\rmi\lambda}}^{-D}\;.
  \end{align}
\end{lemma}
\begin{proof}
  We first  consider $\Psi(\rho,\lambda)$ as
  in~(\ref{eq:uniform-cv-small-freq11}). In this case, for
  all $0\leq\rho\leq1$ and $\lambda\neq0$,
  $$
  \norm{\Psi(\rho,\lambda)g_X(\lambda)\Psi^\adjoint(\rho,\lambda)}_1=\norm{\Psi(\rho,\lambda)\lr{g_X(\lambda)}^{1/2}}_2^2\leq
  4\,\sup_{0\leq\rho\leq1}\norm{\lr{1-\rho\rme^{-\rmi\lambda}}^{-D}\,\lr{g_X(\lambda)}^{1/2}}_2^2\;.
  $$
  Thus, the
  bound~(\ref{eq:uniform-cv-small-freq1}) is implied by
  \begin{align}
    \label{eq:uniform-cv-small-freq11bis}
  \int_{-\pi/3}^{\pi/3}\sup_{0\leq\rho\leq1}\norm{\lr{1-\rho\rme^{-\rmi\lambda}}^{-D}\,\lr{g_X(\lambda)}^{1/2}}_2^2\;\mu(\rmd\lambda)<\infty\;,
  \end{align}
  Since $D$ is assumed to be normal, we can proceed as in the proof of
  \Cref{thm:cns-fi-operator-functional-fiarma} and use its singular value function
  $\mathrm{d}$ on $\cG_0:=L^2(\Vset, \Vsigma, \xi)$ and decomposition
  operator $U$ so
  that~(\ref{eq:uniform-cv-small-freq11bis}) is implied by
  \begin{equation*}
    \int _{\Vset^2\times(-\pi/3,\pi/3)}\;\sup_{0\leq\rho\leq1} \left|\lr{1-\rho\rme^{-\rmi\lambda}}^{-\mathrm{d}(v)}\,\kernelope{h}(v,v';\lambda)\right|^2\;\xi(\rmd
v)\xi(\rmd v')\mu(\rmd\lambda)<\infty\;,
\end{equation*}
where $\kernelope{h}$ denote the $\tore$-joint kernel
  function of $h$ such that $
  h(\lambda)[h(\lambda)]^\adjoint=U\,g_X(\lambda)\,U^\adjoint
  \quad\text{for }\mu\text{-}\mae\;\lambda\in\tore\;. 
  $
  Now, by~\Cref{lem:integral-1minusexp-fiarma-sup-rho}, using that $\mathrm{d}$ is
  bounded over $\Vset$, the previous condition holds if
  \begin{equation}
    \label{eq:uniform-cv-small-freq11-CS}
    \int _{\Vset^2\times(-\pi/3,\pi/3)}\; |\lambda|^{-2\Re_+(\mathrm{d}(v))} \,\lrav{\kernelope{h}(v,v';\lambda)}^2\;\xi(\rmd
v)\xi(\rmd v')\mu(\rmd\lambda)<\infty\;.
\end{equation}
On the other hand, we assumed that
$X\in\processtransfer{\fracintoptransfer{D}}(\Omega,\cF,\PP)$, which,
by \Cref{thm:cns-fi-operator-functional-fiarma} is equivalent to
have~(\ref{eq:cns-fi-operator-functional-new-fiarma-pos}), which
implies~(\ref{eq:uniform-cv-small-freq11-CS}) (since
$\mu(\{0\})=0$). We thus proved~(\ref{eq:uniform-cv-small-freq1}) for
$\Psi(\rho,\lambda)$ as in~(\ref{eq:uniform-cv-small-freq11}).

  We now consider $\Psi(\rho,\lambda)$ as in~(\ref{eq:uniform-cv-small-freq2}).
  By~(\ref{eq:best-pred-fiarma-def}), for all
  $\lambda\in\tore\setminus\{0\}$ and $0\leq\rho\leq1$, the
  difference $\fiarmapred\lr{\rho\rme^{-\rmi\lambda}}-\fiarmapred\lr{\rme^{-\rmi\lambda}}$ can be written as
  $$ 
\lrb{\mapol(\rme^{-\rmi\lambda})}^{-1}\,\arpol(\rme^{-\rmi\lambda})\,(1-\rme^{-\rmi\lambda})^{D}  -\lrb{\mapol(\rho\rme^{-\rmi\lambda})}^{-1}\,\arpol(\rho\rme^{-\rmi\lambda})\,(1-\rho\rme^{-\rmi\lambda})^{D}\;.
  $$
  By \Cref{lem:lem-predictor}, $\mapol^{-1}$ and $\arpol$ are
  continuous on $\overline{\unitdisk}$, hence have bounded operator
  norms over  $\overline{\unitdisk}$. Thus, to get~(\ref{eq:uniform-cv-small-freq1}), it is thus sufficient to
  show that
  \begin{align}
    \label{eq:uniform-cv-small-freqCS1}
&    \int_{-\pi/3}^{\pi/3}\;\sup_{0\leq\rho\leq1}\norm{\tilde{\Psi}(\rho,\lambda)g_X(\lambda)\tilde{\Psi}^{\adjoint}(\rho,\lambda)}_1\;\mu(\rmd\lambda)<\infty\;,\\
    \label{eq:uniform-cv-small-freqCS2}
  \text{where}\quad  & \tilde{\Psi}(\rho,\lambda):=\lr{1-\rho\rme^{-\rmi\lambda}}^{D}\,\lr{1-\rme^{-\rmi\lambda}}^{-D}\;.
  \end{align}
  Since $D$ is assumed to be normal, we can proceed as in the proof of
  \Cref{thm:cns-fi-operator-functional-fiarma} and use its singular value function
  $\mathrm{d}$ on $\cG_0:=L^2(\Vset, \Vsigma, \xi)$ and decomposition
  operator $U$ so
  that~(\ref{eq:uniform-cv-small-freqCS1}) is implied by
  \begin{equation}
    \label{eq:uniform-cv-small-freq-CSbis}
    \int _{\Vset^2\times(-\pi/3,\pi/3)}\;\sup_{0\leq\rho\leq1} \left|\lr{\frac{1-\rho\rme^{-\rmi\lambda}}{1-\rme^{-\rmi\lambda}}}^{\mathrm{d}(v)}\kernelope{h}(v,v';\lambda)\right|^2\;\xi(\rmd
v)\xi(\rmd v')\mu(\rmd\lambda)<\infty\;.
  \end{equation}
  Now, by~\Cref{lem:integral-1minusexp-fiarma}
  and~\Cref{lem:integral-1minusexp-fiarma-sup-rho}, we have for all $z\in\cset$
  and $\lambda\in[-\pi/3,\pi/3]\setminus\{0\}$,
  $$
  \sup_{0\leq\rho\leq1}\lrav{\lr{\frac{1-\rho\rme^{-\rmi\lambda}}{1-\rme^{-\rmi\lambda}}}^{z}}
\leq \left(2\pi/(3\sqrt{3})\right)^{\Re_-(z)}\,\left(\pi/2\right)^{\Re_-(-z)} \;|\lambda|^{-\Re_+(z)}\; \rme^{\pi
   \abs{\Im(z)}}\,  \;.
 $$
 Plugging this bound in~(\ref{eq:uniform-cv-small-freq-CSbis}) and
 using that $\mathrm{d}$ is bounded on $\Vset$, we get
 that~(\ref{eq:uniform-cv-small-freq-CSbis}) is again implied
 by~(\ref{eq:uniform-cv-small-freq11-CS}). Hence we
 proved~(\ref{eq:uniform-cv-small-freq1}) in the case given by~(\ref{eq:uniform-cv-small-freq2}).
\end{proof}
In the following, for any positive integer $p$, we endow the set of
polynomials of degree less than of equal to $p$ (or any
of its subsets  $\mathcal{P}_p(\cH_0)$, $\mathcal{P}_p^\dagger(\cH_0)$ or
$\mathcal{P}_p^\ast(\cH_0)$) with the max of the $\norm{\cdot}_\infty$-norms
of its $\cL_b(\cH_0)$ coefficients. For instance if
$\bbpsi(z):=\sum_{k=0}^p A_k z^k$ we denote
$$
\norm{\bbpsi}=\max\set{\norm{A_k}_\infty}{k=1,\dots,p} \;.
$$
It is straightforward to show that the convergence of a
$\mathcal{P}_p(\cH_0)$-valued sequence in the obtained Banach space is
equivalent to the uniform convergence of this sequence in
$\cC\lr{\unitcircle,\cL_b(\cH_0)}$.  In particular the continuity of
$(\theta,\lambda)\mapsto\arpol_\theta\lr{\rme^{-\rmi\lambda}}$ and
$(\theta,\lambda)\mapsto\mapol_\theta\lr{\rme^{-\rmi\lambda}}$ on
$\Theta\times\tore$ onto $\cL_b(\cH_0)$ assumed
in~\ref{item:fiarma-mode} imply the convergence of
$\theta\mapsto\arpol_\theta$ and $\theta\mapsto\mapol_\theta$ on
$\Theta$ onto $\mathcal{P}_p(\cH_0)$.

We have the following lemma.
\begin{lemma}
  \label{lem:prediction3}
  Let $\cH_0$ be a separable Hilbert space and $p,q$ be two
  non-negative integers. Let $\fiarmapred$ be defined 
  by~(\ref{eq:best-pred-fiarma-def}). Then, for all
  $(\mapol,\arpol,D)\in\mathcal{P}_q^\dagger(\cH_0)\times\mathcal{P}_p(\cH_0)\times\cL_b(\cH_0)$
  and all $k\in\nset$,
\begin{equation}
  \label{eq:best-pred-fiarma-coeff-def}
  \fiarmapredcoef{k} :=  \frac1{2\rmi\pi}\int_{z\in\cset,|z|=\rho}
  \fiarmapred\lr{z}\,z^{-k-1}\;\rmd z
\end{equation} 
is well defined as $\cL_b(\cH_0)$-valued  Bochner integral for any $\rho\in(0,1)$ and does not
depend on $\rho$. Moreover, the following assertions
hold.
\begin{enumerate}[label=(\roman*)]
\item \label{item:lem:prediction31} For all
  $(\mapol,\arpol,D)\in\mathcal{P}_q^\dagger(\cH_0)\times\mathcal{P}_p(\cH_0)\times\cL_b(\cH_0)$
  and $z\in\unitdisk$,
\begin{equation}
  \label{eq:best-pred-fiarma-coeff-insideU}
  \fiarmapred\lr{z} = \sum_{k=1}^\infty  \fiarmapredcoef{k}\, z^k \;\;.
\end{equation}
\item \label{item:lem:prediction32} For any $k\geq1$,
$(\mapol,\arpol,D)\mapsto \fiarmapredcoef{k}$ is continuous on
$\mathcal{P}_q^\dagger(\cH_0)\times\mathcal{P}_p(\cH_0)\times\cL_b(\cH_0)$.
\item \label{item:lem:prediction33}  For all
  $(\mapol,\arpol,D)\in\mathcal{P}_q^\dagger(\cH_0)\times\mathcal{P}_p(\cH_0)\times\cN^\dagger(\cH_0)$
  and $z\in\overline{\unitdisk}\setminus\lrcb{1}$,
\begin{equation}
  \label{eq:best-pred-fiarma-coeff-in-closure-U}
  \fiarmapred\lr{z} = \sum_{k=1}^\infty  \fiarmapredcoef{k}\, z^k \;\;.
\end{equation}
\item \label{item:lem:prediction34} For any compact subset
$K\subset\mathcal{P}_q^\dagger(\cH_0)\times\mathcal{P}_p(\cH_0)\times\lr{\lrcb{0}\cup\cN^\dagger(\cH_0)}$,
we have
   \begin{equation}
     \label{eq:unfi-bound-abs-cv}
     \sum_{k=1}^\infty\sup_{(\mapol,\arpol,D)\in
       K}\norm{\fiarmapredcoef{k}}_{\infty}<\infty \; .
   \end{equation}
\end{enumerate}
 \end{lemma}
 \begin{proof}
   Recall that $\fiarmapred$ is defined for all $(\mapol,\arpol,D)\in\mathcal{P}_q^\dagger(\cH_0)\times\mathcal{P}_p(\cH_0)\times\cL_b(\cH_0)$
   by~(\ref{eq:best-pred-fiarma-def}) as a holomorphic function defined on
   $\cset\setminus[1,\infty)$ onto the Banach space
   $\cL_b(\cH_0)$.  Using \Cref{lem:lem-predictor},
   and \cite[Theorem~1.8.5]{gohberg-leiterer09}, we can expand
   $\fiarmapred$ as a power series on the open unit disk $\unitdisk$, 
   that is~(\ref{eq:best-pred-fiarma-coeff-insideU}) holds with $\fiarmapredcoef{k}$ (well) defined
   by~(\ref{eq:best-pred-fiarma-coeff-def}) for any $\rho\in(0,1)$. Note that the sum in the right-hand
side of~(\ref{eq:best-pred-fiarma-coeff-insideU}) starts at $k=1$ because $\fiarmapredcoef{0}=0$ since, by the Cauchy Formula (see
\cite[Theorem~1.5.1]{gohberg-leiterer09}), we have
$\fiarmapredcoef{0}=\fiarmapred(0)$, which is the null operator
following~(\ref{eq:best-pred-fiarma-def}) and
$\arpol(0)=\arpol(0)=1^D=\Id_{\cH_0}$. 

   Assertion~\ref{item:lem:prediction32} follows from~(\ref{eq:best-pred-fiarma-coeff-def}) by dominated
   convergence, since $(z,\mapol,\arpol,D)\mapsto\fiarmapred\lr{z}$ is
   continuous on
   $\unitdisk\times\mathcal{P}_q^\dagger(\cH_0)\times\mathcal{P}_p(\cH_0)\times\cL_b(\cH_0)$
   by~(\ref{eq:best-pred-fiarma-def}) and \Cref{lem:lem-predictor}.

Let us now prove Assertions~\ref{item:lem:prediction33}
and~\ref{item:lem:prediction34}. In fact,
Assertions~\ref{item:lem:prediction32}
and~\ref{item:lem:prediction34} imply that the right-hand side
of~(\ref{eq:best-pred-fiarma-coeff-in-closure-U}) is continuous on
$\overline{\unitdisk}$. Since, by \Cref{lem:lem-predictor}, the left-hand side is  continuous on
$\overline{\unitdisk}\setminus\lrcb{1}$, with
Assertion~\ref{item:lem:prediction31}, we conclude that
we get both  Assertions~\ref{item:lem:prediction33}
and~\ref{item:lem:prediction34} by proving the bound~(\ref{eq:unfi-bound-abs-cv}). Let $K$ be a
compact subset of
$\mathcal{P}_q^\dagger(\cH_0)\times\mathcal{P}_p(\cH_0)\times\{0\}$
or of
$\mathcal{P}_q^\dagger(\cH_0)\times\mathcal{P}_p(\cH_0)\times\cN^\dagger(\cH_0)$.
Then, there exists $r>1$ such that for all $(\mapol,\arpol,D)\in
   K$, $\mapol$ does not vanish over the open disk of radius
   $r$. It follows that  for all $(\mapol,\arpol,D)\in
   K$, $z\mapsto\lrb{\mapol(z)}^{-1}\,\arpol(z)$
 is a power series with a radius of convergence at least equal to
 $r$ and that, for any $\rho_1\in(r^{-1},1)$, there exists $c_1>0$ 
 such that, for all  $(\mapol,\arpol,D)\in K$ and  $z\in
 r\,\unitdisk$,
 \begin{equation}
   \label{eq:arma-power-series}
\lrb{\mapol(z)}^{-1}\,   \arpol(z) = \Id_{\cH_0}+\sum_{k=1}^\infty \cop_k\;
 z^k\quad\text{with}\quad \norm{\cop_k}_\infty \leq c_1\,\rho_1^k
\end{equation}
If
$K\subset\mathcal{P}_q^\dagger(\cH_0)\times\mathcal{P}_p(\cH_0)\times\{0\}$,
we have $\fiarmapredcoef{k}=-\cop_k$ and~(\ref{eq:unfi-bound-abs-cv}) follows.

We now consider the case where
$K\subset\mathcal{P}_q^\dagger(\cH_0)\times\mathcal{P}_p(\cH_0)\times\cN^\dagger(\cH_0)$.
Let $\sigma$ be an upper bound of $\norm{D}_\infty$ and $\rho$ a
lower bound of the smallest eigenvalue of $(D+D^\adjoint)/2$
over $(\mapol,\arpol,D)\in K$. Then, we have $\varrho>0$ by definition of
$\cN^\dagger(\cH_0)$ in~(\ref{eq:Dond-positive-long-memory}) and since
$K$ is compact.  Then, setting $N=D+\Id_{\cH_0}$,
Condition~(\ref{eq:Ncond-fiarma-infsup}) holds with
$\varrho=\rho+1>1$. 
and $\varsigma=\sigma+1$. Thus, by \Cref{lem:dev-1minusz-fiarma}, we
have, for all $z\in\overline{\unitdisk}\setminus\lrcb{1}$ and
$D\in\cN^\dagger(\cH_0)$,
   $$
   (1-z)^{D} = \sum_{k=0}^\infty   \lr{ \bop (k+1)^{-\Id_{\cH_0}-D}+ \aop^{\ast}_k} z^k \;,
   $$
   where $\norm{\bop}_{\infty}\leq C$ and
   $\norm{\aop^{\ast}_k}\leq C k^{-2-\rho}$ for some constant $C>0$ only
   depending on $\sigma$. Moreover, note that
   $\norm{(k+1)^{-\Id_{\cH_0}-D}}_{\infty}\leq(k+1)^{-1-\rho}$.
   By~(\ref{eq:best-pred-fiarma-def})  and
setting $\cop_{0}=\Id_{\cH_0}$, we obtain, for all
   $z\in\overline{\unitdisk}\setminus\{1\}$,
   $$
   \fiarmapred\lr{z} = - \sum_{\ell=1}^\infty  \lr{ \sum_{k=0}^\ell \cop_{\ell-k}\; \lr{
       \bop\, (k+1)^{-\Id_{\cH_0}-D}+ \aop^{\ast}_k}} z^\ell\;.
   $$
   By~(\ref{eq:best-pred-fiarma-def}), we thus have, for all $k\geq1$,
   $$
   \fiarmapredcoef{k} = -  \lr{ \sum_{k=0}^\ell \cop_{\ell-k}\; \lr{
       \bop \, (k+1)^{-\Id_{\cH_0}-D}+ \aop^{\ast}_k}}\;,
   $$
   and the previous bounds further yield
   $$
\sup_{(\mapol,\arpol,D)\in K}   \norm{\fiarmapredcoef{k}}_{\infty}\leq 2\,c_1\,C\,\sum_{k=0}^\ell \rho_1^{\ell-k}(k+1)^{-1-\rho}\;.
$$
Therefore, we obtain~(\ref{eq:unfi-bound-abs-cv}).
\end{proof}
Finally, the following lemma will be useful.
\begin{lemma}
  \label{lem:von-neuman}
    Let $\cH_0$ be a separable Hilbert space and
    $X=\lr{X_t}_{t\in\zset}$ be an ergodic stationary process defined
    on $(\Omega,\cF,\PP)$ valued
    in $\cH_0$ such that $\PEarg{\norm{X}_{\cH_0}^2}<\infty$. Then we
  have
$$
\lim_{n\to\infty}\PEarg{\norm{\frac1n\sum_{k=1}^nX_k
    -\PEarg{X}}_{\cH_0}^2} = 0\;.
$$
\end{lemma}
\begin{proof}
  The assumptions imply that $X$ is weakly stationary. Moreover, the
  space of shift-invariant elements in $\cH^X$ (where the shift is
  defined by $X_t\mapsto X_{t-1}$) is the null set, otherwise $X$
  would not be ergodic : take $V$ shift-invariant in the sense of
  $\cH^X$, then, for all $x\in\cH_0$, $\pscal{V}{x}_{\cH_0}$ is
  shift-invariant in the $\as[\PP^X]$ sense. See also
  \Cref{lem:l2-ergodic-zero-mass} for a more precise statement in the
  centered case (which can be assumed here without loss of
  generality). Therefore, the result simply follows from the von
  Neumann ergodic theorem (see \cite[Theorem~II.11]{reed_simon80}).
\end{proof}

\subsubsection{Proof of main results}
\begin{proof}[\bf Proof of \Cref{thm:fiarma-predictor}]
  By \Cref{def:fiarma-fiarma}, denoting by $\Sigma$ the covarariance
  operator of $Z$, $Y$ admits the spectral density
  $$
  f_Y(\lambda)= \lr{1-\rme^{-\rmi\lambda}}^{-D}[\arpol(\rme^{-\rmi\lambda})]^{-1}\mapol(\rme^{-\rmi\lambda})\Sigma\lr{\lr{1-\rme^{-\rmi\lambda}}^{-D}[\arpol(\rme^{-\rmi\lambda})]^{-1}\mapol(\rme^{-\rmi\lambda})}^\adjoint
  $$
  with respect to the normalized Lebesgue measure $\lebtore$. 
  Let $t\in\zset$. We first show that the right-hand side
  of~(\ref{eq:best-pred-fiarma}) is well defined, that is, that
  $\lambda\mapsto\rme^{\rmi\lambda t} \,\fiarmapred[\mapol,\arpol,D]\lr{\rme^{-\rmi\lambda}}$
  belongs to $\hat{\cH}^Y$. Since this mapping is continuous on
  $\tore\setminus\{0\}$ onto
  $\cL_b(\cH_0)$, by \Cref{prop:carac-processtransfer-discrete}, this
  is equivalent to have
  $$
  \int\norm{\fiarmapred[\mapol,\arpol,D]\lr{\rme^{-\rmi\lambda}}f_Y(\lambda)\fiarmapred[\mapol,\arpol,D]\lr{\rme^{-\rmi\lambda}}^\adjoint}_1\;\lebtore(\rmd\lambda)<\infty\;.
  $$
  By definition of $f_Y$, we thus have to show that
  \begin{equation}
    \label{eq:cond-predictor-wel-defined}
  \int\norm{\fiarmapred[\mapol,\arpol,D]\lr{\rme^{-\rmi\lambda}}\lr{1-\rme^{-\rmi\lambda}}^{-D}[\arpol(\rme^{-\rmi\lambda})]^{-1}\mapol(\rme^{-\rmi\lambda})\Sigma^{1/2}}_2^2\;\lebtore(\rmd\lambda)<\infty\;.    
  \end{equation}
  By definition of $\fiarmapred$
  in~(\ref{eq:best-pred-fiarma-def}), we have, for all $\lambda\in\tore\setminus\{0\}$,
  $$
  \lr{\Id_{\cH_0}-\fiarmapred[\mapol,\arpol,D]\lr{\rme^{-\rmi\lambda}}}\lr{1-\rme^{-\rmi\lambda}}^{-D}[\arpol(\rme^{-\rmi\lambda})]^{-1}\mapol(\rme^{-\rmi\lambda})
  = \Id_{\cH_0}\;.
  $$
  We thus get that, for all $\lambda\in\tore\setminus\{0\}$,
  $\fiarmapred[\mapol,\arpol,D]\lr{\rme^{-\rmi\lambda}}\lr{1-\rme^{-\rmi\lambda}}^{-D}[\arpol(\rme^{-\rmi\lambda})]^{-1}\mapol(\rme^{-\rmi\lambda})\Sigma^{1/2}$
  can be expressed as
  $$
  \lr{1-\rme^{-\rmi\lambda}}^{-D}[\arpol(\rme^{-\rmi\lambda})]^{-1}\mapol(\rme^{-\rmi\lambda})\Sigma^{1/2}-  \Sigma^{1/2}\;.
  $$
  Thus, since $\norm{\Sigma^{1/2}}_2^2=\norm{\Sigma}_1<\infty$,
  Condition~(\ref{eq:cond-predictor-wel-defined}) is implied by
  $$
  \int\norm{\lr{1-\rme^{-\rmi\lambda}}^{-D}[\arpol(\rme^{-\rmi\lambda})]^{-1}\mapol(\rme^{-\rmi\lambda})\Sigma^{1/2}}_2^2\;\lebtore(\rmd\lambda)<\infty\;.
  $$
  On the other hand, the square $\cS_2$-norm inside the previous
  integral is equal to $\norm{f_Y}_1$ which is $\lebtore$-integrable
  as a spectral density. We thus get that the right-hand side
  of~(\ref{eq:best-pred-fiarma}) is well defined and, in the
  following, we denote
  $$
  \hat{Y}_t=\int \rme^{\rmi\lambda t} \;
  \fiarmapred[\mapol,\arpol,D]\lr{\rme^{-\rmi\lambda}}\;
  \hat{Y}(\rmd\lambda)\;.
  $$
  It only remains to show the two following assertions.
  \begin{enumerate}[label=(\roman*)]
  \item \label{item:fiarma-predictor1} We have
    $\hat{Y}_t\in\cH^Y_{t-1}$.
  \item \label{item:fiarma-predictor2} We have $Y_t-\hat{Y}_t\perp\cH^Y_{t-1}$.
  \end{enumerate}
  Let us first prove
  Assertion~\ref{item:fiarma-predictor1}. By
  Assertion~\ref{item:lem:prediction31} in \Cref{lem:prediction3},
  we immediately have that, for all $\rho\in(0,1)$,
  $$
  \hat{Y}_t^{(\rho)}=\int \rme^{\rmi\lambda t} \;
  \fiarmapred[\mapol,\arpol,D]\lr{\rho\rme^{-\rmi\lambda}}\;
  \hat{Y}(\rmd\lambda)\in\cH^Y_{t-1}\;.
  $$
  To conclude Assertion~\ref{item:fiarma-predictor1}, it is thus
  sufficient to show that
  \begin{equation}
    \label{eq:Yrho-approx}
    \lim_{\rho\uparrow1}\hat{Y}_t^{(\rho)}={Y}_t
    \quad\text{in $\cM(\Omega,\cF, \cH_0, \PP)$.}
  \end{equation}
  By the Kolmogorov Gramian isometric theorem (see
  \cite[\Cref{thm:kolmo-isomorphism-thm}]{surveyREFnew}), setting
  $$
\Psi(\rho,\lambda):=\fiarmapred[\mapol,\arpol,D]\lr{\rho\rme^{-\rmi\lambda}}-\fiarmapred[\mapol,\arpol,D]\lr{\rme^{-\rmi\lambda}}\;,
  $$
we can write, for any $\eta\in(0,\pi)$ and $\rho\in(0,1)$,
  \begin{align*}
    \PEarg{\norm{\hat{Y}^{(\rho)}_t-\hat{Y}_t}_{\cH_0}^2}=&\int\norm{\Psi(\rho,\lambda)f_Y(\lambda)\Psi^\adjoint(\rho,\lambda)}_1\; \lebtore(\rmd\lambda)& \\
    \leq&\lr{\int\norm{f_Y}_1\rmd\lebtore}\,\sup_{\lambda\in\tore\setminus[-\eta,\eta]}\norm{\Psi(\rho,\lambda)}_{\infty}^2\\
    &+
    \int_{-\eta}^\eta\sup_{0\leq\rho\leq1}\norm{\Psi(\rho,\lambda)f_Y(\lambda)\Psi^\adjoint(\rho,\lambda)}_1\;\lebtore(\rmd\lambda)\;.    
  \end{align*}
  By \Cref{lem:lem-predictor} the first term of this bound tends to
  zero as $\rho\uparrow1$ for all $\eta\in(0,\pi)$. It
  thus only remain to check that the
  second term can be made arbitrarily small as $\eta\downarrow0$,
  which follows if  there exists $\eta>0$ such that
  \begin{align*}
&
                   \int_{-\eta}^\eta\sup_{0\leq\rho\leq1}\norm{\Psi(\rho,\lambda)f_Y(\lambda)\Psi^\adjoint(\rho,\lambda)}_1\;\lebtore(\rmd\lambda)<\infty\;.
  \end{align*}
  By definition of $f_Y$, setting $X$ as the  ARMA($p,q$) process defined
    by
    \begin{equation}
      \label{eq:X-intermediaire-prediction-theorem}
    \hat
    X(\rmd\lambda)=[\arpol(\rme^{-\rmi\lambda})]^{-1}\mapol(\rme^{-\rmi\lambda})\hat
    Z(\rmd\lambda)\;,
  \end{equation}
denoting by $f_X$ the density of $X$ with respect
    to $\lebtore$, the previous condition is equivalent to      
  \begin{align*}
&  \int_{-\eta}^\eta\sup_{0\leq\rho\leq1}\norm{\tilde{\Psi}(\rho,\lambda)f_X(\lambda)\tilde{\Psi}^\adjoint(\rho,\lambda)}_1\;\lebtore(\rmd\lambda)<\infty\;,\\
  \text{where}\quad  & \tilde{\Psi}(\rho,\lambda):={\Psi}(\rho,\lambda)\,\lr{1-\rme^{-\rmi\lambda}}^{-D}=\lr{\fiarmapred[\mapol,\arpol,D]\lr{\rho\rme^{-\rmi\lambda}}-\fiarmapred[\mapol,\arpol,D]\lr{\rme^{-\rmi\lambda}}}\,\lr{1-\rme^{-\rmi\lambda}}^{-D}\;.
  \end{align*}
  Since
  $X\in\processtransfer{\fracintoptransfer{D}}(\Omega,\cF,\PP)$ by
  definition of $Y$, \Cref{lem:lem-predictor2} gives us that the
  latter condition holds, which concludes the proof
  of~(\ref{eq:Yrho-approx}) and thus of
  Assertion~\ref{item:fiarma-predictor1}.

  We now prove  Assertion~\ref{item:fiarma-predictor2}. By definition
  of $\hat{Y}_t$,~(\ref{eq:best-pred-fiarma-def}) and~(\ref{eq:fiarma-def-spectral-fiarma}), we have
  \begin{align*}
  Y_t-\hat{Y}_t&=\int \rme^{\rmi\lambda t} \;
\lr{\Id_{\cH_0}-  \fiarmapred[\mapol,\arpol,D]\lr{\rme^{-\rmi\lambda}}}\;\hat{Y}(\rmd\lambda)\\
    &=\int  \rme^{\rmi\lambda t} \; \lrb{\mapol(\rme^{-\rmi\lambda})}^{-1}\,\arpol(\rme^{-\rmi\lambda})\,(1-\rme^{-\rmi\lambda})^{D}\; \hat{Y}(\rmd\lambda)\\
   &=\int  \rme^{\rmi\lambda t} \; \hat{Z}(\rmd\lambda)= Z_t\;.    
  \end{align*}
  Since $Z$ is a white noise we have $Z_t\bot\cH^Z_{t-1}$. To prove
  Assertion~\ref{item:fiarma-predictor2}, it thus only remains to show
  that $\cH^Y_{s}$ is included in $\cH^Z_{s}$ for all
  $s\in\zset$. Since we assumed
  $\arpol\in\mathcal{P}_p^\dagger(\cH_0)$, we have that $\arpol^{-1}$
  is holomorphic in an open domain that includes
  $\overline{\unitdisk}$ and thus can be written as a power series on
  the unit circle. It follows that the ARMA process $X$ defined
  by~(\ref{eq:X-intermediaire-prediction-theorem}) satisfies
  $\cH^X_s\subset\cH^Z_{s}$ for all $s\in\zset$. To conclude, we now
  prove that $\cH^Y_{s}\subset\cH^X_{s}$ for all $s\in\zset$.  Observe
  that, for all $s\in\zset$,
  $$
  Y_s=\int \rme^{\rmi\lambda s}\lr{1-\rme^{-\rmi\lambda}}^{-D}\;\hat{X}(\rmd\lambda)\;.
  $$
  Using that $z\mapsto\lr{1-z}^{-D}$ is holomorphic on
  $\cset\setminus[1,\infty)$, it can be expanded as a power series on
  $\unitdisk$ and it follows that, for
  all $s\in\zset$ and $\rho\in(0,1)$,
  $$
  Y^{(\rho)}_s=\int \rme^{\rmi\lambda s}
  \lr{1-\rho\rme^{-\rmi\lambda}}^{-D}\;\hat{X}(\rmd\lambda)\in \cH^X_s\;.
  $$
  Using the same trick as for the proof of
  Assertion~\ref{item:fiarma-predictor1}, we write, for any $\eta\in(0,\pi)$
  and $\rho\in(0,1)$,
  \begin{align*}
    \PEarg{\norm{Y_s-Y^{(\rho)}_s}_{\cH_0}^2}\leq&\lr{\int\norm{f_X}_1\rmd\lebtore}\,\sup_{\lambda\in\tore\setminus[-\eta,\eta]}\norm{\lr{1-\rho\rme^{-\rmi\lambda}}^{-D}-\lr{1-\rme^{-\rmi\lambda}}^{-D}}_{\infty}^2\\
    &+
    \int_{-\eta}^\eta\sup_{0\leq\rho\leq1}\norm{\Psi(\rho,\lambda)f_X(\lambda)\Psi^\adjoint(\rho,\lambda)}_1\;\lebtore(\rmd\lambda)\;,\\
  \text{where, here,}\quad  & \Psi(\rho,\lambda):=\lr{1-\rho\rme^{-\rmi\lambda}}^{-D}-\lr{1-\rme^{-\rmi\lambda}}^{-D}\;.    
  \end{align*}
  By continuity of $z\mapsto\lr{1-z}^{-D}$ on
  $\cset\setminus[1,\infty)$, the first term in the upper bound tends
  to zero for all $\eta$, while the second bound can be made
  arbitrarily small as $\eta\downarrow0$ as a consequence of
  \Cref{lem:lem-predictor2}. We thus get the claim
  $\cH^Y_{s}\subset\cH^X_{s}$ and the proof is concluded.
\end{proof}

\begin{proof}[\bf Proof of \Cref{prop:firama-pred-def}]
All the assertions follow straightforwardly from \Cref{lem:prediction3} and other
previous results. Observe indeed that~(\ref{eq:fiarmapred-def2})
immediately follows from the fact that $\fop^\dagger(\theta,\lambda) =
\Id_{\cH_0}  -
\fiarmapred[{D_{\theta},\arpol_\theta,\mapol_\theta}](\rme^{-\rmi\lambda})$.
As for the other claimed facts, here are some
details. 

Moreover, the continuity of
$\lr{D,\arpol,\mapol}\mapsto\fiarmapredcoef{k}$ (and thus
with~\ref{item:fiarma-mode}, that of
$\theta\mapsto\fiarmapredcoef[\theta]{k}$, and the
bound~(\ref{eq:unfi-bound-abs-cv}) gives us that
$\theta\mapsto\hat{Y}_t(\theta)$ is continuous on $\Theta$ onto
$\cH^Y_{t-1}$, hence the expectation in the right-hand side
of~(\ref{eq:fiarmacol-pred-best-error}) is continuous in $\theta$,
which shows that the inf is attained on a compact subset of
$\Theta$.

When the best predictor $\hat{Y}^\ast_t$ is well defined for one
$t\in\zset$, by weak stationarity of $Y$ it must be well defined for
all $t$. Then, for all $t\in\zset$ and all $\theta\in\Theta^\ast_Y$,
we have $\hat{Y}^\ast_t=\hat{Y}_t(\theta)$, $\as$. Of course, since
since $\zset$ is countable, we can exchange the $\as$ with the ``for
all $t\in\zset$''.  We can also exchange the $\as$ with the ``for all
$\theta\in\Theta^\ast_Y$'' as claimed in assertion~\ref{item:prop:firama-pred-def6} of the proposition because the
arguments above also give that, $\as$,
$\theta\mapsto\hat{Y}_t(\theta)$ is continuous on $\Theta$ onto
$\cH_0$, and consequently, is uniquely defined by its value on a dense
countable subset of (the compact set) $\Theta^\ast_Y$.

Finally, in the well-specified case, we apply
\Cref{thm:fiarma-predictor} and notice that, by
Assertion~(\ref{eq:best-pred-fiarma-coeff-in-closure-U}) of
\Cref{lem:prediction3}, $\hat{Y}_t(\theta)$ is the right-hand side
of~(\ref{eq:best-pred-fiarma}) with $(\mapol,\arpol,D)$ replaced by
$\lr{D_{\theta},\arpol_\theta,\mapol_\theta}$, which
gives~(\ref{eq:fiarmacol-pred-best-error-well-specified2}) and,
consequently,~(\ref{eq:fiarmacol-pred-best-error-well-specified}).
\end{proof}

\begin{proof}[\bf Proof of \Cref{thm:fiarma-predictor-esimation}]
  By \Cref{lem:prediction3}, we have that
  $\fop^\dagger\in \cC(\Theta\times\tore,\cL_b(\cH_0))$ with
  $\fop^\dagger$ defined in \eqref{eq:fiarmapred-def2}. Similarly, using
  Assertion~\ref{item:LR-sob-assump-est} in
  \Cref{thm:fiarma-predictor-esimation}, we also have
  $\fop_{s}^\dagger\in\cC(\Theta\times\tore,\cL_b(\cH_0^s))$,
  with
  $\fop_{s}^\dagger (\theta,\lambda):=\fop^\dagger(\theta,\lambda)_{|\cH_0^s}$.
  Hence Assertions~\ref{item:s-sob-assump},~\ref{item:X-sob-assump}
  and~\ref{item:LR-sob-assump} of \Cref{thm:periodo-func} hold with
  $L=R=\fop^\dagger$. Applying this theorem, with the fact that the trace is
  continuous on $\cS_1(\cH_0)$, we obtain that, $\as$,
  $$
\lim_{n\to\infty}  \Lambda_n =
\tr\lr{\tilde{\cQ}^{(\fop^\dagger,\fop^\dagger)}_{\nu_X}}\quad\text{uniformly in $\cC\lr{\Theta,\rset}$.}
$$
Now, observe that by~(\ref{eq:fiarmapred-def2}) and~(\ref{eq:fiarmacol-pred}), for all $\theta\in\Theta$,
$$
\tr\lr{\tilde{\cQ}^{(\fop^\dagger,\fop^\dagger)}_{\nu_X}(\theta)}=\PEarg{\norm{Y_0-\hat{Y}_0(\theta)}_{\cH_0}^2}\;.
$$
Therefore, with~(\ref{eq:theta-hat-def}) and~(\ref{eq:fiarmacol-pred-best-error}), we can write that, $\as$, 
\begin{align*}
  \limsup_{n\to\infty}\Lambda_n(\hat{\theta}_n)
  \leq \limsup_{n\to\infty}\inf_{\theta\in\Theta}\Lambda_n(\theta) 
  = \mathbb{E}^2\lr{Y,\fiarmacol} \;.
\end{align*}
Thus, $\as$, all accumulation points $\theta$ of the $\Theta$-valued sequence
$(\hat{\theta}_n)_{n\geq1}$ satisfy
$$
\PEarg{\norm{Y_0-\hat{Y}_0(\theta)}_{\cH_0}^2} \leq
\mathbb{E}^2\lr{Y,\fiarmacol}\;,\quad\text{which implies}\quad\theta\in\Theta^\ast_Y\;.
$$
Since $\Theta$ is compact, we
obtain~(\ref{eq:consitency-not-well-specified}).

We now prove~(\ref{eq:best-fiarmapred-not-well-specified-error}).
Let us define, for all $\theta\in\Theta$,
\begin{equation}
  \label{eq:pred-finite-obs-with-emp-mean-Y}
\mathbb{E}^{2}_{Y,\infty}\lr{\theta}= \PEarg{\norm{Y_0-\lr{\sum_{k=1}^\infty\fiarmapredcoef[\theta]{k}\;Y_{-k}}}_{\cH_0}^2}  \;.  
\end{equation}
By~(\ref{eq:pred-finite-obs-with-emp-mean}) and Minkowski's
inequality, we have, for all $(m,\theta)\in\cH_0\times\Theta$ and $n\geq1$,
\begin{align*}
  \lrav{\mathbb{E}_{Y,\infty}\lr{\theta}-\mathbb{E}_{X,n}\lr{m,\theta}}
  &\leq
    \lrav{\mathbb{E}_{Y,\infty}\lr{\theta}-\mathbb{E}_{Y,n}\lr{0,\theta}}+\norm{\lr{\Id_{\cH_0}-\sum_{k=1}^\infty\fiarmapredcoef[\theta]{k}}\,\lr{\PEarg{X_0}-m}}_{\cH_0}
    \;.
\end{align*}
We further have, for all $m\in\cH_0$ and
$n\geq1$,
\begin{align*}
 \sup_{\theta\in\Theta} \lrav{\mathbb{E}_{Y,\infty}\lr{\theta}-\mathbb{E}_{Y,n}\lr{0,\theta}}  &\leq
\sup_{\theta\in\Theta} \lr{\PEarg{\norm{\sum_{k=n+1}^\infty\fiarmapredcoef[\theta]{k}\;Y_{-k}}_{\cH_0}^2}}^{1/2}\\
  \\
&\leq
\lr{\sum_{k=n+1}^\infty\sup_{\theta\in\Theta}\norm{\fiarmapredcoef[\theta]{k}}_{\infty}}\;\lr{\PEarg{\norm{Y_{0}}_{\cH_0}^2}}^{1/2}\;,    
\end{align*}
which, by  \Cref{lem:prediction3}, converges to zero as $n\to\infty$.

On the other hand, we have, for all $m\in\cH_0$,
$$
\sup_{\theta\in\Theta}\norm{\lr{\Id_{\cH_0}-\sum_{k=1}^\infty\fiarmapredcoef[\theta]{k}}\,\lr{\PEarg{X_0}-m}}_{\cH_0}
\leq\lr{1+\sum_{k=1}^\infty\sup_{\theta\in\Theta}\norm{\fiarmapredcoef[\theta]{k}}_{\infty}}\;\lr{\PEarg{X_0}-m}\;,
$$
which, by \Cref{lem:prediction3} again, converges to zero as
$m\to\PEarg{X_0}$.
Consequently,~(\ref{eq:best-fiarmapred-not-well-specified-error})
follows if we can show that
\begin{align}
  \label{eq:eq:best-fiarmapred-not-well-specified-error1}
\lim_{n\to\infty} \mathbb{E}^2_{Y,\infty}(\hat\theta_n)= \mathbb{E}^2\lr{Y,\fiarmacol}\quad\as  \\
  \label{eq:eq:best-fiarmapred-not-well-specified-error2}
  \lim_{n\to\infty}\frac1n\sum_{k=1}^nX_k = \PEarg{X_0}\quad \as
\end{align}
\Cref{lem:as-mean-cv} gives
us~(\ref{eq:eq:best-fiarmapred-not-well-specified-error2}).
We now prove~(\ref{eq:eq:best-fiarmapred-not-well-specified-error1}).
Observe that, by~(\ref{eq:pred-finite-obs-with-emp-mean-Y}) and by
continuity of $\theta\mapsto\sum_k\fiarmapredcoef[\theta]{k}$ (see
\Cref{lem:prediction3}), we have that
$\theta\mapsto\mathbb{E}^2_{Y,\infty}(\theta)$ is continuous on
$\Theta$ onto $\rset_+$. Now, since $\Theta$ is compact, we have that, $\as$, $\lr{\ell_n}_{n\geq1}$
is a bounded sequence in $\rset_+$, where we set, for all $n\geq1$,
$\ell_n:=\mathbb{E}^{2}_{Y,\infty}(\hat\theta_n)$.
We now show that, $\as$, all accumulation points of this
sequence is in fact equal to the right-hand side
of~(\ref{eq:eq:best-fiarmapred-not-well-specified-error1}). 
This follows from the fact that, by compactness of $\Theta$, from any increasing sequence
$(n_j)_{j\in\nset}$ of positive integers, we can extract a subsequence
$(n'_j)_{j\in\nset}$ such that $\hat{\theta}_{n'_j}$
converges. Furthermore, by~(\ref{eq:consitency-not-well-specified}), $\as$,
the limit of this sequence must belong to $\Theta^\ast_Y$.  With
the continuity of $\theta\mapsto\mathbb{E}^2_{Y,\infty}(\theta)$
previously established, we conclude that, $\as$, all accumulation
points of  $\lr{\ell_{n}}_{n\geq1}$ is of the form
$\mathbb{E}^{2}_{Y,\infty}\lr{\theta}$ for some
$\theta\in\Theta^\ast_Y$, hence is equal to
$\mathbb{E}^2\lr{Y,\fiarmacol}$ by~(\ref{eq:fiarmacol-pred-best-error})
and definition of $\Theta^\ast_Y$ in~(\ref{eq:fiarmacol-pred-best-set}). 

Let us show the last assertion of the theorem. To this end, we suppose that $\hat{Y}_t^\ast$ is well defined and show
that~(\ref{eq:best-fiarmapred-not-well-specified-error-well-spec}) holds. We
first observe that, since $X^c_{n,k}=Y^c_{n,k}$,
$$
X_{n+1}-\hat{X}_{n+1,n}=Y_{n+1}-\sum_{k=1}^n\fiarmapredcoef[\hat{\theta}_n]{k}\;Y_{n+1-k}- \lr{\Id_{\cH_0}-\sum_{k=1}^n\fiarmapredcoef[\hat{\theta}_n]{k}}\,\lr{\frac1n\sum_{k=1}^nY_k}\;.
$$
Therefore, to
obtain~(\ref{eq:best-fiarmapred-not-well-specified-error-well-spec}), we only
need to show that
\begin{align}
  \label{eq:to-show-best-fiarmapred-not-well-specified-error1}
  \limsup_{n\to\infty}\PEarg{\norm{Y_{n+1}-\sum_{k=1}^n\fiarmapredcoef[\hat{\theta}_n]{k}\;Y_{n+1-k}}_{\cH_0}^2}\leq\mathbb{E}^2\lr{Y,\fiarmacol}\;,\\
  \label{eq:to-show-best-fiarmapred-not-well-specified-error2}
    \lim_{n\to\infty}\PEarg{\norm{\lr{\Id_{\cH_0}-\sum_{k=1}^n\fiarmapredcoef[\hat{\theta}_n]{k}}\,\lr{\frac1n\sum_{k=1}^nY_k}}_{\cH_0}^2} = 0 \;.
\end{align}
Let us start
with~(\ref{eq:to-show-best-fiarmapred-not-well-specified-error2}). By
\Cref{lem:prediction3}, we have
$$
\sup_{n\in\nset}\sup_{\theta\in\Theta}\norm{\Id_{\cH_0}-\sum_{k=1}^n\fiarmapredcoef[\theta]{k}}_{\infty}\leq 1+\sum_{k=1}^\infty\sup_{\theta\in\Theta}\norm{\fiarmapredcoef[\theta]{k}}_{\infty}<\infty\;.
$$
Then, we
get~(\ref{eq:to-show-best-fiarmapred-not-well-specified-error2}) by
applying \Cref{lem:von-neuman}.

We now prove~(\ref{eq:to-show-best-fiarmapred-not-well-specified-error1}).
For all $n\geq1$, the squared norm in the left-hand side's expectation
of~(\ref{eq:to-show-best-fiarmapred-not-well-specified-error1}) can be
written as
$$
\inf_{\theta\in\Theta^\ast_Y}\norm{Y_{n+1}-\sum_{k=1}^n\fiarmapredcoef[\theta]{k}\;Y_{n+1-k}+\sum_{k=1}^n\lr{\fiarmapredcoef[\theta]{k}-\fiarmapredcoef[\hat{\theta}_n]{k}}\;Y_{n+1-k}}_{\cH_0}^2\;.
$$
We thus have, for all $n\geq1$,
\begin{equation}
   \label{eq:best-asymp-erro-with--interm}
\PEarg{\norm{Y_{n+1}-\sum_{k=1}^n\fiarmapredcoef[\hat{\theta}_n]{k}\;Y_{n+1-k}}_{\cH_0}^2}
\leq \PEarg{\lr{A_n+B_n}^2}\;,
\end{equation}
where we set
\begin{align*}
  &A_n:=\inf_{\theta\in\Theta^\ast_Y}\norm{\sum_{k=1}^n\lr{\fiarmapredcoef[\theta]{k}-\fiarmapredcoef[\hat{\theta}_n]{k}}\;Y_{n+1-k}}_{\cH_0}\;,\\
  &B_n:=
\sup_{\theta\in\Theta^\ast_Y}\norm{Y_{n+1}-\sum_{k=1}^n\fiarmapredcoef[\theta]{k}\;Y_{n+1-k}}_{\cH_0}\;.
\end{align*}
We are going to show, successively that
\begin{align}
   \label{eq:best-asymp-erro-with--interm1}
  \lim_{n\to\infty}  \PEarg{A_n^2} = 0\;,\\
   \label{eq:best-asymp-erro-with--interm2}
  \lim_{n\to\infty}  \PEarg{B_n^2} = \mathbb{E}^2\lr{Y,\fiarmacol} \;.
\end{align}
These two facts with~(\ref{eq:best-asymp-erro-with--interm}) indeed
imply~(\ref{eq:to-show-best-fiarmapred-not-well-specified-error1}).
First observe that \Cref{lem:prediction3} straightforwardly yields
\begin{align}
\label{eq:to-show-best-fiarmapred-not-well-specified-error12}
\lim_{n\to\infty}\PEarg{\sup_{\theta\in\Theta}\norm{\sum_{k=n+1}^\infty\fiarmapredcoef[\theta]{k}\;Y_{n+1-k}}_{\cH_0}^2} =0\;.
\end{align}
Thus, to have~(\ref{eq:best-asymp-erro-with--interm1})
and~(\ref{eq:best-asymp-erro-with--interm2}), and by stationarity of
$Y$, we can use
\begin{align*}
  &A'_n:=\inf_{\theta\in\Theta^\ast_Y}\norm{\sum_{k=1}^\infty\lr{\fiarmapredcoef[\theta]{k}-\fiarmapredcoef[\hat{\theta}_n]{k}}\;Y_{-k}}_{\cH_0}\;,\\
  &B':=
\sup_{\theta\in\Theta^\ast_Y}\norm{Y_{0}-\sum_{k=1}^\infty\fiarmapredcoef[\theta]{k}\;Y_{-k}}_{\cH_0}\;,
\end{align*}
and prove instead
\begin{align}
   \label{eq:best-asymp-erro-with--interm1prim}
  \lim_{n\to\infty}  \PEarg{A_n^{\prime 2}} = 0\;,\\
   \label{eq:best-asymp-erro-with--interm2prim}
  \PEarg{B^{\prime 2}} = \mathbb{E}^2\lr{Y,\fiarmacol} \;.
\end{align}
To get Relation~(\ref{eq:best-asymp-erro-with--interm2prim}), we
observe that, with the assumption that the best
$\fiarmacol$-predictor is well defined, Assertion~\ref{item:prop:firama-pred-def6} in
\Cref{prop:firama-pred-def} and~(\ref{eq:fiarmacol-pred}) give that, $\as$,
for all $\theta\in\Theta^\ast_Y$,
$\hat{Y}_0(\theta)=\sum_{k=1}^\infty\fiarmapredcoef[\theta]{k}\;Y_{-k}=\hat{Y}_{0}^\ast$.
Hence, $\as$, $B'=\norm{Y_{0}-\hat{Y}_{0}^\ast}$. We thus have~(\ref{eq:best-asymp-erro-with--interm2prim}).

We conclude with the proof
of~(\ref{eq:best-asymp-erro-with--interm1prim}). Note that
$$
A'_n\leq2\,
\sup_{\theta\in\Theta^\ast_Y}\norm{\sum_{k=1}^\infty\fiarmapredcoef[\theta]{k}\;Y_{-k}}_{\cH_0}
\;.
$$
Note that the $L^2$-norm of this upper bound satisfies
$$
\lr{\PEarg{\sup_{\theta\in\Theta^\ast_Y}\norm{\sum_{k=1}^\infty\fiarmapredcoef[\theta]{k}\;Y_{-k}}_{\cH_0}^2}}^{1/2}\leq
\lr{\sum_{k=1}^\infty\sup_{\theta\in\Theta^\ast_Y}\norm{\fiarmapredcoef[\theta]{k}}_{\infty}}
\; \lr{\PEarg{\norm{Y_{0}}_{\cH_0}^2}}^{1/2}\;,
$$
which is finite by~\Cref{lem:prediction3}. Thus, we can apply the
dominated convergence theorem,
and~(\ref{eq:best-asymp-erro-with--interm1prim}) follows from
$$
\lim_{n\to\infty}A'_n = 0 \quad\as\;,
$$
which we now prove by contradiction. Suppose that, with positive
probability, we can find $\eta>0$ and an increasing sequence $(n_j)_{j\in\nset}$ of
integers such that $A'_{n_j}\geq\eta$ for all
$j$. Then, by~(\ref{eq:consitency-not-well-specified}) and since
$\Theta$ is compact, with positive
probability, there also exists a subsequence  $(n'_j)_{j\in\nset}$ 
integers such that $A'_{n'_j}\geq\eta$ for all $j$ and $\hat{\theta}_{n'_j}$
converges to some $\theta\in\Theta^\ast_Y$ as $j\to\infty$. By \Cref{lem:prediction3},
this latter fact implies that, for this $\theta$,
$$
\lim_{j\to\infty}
\norm{\sum_{k=1}^\infty\lr{\fiarmapredcoef[\theta]{k}-\fiarmapredcoef[\hat{\theta}_{n_j}]{k}}\;Y_{-k}}_{\cH_0}
=0\;.
$$
But since $\theta\in\Theta^\ast_Y$, this contradicts the assumption
that yielded $A'_{n'_j}\geq\eta>0$ for all $j$. This finishes the proof.  
\end{proof}

\appendix
  \section{Technical lemmas}
  We start with two lemmas on complex analysis. 
\begin{lemma}\label{lem:integral-1minusexp-fiarma}
  For all $z \in \cset$ and $\lambda\in[-\pi,\pi] \setminus\{0\}$, we
  have
  \begin{align}\label{eq:lem:integral-1minusexp1-fiarma}
\left(2/\pi\right)^{2\,\Re_+(z)}\;|\lambda|^{2\Re(z)}\; \rme^{-\pi
   \abs{\Im(z)}}  \leq \abs{(1 - \rme^{- \rmi \lambda})^{z}}^2\leq \left(\pi/2\right)^{2\Re_-(z)} \;|\lambda|^{2\Re(z)}\; \rme^{\pi
   \abs{\Im(z)}} \;,
  \end{align}
  where $\Re(z)=(z+\bar z)/2$, $\Re_{+}(z) =\max(\Re(z),0)$ and
  $\Re_{-}(z)= \max(-\Re(z),0)$.
\end{lemma}
\begin{proof}
  Let $z \in \cset$ and
  $\lambda \in \oseg{-\pi, \pi} \setminus\{0\}$. By definition of the
  principal logarithm, we have
  \begin{equation}
    \label{eq:princ-log-normpower}
\text{for all $y\in\cset\setminus\rset_-$,}\quad  \abs{y^z}^2=\abs{\exp\lr{z\ln(y)}}^2 = \abs{y}^{2
    \Re(z)} \rme^{- 2 \Im(z) b(y)}\;,  
  \end{equation}
  where $b(y)$ denotes the argument in the polar form of $y$ in
 $\osego{-\frac{\pi}{2}, \frac{\pi}{2}}$. It follows that
 $
 \rme^{-\pi
   \abs{\Im(z)}}\leq \rme^{- 2 \Im(z) b(y)} \leq \rme^{\pi
   \abs{\Im(z)}}
 $. 
  Applying~(\ref{eq:princ-log-normpower}) with $y=1-\rme^{-\rmi
    \lambda}$, using that
  $\frac{2\abs{\lambda}}{\pi} \leq \abs{2\sin(\lambda/2)}=\abs{1 - \rme^{- \rmi \lambda}} \leq
  \abs{\lambda}$ for all $\lambda\in(-\pi,\pi)$ and
  separating the cases where $\Re(z)\geq0$ and where  $\Re(z)<0$,
  we get~(\ref{eq:lem:integral-1minusexp1-fiarma}).
\end{proof}

\begin{lemma}\label{lem:integral-1minusexp-fiarma-sup-rho}
For all $z \in \cset$ and $\lambda\in[-\pi/3,\pi/3] \setminus\{0\}$,
we have
  \begin{align}\label{eq:lem:integral-1minusexp1-fiarma-sup-rho}
\sup_{0\leq\rho\leq1} \abs{(1 -\rho\, \rme^{- \rmi \lambda})^{z}}^2\leq \left(2\pi/(3\sqrt{3})\right)^{2\Re_-(z)} \;|\lambda|^{-2\Re_-(z)}\; \rme^{\pi
   \abs{\Im(z)}} \;,
  \end{align}
  where $\Re(z)=(z+\bar z)/2$, $\Re_{+}(z) =\max(\Re(z),0)$ and
  $\Re_{-}(z)= \max(-\Re(z),0)$.
\end{lemma}
\begin{proof}
  Applying~(\ref{eq:princ-log-normpower}) with with $y=1-\rho\rme^{-\rmi
    \lambda}$   and using that
  $b(y)\in\osego{-\frac{\pi}{2}, \frac{\pi}{2}}$, we get that
  $$
  \sup_{0\leq\rho\leq1} \abs{(1 -\rho\, \rme^{- \rmi
      \lambda})^{z}}^2\leq
  \sup_{0\leq\rho\leq1} \abs{1-\rho\rme^{-\rmi
    \lambda}}^{2
    \Re(z)} \;\rme^{\pi
   \abs{\Im(z)}} \;.  
  $$
    Let now $z \in \cset$ and $\lambda\in[-\pi/3,\pi/3]\setminus\{0\}$. It is
  straightforward to show that, in this case,
  $$
  \sin^2(\lambda) =  \inf_{0\leq\rho\leq1} \abs{(1 -\rho\, \rme^{- \rmi \lambda})}^2\leq
  \sup_{0\leq\rho\leq1} \abs{(1 - \rho\,\rme^{- \rmi \lambda})}^2= 1 \;.
  $$
  Separating the cases where $\Re(z)\geq0$ and where $\Re(z)<0$, and,
  in the latter case, using that
  $\lrav{\sin(\lambda)}\geq3\sqrt{3}\lrav{\lambda}/(2\pi)$ for
  $\lrav{\lambda}\leq\pi/3$, we easily
  get~(\ref{eq:lem:integral-1minusexp1-fiarma-sup-rho}).
\end{proof}

The next two lemmas involve Banach-space-valued and non-negative series. 

\begin{lemma}\label{lem:powerseries-vector-fiarma}
  Let $E$ be a Banach space and $(a_n)_{n \in \nset} \in E^\nset$ such that
  $\norm{a_n}_{E} \xrightarrow[n \to \infty]{} 0$ and the series
  $\sum \norm{a_n - a_{n+1}}_{E}$ converges. Then for all
  $z_0 \in \overline{\unitdisk} \setminus \{1\}$, the series
  $\sum_{n=0}^{\infty} a_n z_0^n$ converges in $E$ and the mapping
  $z \mapsto \sum_{n=0}^{\infty} a_n z^n$ is uniformly continuous on $[0,z_0]$.
\end{lemma}
\begin{proof}
  By assumption on $(a_n)$, $\sum a_nz^n$ is a power series valued in $E$ with a
  convergence radius at least equal to 1, and hence is uniformly continuous on any compact subset of $\unitdisk$. 
  When $\abs{z_0} = 1$, the result follows using
  Abel's transform. 
\end{proof}

\begin{lemma}
  \label{lem:1}
  Let $\lr{u_k}_{k\in\nset}$ be a non-negative non-increasing sequence
  such that $\sum_{k\in\nset}u_k<\infty$. Then there exists a
  non-decreasing sequence $\lr{v_k}_{k\in\nset}$, going to $\infty$ as $k\to\infty$
  such that $\sum_{k\in\nset}u_kv_k<\infty$. 
\end{lemma}
\begin{proof}
  Let $k_0=0$, and for all $n\geq1$, define by induction
  $$
  k_n=\min\set{j> k_{n-1}}{\sum_{k=j}^\infty u_k\leq 4^{-n}}\;.
  $$
  Then $(k_n)$ is an increasing sequence of integers going to $\infty$ as
  $n\to\infty$. Define, for all $n\geq1$, and for all $k_{n-1}\leq k
  <k_{n}$, $v_k=2^{n}$. Then $(v_k)_{k\in\nset}$ is a non-decreasing
  sequence going to $\infty$ and we have, by definition of $(v_k)$, using that $(u_k)$
  is non-negative  and then, by definition
  of $(k_n)$,
  $$
  \sum_{k\in\nset}u_kv_k=\sum_{n=0}^\infty2^n
  \lr{\sum_{k_{n-1}\leq k<k_{n}}u_k}\leq\sum_{n=0}^\infty2^n
  \lr{\sum_{k=k_{n-1}}^\infty u_k}\leq4\sum_{n=0}^\infty2^{-n}<\infty\;.
  $$
  The proof is concluded.
\end{proof}

  We end this section with the following lemma which relates the ergodicity of a stationary process
  valued in $\cH_0$ with finite second moment to the behavior of its
  spectral measure at the origin.
  \begin{lemma}\label{lem:l2-ergodic-zero-mass}
    Let $\cH_0$-valued be a separable Hilbert space and
    $X:=(X_t)_{t\in\zset}$ be a centered $\cH_0$-valued weakly
    stationary process. Denote by $U^X$ the shift operator defined on
    the modular time domain $\cH^X$ by $U^X: X_t \mapsto X_{t+1}$ and let $\nu_X$
    be the spectral operator measure of $X$.  Then the two following
    assertions are equivalent and they hold if $X$ is an ergodic
    stationary process.
  \begin{enumerate}[label=(\roman*)]
    \item\label{itm:l2-ergodic} For all $Y\in\cH^X$, we have $U^X Y = Y$ if and only if
      $Y=0$. 
     \item\label{itm:zero-mass} We have $\nu_X(\{0\}) = 0$.  
     \end{enumerate}
  \end{lemma}
  \begin{proof}
    By the Kolmogorov isomorphism theorem  (see
  \cite[\Cref{thm:kolmo-isomorphism-thm}]{surveyREFnew}), we can represent any
    $Y\in\cH^X$ as $Y=\int\Phi\;\rmd\hat{X}$ with $\Phi\in\hat\cH^X$, Assertion
    \ref{itm:l2-ergodic} is thus equivalent to saying that for all
    $\Phi\in\widehat{\cH}^X$, we have $\int_\tore 
    \abs{1-\rme^{\rmi\lambda}}^2 \norm{\Phi f_X^{1/2}}_2^2 \,\rmd\norm{\nu_X}_1= 0$ if
    and only if $\int_\tore \norm{\Phi f_X^{1/2}}_2^2
    \rmd\norm{\nu_X}_1= 0$, where $f_X = \frac{\rmd \nu_X}{\rmd
      \norm{\nu_X}_1}$. Since $\norm{f_X}_1 = 1$ $\norm{\nu_X}_1$-a.e., 
    $\nu_X(\{0\}) \neq 0$ is equivalent to have 
    $\norm{\nu_X}_1(\{0\}) > 0$ and we clearly obtain that
    Assertions~\ref{itm:l2-ergodic} and~\ref{itm:zero-mass} are
    equivalent.

    Suppose now that $X$ is an ergodic stationary process and let us
    show that Assertion~\ref{itm:l2-ergodic} holds. The ergodicity of
    $X$ means that
    $(\cH_0^\zset,\borel(\cH_0)^{\otimes\zset},\PP^X,T)$ is an ergodic
    measure preserving dynamical system, where $\PP^X$ is the
    distribution of $X=(X_t)_{t\in\zset}$ defined on the canonical space
    $(\cH_0^\zset,\borel(\cH_0)^{\otimes\zset})$ and $T$ is the shift
    operator on $\cH_0^\zset$ defined by
    $(x_t)_{t\in\zset}\mapsto(x_{t+1})_{t\in\zset}$. Now take
    $Y\in\cH^X$. Setting $\Omega=\cH_0^\zset$ and
    $\cF=\borel(\cH_0)^{\otimes\zset}$, $Y$ can be seen as the
    equivalence class in $L^2(\Omega,\cF,\cH_0,\PP^X)$ of a measurable
    function $h:\Omega\to\cH_0$. Then $h\circ T$ belongs to the
    equivalence class $U^X Y$. To prove Assertion~\ref{itm:l2-ergodic},
    let us suppose that $Y=U^X Y$ (as elements of $\cH^X$) and show
    that $Y=0$ (since the reverse implication is obvious). From what
    precedes, $Y=U^X Y$ implies $h\circ T= h$, $\as[\PP^X]$ Since
    $(\cH_0^\zset,\borel(\cH_0)^{\otimes\zset},\PP^X,T)$ is ergodic,
    this implies that $h$ is constant, $\as[\PP^X]$, which in turn
    implies $Y=0$, since all elements in $\cH^X$ have mean zero. The
    proof is concluded.
     \end{proof}

\section{$L^2(\Vset,\Vsigma,\xi)$-valued weakly stationary time series}\label{sec:functional-time-series-fiarma}
Within this appendix, we set $\cH_0=L^2(\Vset,\Vsigma,\xi)$ for a
$\sigma$-finite measured space $(\Vset,\Vsigma,\xi)$ and we
assume that the Hilbert space $\cH_0$ is separable with dimension
$N\in\{1,2,\dots,\infty\}$. This will allow us to use a  
 Hilbert basis $(\phi_i)_{0\leq i<N}$ of $\cH_0$.

We first show that we can always find a version of an $\cH_0$-valued
random variable which is jointly measurable on $\Vset\times\Omega$.
\begin{proposition}\label{prop:joint-meas-version-fiarma}
Let $(\Vset,\Vsigma,\xi)$ be a $\sigma$-finite measured space. Assume that $\cH_0=L^2(\Vset,\Vsigma,\xi)$ is separable and let $Y$ be an $\cH_0$-valued random variable defined on
$(\Omega,\cF,\PP)$. Then $Y$ admits a version $(v,\omega)\mapsto \tilde Y(v,\omega)$  jointly measurable on
$(\Vset \times \Omega, \Vsigma \otimes \cF)$. 
\end{proposition}
\begin{proof}
  Let us define for all $0\leq n<N$, $\omega\in\Omega$, $v\in\Vset$
  and $\epsilon>0$,
  $S^Y_n(v,\omega):=\sum_{k=0}^n\pscal{Y(\omega)}{\phi_k}\phi_k(v)$
  and
  $N^Y_{\epsilon}(\omega):=\inf\set{n<N}{\norm{S^Y_n(\cdot,\omega)-Y(\omega)}^2_{\cH_0}\leq\epsilon}$. Then
  it is straightforward to show that, for all $\omega\in\Omega$ and $\epsilon>0$,
  $N^Y_{\epsilon}(\omega)$ is well defined in $\nset$ and that
  $(N^Y_{2^{-n}}(\omega))_n$ is a non-decreasing sequence. We now
  defined $\tilde{Y}$ on $\Vset\times\Omega$ by
  $\tilde Y(v,\omega)=\displaystyle\lim_{n\to\infty}
  S^Y_{N^Y_{2^{-n}}(\omega)}(v,\omega)$ if the limit exists in $\cset$
  and $0$ otherwise. It follows that, for all
  $\omega\in\Omega$, $S^Y_{N^Y_{2^{-n}}(\omega)}(\cdot,\omega)$
  converges to $Y$ in $\cH_0$ and
  $S^Y_{N^Y_{2^{-n}}(\omega)}(v,\omega)$ converges to
  $\tilde Y(v,\omega)$ for $\xi$-$\mae$ $v\in\Vset$ as $n\to\infty$
  and that $\tilde Y(\cdot,\omega)=Y(\omega)$ (as elements of
  $\cH_0$). The result follows since $S_n^Y$ is jointly measurable on
  $\Vset\times\Omega$ for all $n\in\nset$ and $N^Y_\epsilon$ is
  measurable on $\Omega$ for all $\epsilon>0$.
\end{proof}
Hence, an $\cH_0$-valued random variable $Y$ can always be assumed to
be represented by a $\Vset\times\Omega\to\cset$-measurable function
$\tilde Y$. If, moreover, $Y\in L^2(\Omega,\cF,\cH_0,\PP)$, then, by
Fubini's theorem, we can see $\tilde Y$ as an element of
$L^2(\Vset\times\Omega,\Vsigma\otimes\cF,\xi\otimes\PP)$, and we can
write
$ \tilde
Y(v,\omega)=\sum_{0\leq k<N}\pscal{Y(\omega)}{\phi_k}\phi_k(v) $,
where the convergence holds in
$L^2(\Vset\times\Omega,\Vsigma\otimes\cF,\xi\otimes\PP)$. As expected,
in this case, the covariance operator $\PCov(Y)$ is an integral
operator with kernel
$(v,v')\mapsto\cov{\tilde Y(v,\cdot)}{\tilde Y(v',\cdot)}$. It is
tempting to write that $\var{\tilde Y(v,\cdot)}$ is equal to the
kernel of $\PCov(Y)$ on the diagonal
$\set{v=v'}{(v,v')\in\Vset^2}$. However, because this diagonal set has
null $\xi^{\otimes2}$-measure, this ``equality'' is meaningless. In
the following lemma we make this statement rigorous by relying on a
decomposition of the form $\PCov(Y)=K K^\adjoint$ for some
$K \in\cS_2(\cH_0)$. In particular, this can be used to give a
rigorous definition of $\sigma_W$ in
\Cref{cor:cns-arma-fi-operator-functional-fiarma} or
\eqref{eq:ass-LRD-fiarma}.
\begin{lemma}\label{lem:kernel-cov-fiarma}
Let $(\Vset,\Vsigma,\xi)$ be a $\sigma$-finite measured space. Assume that $\cH_0=L^2(\Vset,\Vsigma,\xi)$ is separable and let $Y$ be an $\cH_0$-valued random variable defined on
$(\Omega,\cF,\PP)$.  Let $K \in\cS_2(\cH_0)$ and denote by $\kernelope{K}$ its
  kernel in $L^2(\Vset^2,\Vsigma^{\otimes2},\xi^{\otimes2})$. Suppose
  that $\PCov(Y)=K K^\adjoint$. Then, we
  have,  for $\xi$-$\mae\;v\in\Vset$,
  \begin{equation}
    \label{eq:kernel-cov-fiarma}
  \PEarg{\abs{\tilde Y(v,\cdot)}^2}=\int\abs{\kernelope{K}(v,v')}^2\;\xi(\rmd
  v')=\norm{\kernelope{K}(v,\cdot)}_{\cH_0}^2\;,    
  \end{equation}
  where $\tilde Y$ is a version of $Y$ in
  $L^2(\Vset\times\Omega,\Vsigma\otimes\cF,\xi\otimes\PP)$. 
\end{lemma}
\begin{proof}
  As explained before the lemma, we have that
  $\tilde Y_n :(v,\omega) \mapsto \sum_{0\leq
    k<n}\pscal{Y(\omega)}{\phi_k}_{\cH_0}\,\phi_k(v)$ converges to
  $\tilde Y$ as $n\to N$ in
  $L^2(\Vset\times\Omega,\Vsigma\otimes\cF,\xi\otimes\PP)$.  Let us
  define, for all $v,v'\in\Vset$ and $0\leq n\leq N$,
  $\kernelope{K}_n(v,v')=\sum_{0\leq k<n}
  \pscal{\kernelope{K}(\cdot,v')}{\phi_k}_{\cH_0}\phi_k(v)$. Then,
  using that
  $\kernelope{K}\in L^2(\Vset^2,\Vsigma^{\otimes2},\xi^{\otimes2})$,
  it is easy to show that $\kernelope{K}_n$ converges to
  $\kernelope{K}$ in $L^2(\Vset^2,\Vsigma^{\otimes2},\xi^{\otimes2})$
  as $n \to N$.  By the Cauchy-Schwartz inequality, the mappings
  $(g,h)\mapsto [v\mapsto\PEarg{g(v,\cdot)\,\overline{h(v,\cdot)}}]$
  and $(g,h)\mapsto [v\mapsto\pscal{g(v,\cdot)}{h(v,\cdot)}_{\cH_0}]$
  sesquilinear continuous from
  $L^2(\Vset\times\Omega,\Vsigma\otimes\cF,\xi\otimes\PP)$ to
  $L^1(\Vset,\Vsigma,\xi)$ and from
  $L^2(\Vset^2,\Vsigma^{\otimes2},\xi^{\otimes2})$ to
  $L^1(\Vset,\Vsigma,\xi)$, respectively. This, with the two previous
  convergence results, shows that
  $v\mapsto\PEarg{\abs{\tilde Y_n(v,\cdot)}^2}$ and
  $v\mapsto\norm{\kernelope{K}_n(v,\cdot)}_{\cH_0}^2$ both converge in
  $L^1(\Vset,\Vsigma,\xi)$, to $\PEarg{\abs{\tilde Y(v,\cdot)}^2}$ and
  $\norm{\kernelope{K}(v,\cdot)}_{\cH_0}^2$, respectively, that is to
  the left-hand side and right-hand side
  of~(\ref{eq:kernel-cov-fiarma}).  Hence, to conclude, we only have
  to show that, for all $v\in\Vset$ and $0\leq n<N$, $
  \PEarg{\abs{\tilde
      Y_n(v,\cdot)}^2}=\norm{\kernelope{K}_n(v,\cdot)}_{\cH_0}^2$. To
  this end, we write $ \PEarg{\abs{\tilde Y_n(v,\cdot)}^2}=
  \PEarg{\sum_{0\leq
      j,k<n}\pscal{Y}{\phi_j}_{\cH_0}\pscal{\phi_k}{Y}_{\cH_0}\phi_j(v)\overline{\phi_k(v)}}
  =\sum_{0\leq
    j,k<n}\phi_j^\adjoint\PCov(Y)\phi_k\,\phi_j(v)\overline{\phi_k(v)}$. Using
    the fact that $\PCov(Y)=K
    K^\adjoint$ and Fubini's theorem, we get $\phi_j^\adjoint\PCov(Y)\phi_k=\int
    \pscal{\kernelope{K}(\cdot,v'')}{\phi_j}_{\cH_0}\overline{\pscal{\kernelope{K}(\cdot,v'')}{\phi_k}_{\cH_0}}\,\xi(\rmd
    v'')$.  Inserting this in the previous equation and moving the double sum
    inside the integral with respect to $\xi(\rmd
    v'')$, this double sum becomes a product of two conjugate
    sums. Namely, we get that
    $$
\PEarg{\abs{\tilde Y_n(v,\cdot)}^2}=\int\abs{\sum_{0\leq k<n}
      \pscal{
        \kernelope{K}(\cdot,v'')}{\phi_k}_{\cH_0}\;\phi_k(v)}^2\;\xi(\rmd
    v'') =
    \norm{\kernelope{K}_n(v,\cdot)}_{\cH_0}^2\;,
$$ which concludes the
    proof.
  \end{proof}

\end{document}